\documentclass[a4paper, 11pt, headings = small, abstract,numbers=endperiod]{scrartcl}

\usepackage{amscd,amsmath,amsthm,array,bbm,hhline,mathrsfs, enumerate,dsfont,changes, booktabs,fancybox,calc,textcomp,xcolor,graphicx,bbm,xspace,nicefrac,stmaryrd,url,arcs,listings,nameref}

\usepackage[title]{appendix}
\usepackage[semibold]{libertine}
\usepackage[upint]{libertinust1math}
\usepackage[scr=boondoxupr, scrscaled = 1.0]{mathalpha}

\usepackage[colorlinks=true, allcolors=myteal]{hyperref}
\definecolor{WIMgreen}{RGB}{60 134 132}
\definecolor{red_pers}{RGB}{204 37 41}
\definecolor{UMblue}{RGB}{4 47 86}
\definecolor{myteal}{RGB}{0 123 137}
\definecolor{dartmouthgreen}{rgb}{0.05, 0.5, 0.06}\definecolor{cobalt}{rgb}{0.0, 0.28, 0.67}\definecolor{coolblack}{rgb}{0.0, 0.18, 0.39}
\definecolor{glaucous}{rgb}{0.38, 0.51, 0.71}\definecolor{hooker\'sgreen}{rgb}{0.0, 0.44, 0.0}\definecolor{lemonchiffon}{rgb}{1.0, 0.98, 0.8}\definecolor{oucrimsonred}{rgb}{0.6, 0.0, 0.0}\definecolor{radicalred}{rgb}{1.0, 0.21, 0.37}\definecolor{raspberry}{rgb}{0.89, 0.04, 0.36}\definecolor{royalazure}{rgb}{0.0, 0.22, 0.66}
\definecolor{dex}{RGB}{138 18 34}
\definecolor{darkgreen}{RGB}{0 69 0}
\definecolor{darkblue}{RGB}{0 0 99}

\usepackage[utf8]{inputenc}
\usepackage[ngerman,english]{babel}								
\usepackage[numbers, square]{natbib}

\usepackage{mathtools}
\usepackage{graphicx}

\newcommand{\R}{\mathbb{R}}

\newcommand{\N}{\mathbb{N}}

\newcommand{\Ex}{\mathbb{E}}
\newcommand{\PR}{\mathbb{P}}

\newcommand{\ind}[1]{\mathds{ 1 }_{\{{#1}\}}}
\newcommand{\inda}[1]{\mathds{ 1 }_{{#1}}}
\newcommand*{\f}{\mathcal{F}}
\newcommand{\eps}{\varepsilon}
\DeclareMathOperator{\Exp}{Exp}

\DeclareMathOperator{\ccc}{C}

\DeclareMathOperator{\dist}{dist}
\DeclareMathOperator{\rm}{RM}
\DeclareMathOperator{\supp}{supp}
\DeclareMathOperator{\cpt}{cpt}
\DeclareMathOperator{\ran}{\mathcal{R}}

\makeatletter
\newcommand*\bigcdot{\mathpalette\bigcdot@{.6}}
\newcommand*\bigcdot@[2]{\mathbin{\vcenter{\hbox{\scalebox{#2}{$\,\m@th#1\bullet\,$}}}}}
\makeatother

\def\XXint#1#2#3{{\setbox0=\hbox{$#1{#2#3}{\int}$ }
\vcenter{\hbox{$#2#3$ }}\kern-.6\wd0}}

\theoremstyle{definition}
\newtheorem*{assumption}{Assumptions}
\newtheorem*{remark*}{Remark}
\newtheorem{theorem}{Theorem}[section]
\newtheorem{proposition}[theorem]{Proposition}
\newtheorem{remark}[theorem]{Remark}
\newtheorem{lemma}[theorem]{Lemma}
\newtheorem{definition}[theorem]{Definition}
\newtheorem{corollary}[theorem]{Corollary}

\title{\fontsize{16}{19} \selectfont 		On the existence of Markovian randomized equilibria in Dynkin games of war-of-attrition-type
}
\author{Sören Christensen\thanks{Kiel University, Mathematical Department\\
email: \href{mailto:christensen@math.uni-kiel.de}{christensen@math.uni-kiel.de}}, Boy Schultz\thanks{Kiel University, Mathematical Department\\email: \href{mailto:schultz@math.uni-kiel.de}{schultz@math.uni-kiel.de}}}
\date{\today}

\begin{document}

\maketitle
\begin{abstract}
In optimal stopping problems, a Markov structure guarantees Markovian optimal stopping times (first exit times). Surprisingly, there is no analogous result for Markovian stopping games once randomization is required. This paper addresses this gap by proving the existence of Markov-perfect equilibria in a specific type of stopping game -- a general nonzero-sum Dynkin games of the war-of-attrition type with underlying linear diffusions. 
Our main mathematical contribution lies in the development of appropriate topologies for Markovian randomized stopping times. This allows us to establish the existence of equilibria within a tractable and interpretable class of stopping times, paving the way for further analysis of Markovian stopping games.
\end{abstract}
\noindent
\small{\textit{\href{https://mathscinet.ams.org/mathscinet/msc/msc2020.html}{2020 MSC:}} Primary 91A55; secondary 60G40, 91A15\\
\textit{key words:} Dynkin games, diffusion, war of attrition, Markovian randomized stopping times, Markov-perfect equilibria.}
\normalsize



\section{Introduction}
The aim of this paper is to develop a method to prove the existence of Markovian randomized equilibria for stopping games with underlying linear diffusions. 

\subsection{Stopping problems and stopping games}
To clarify the relevance and the basic idea, we first give a brief overview of the stopping problem, i.e. the game with one player, since in this case the theoretical results are already well established, and then compare the situation with that of stopping games. 
In general optimal stopping problems, one asks for the maximisation of 
$
V = \sup_{\tau} \mathbb{E} \left[ G_\tau\right]
$ over all stopping times $\tau$. For a concrete solution, however, more structure is needed to make the problem accessible. This structure is typically given by a driving Markov process $X$, in our case a linear diffusion, and a value function of the form
\begin{align*}
	V(x)= \sup_\tau \mathbb{E}_{x}\left[ e^{-r \tau}F(X_\tau)\right], \enskip X_0=x \in E.
\end{align*}
In this setting, it is well known that under certain assumptions the first exit time $\tau^\ast = \inf \{ t \ge 0 : X_t \notin D\}$ from the {continuation region} $D=\{x\in E:V(x)>F(x)\}$ is optimal. It is therefore sufficient to restrict the optimization from the unmanageably large class of all stopping times to the class of first exit times, which means in particular that the stopping decision only depends on the current state of the process. This makes the problem time-consistent (Bellman principle).
Moreover, it opens the door to the solution of the problem, e.g. via an associated free boundary problem or in the form of concave envelopes \cite{peskir2006optimal,Dayanik,MR4414697}. 

However, many important real-world situations involve two or more decision makers with possibly conflicting goals, leading to stopping games. Stopping games arise in many applications, such as preemption games \cite{fudenberg1985preemption}, game options \cite{KiferGameOptions}, wars of attrition \cite{decamps2022mixed}, and market entry/exit decisions \cite{fine1989exit,fudenberg1985preemption, ghemawat1985exit,weeds2002strategic}.

Starting with \cite{dynkin1969game}, \emph{Dynkin games} have become a fundamental mathematical model for stopping games. In these, two players $i=1,2$ choose  stopping times $\tau_i$ and are interested in maximizing functionals of the form
\begin{equation*}
J_i(\tau_i, \tau_j) = \mathbb{E} \left[ F^i_{\tau_i} \mathbb{I}_{\{\tau_i < \tau_j\}} + G^i_{\tau_j} \mathbb{I}_{\{\tau_j < \tau_i\}} + H^i_{\tau_i} \mathbb{I}_{\{\tau_i = \tau_j\}} \right], \quad i \in \{1,2\}, j = 3-i.
\end{equation*} The question of the existence and representation of equilibria is then investigated. 
There are basically two types of results. 
In Markovian games, a number of papers study the existence and characterization of equilibria in the first exit times, reflecting the Markovian structure of the problem. However, this can only be expected under rather restrictive ordering conditions on the payoffs, see e.g. \cite{attard2017nash,attard2018nonzero,Bensoussan_Friedman_1974,cattiaux1990markov,deanglis2018nash,ekstrom2008optimal,ekstrom2006value,Friedman_1973}.

For other classes of stopping games, i.e. if the mentioned ordering conditions are not fulfilled, there are no equilibria in first exit times anymore. Instead, it is now necessary to consider randomized stopping times \cite{shmaya2014equivalence} to obtain  the existence of ($\epsilon$-)equilibria \cite{Shmaya_Deterministic_Epsilon_Existence,laraki2013equilibrium,laraki2005value,touzi2002continuous}. 
 While these results at least essentially prove the existence of equilibria, a central challenge remains: in general, the randomized stopping times have a strongly path dependent nature and form an almost unmanageably large class. Even with an underlying Markov structure, it remains unclear whether the equilibrium stopping decision is Markovian in the sense that it depends only on the current state of the underlying process -- as one would hope from the case of optimal stopping problems. Therefore, it is necessary to search for a more specific class of randomized stopping times that allows a clearer characterization and game-theoretic interpretation of the equilibria.

Although Dynkin games are probably the most fundamental class of stopping games, many other classes have been considered. On the one hand, these are variants of Dynkin games, e.g. games with more than two players \cite{HamadeneConti,laraki2005timing_games}, games with {partial information}. \cite{deangelis2020ghost, deangelis2022dynkin}, but also contests \cite{feng2015gambling,seel2013gambling} and equilibrium formulations of time-inconsistent problems \cite{bayraktar2024optimal, bayraktar2023stability, christensen2018finding,huang2021optimal}. We will not discuss these problems in detail here, but our goal in this paper is to develop a technique that has the potential to be applicable to these situations as well.

\subsection{Markovian randomized stopping times}
From the previous discussion, it seems natural to consider Markovian stopping times as candidates for equilibria in stopping games. In contrast to stopping problems, however, these must be randomized times in order to have the potential for proving general existence theorems in the Markovian case. 

Examples of such Markovian randomized stopping times have already been considered in certain stopping games. 
First, there are stopping time of the form
\begin{align}\label{gwebshery}
	\inf\{t\geq 0: X_t \notin D\mbox{ or }A_t\geq E\},\;\;A_t:=\int_0^t \psi(X_t)dt ,
\end{align}
where $D$ is a set in the state space of $X$,
$E\sim \mbox{Exp}(1)$ independent of all other random sources and $\psi$ some non negative function. For Dynkin games such equilibria have been considered in \cite{steg2015symmetric} and for time-inconsistent problems in \cite{christensen2020time,christensen2023time}.
In the time-inconsistent problem discussed in \cite{tan2021failure}, it turns out that stopping times of the form \eqref{gwebshery} are not sufficient to obtain equilibria and in \cite{bodnariu2022local} the considered class was extended to 
\begin{align}
    \label{asdkasdi1}
	A_t:=\sum_{i=1}^{n}\psi_i l_t^{x_i}(X)+\int_0^t \psi(X_t)dt,
\end{align}
where $\left(l_t^{x_i}(X)\right)_{t\geq 0}$ is the \emph{local time} of $X$ at $x_i$, 
$n$ is a fixed number, 
$x_i\in\{x_1,\dots,x_n\} \subset D$, 
$\psi_i\in (0,\infty)$ are fixed constants. So, in the initial states $x_i$ the process can now be stopped with probability in $(0,1)$.  Such stopping times with $\psi=0$ were also considered in \cite{ekstrom2017dynkin} in the context of zero-sum Dynkin games with heterogeneous beliefs.

\cite{decamps2022mixed} investigates a variant of the stopping game considered later in this article, a war-of-attrition game for an underlying linear diffusion. First, the general class of Markovian randomized stopping times for linear diffusions is introduced and studied, and it is discussed why the term ``Markovian'' is justified. It is also emphasized that the general class of such stopping times can be described by sets and Radon measures. We will return to this in the following section. Based on this, Markov-perfect equilibria are characterized via a variational system and are found to be given by special additive functionals of the form \eqref{asdkasdi1}.


However, theorems on the existence of equilibria in Markovian randomized stopping times are not given in any of the previously mentioned articles. Until the very recent independent preprint by Decamps, Gensbittel and Mariotti \cite{decamps2024existence} the only result in this direction known to the authors was for the discrete-time and discrete-space case in \cite{christensen2023markovian}, see also \cite{ChristensenLindensjoe_meanvariance} for related results in discrete time-inconsistent problems. In fact, shortly after the publication of the first preprint version of the present paper \cite{christensen2024existencemarkovianrandomizedequilibria}, which only contained the results up to the end of Section \ref{sec:final_proof}, \cite{decamps2024existence} gave an alternative direct proof of Theorem \ref{natbdryexists}. In their work, Decamps, Gensbittel and Mariotti apply the fixed point theorem of Eilenberg and Montgomery \cite{eilenberg1946fixed}, based on a topology that turns out to be very similar to the ones introduced here. 

\subsection{Main contributions and structure of the paper}

The previous discussion shows that one of the central open questions in the treatment of Markovian stopping games is the existence of Markovian equilibria. 
The aim of this paper is to take a step in this direction and to give such a result for an important class of Dynkin games for underlying linear diffusions. More specifically, we consider war-of-attrition-type stopping games, in which it is well known that randomization is necessary. However, we have treated the problem in such a way that our approach can also serve as a blueprint for other stopping games. 

In Section \ref{T1} we will introduce the setting and formulate the main results, Theorem \ref{absbdryexists} and Theorem \ref{natbdryexists}. The former deals with the existence of Markov-perfect equilibria for absorbing boundaries and the latter for natural boundaries. Unsurprisingly, our approach to existence is an application of a Kakutani-type fixed point theorem to a best-response mapping. The key question is what topology to use on the space of Markovian randomized stopping times $\mathcal Z$. 
We address this question in Section \ref{sectopo}, where we introduce two topologies. In particular, we will show that $\mathcal Z$ has suitable compactness properties, see Theorem \ref{comp2}. This result is also of interest on its own.
As a first step towards the proof our first main existence Theorem \ref{absbdryexists}, in Section \ref{secdisc} we then consider an auxiliary problem in which the players are only allowed to stop in a finite subset of the state space. In this case, the space underlying the fixed point problem is finite-dimensional, so that the topological questions are easier to handle. 
We then use the compactness of $\mathcal Z$ to select a convergent subsequence from the equilibria of the auxiliary problems. Its limit in $\mathcal Z$ is then the candidate for an equilibrium, which we will verify in Section \ref{sec:final_proof}.
In the subsequent Section \ref{sec:naturalboundary} we extend our first main theorem with an exhaustion argument to prove the other main result, Theorem \ref{natbdryexists}.
One advantage of our approach is that the approximation and exhaustion scheme used in the proofs is flexible enough to be adaptable to different types of boundary behaviors of the underlying diffusion, probably even beyond those discussed in the present work. Second, the approach shows that general equilibria evolve from equilibria in simpler, possibly even numerically tractable, games as a limit in distribution.

\section{Setup and main results} \label{T1}
In this paper we consider a regular linear Itô diffusion $X = (X_t)_{t \in [0,\infty)}$ taking values in an interval $I$ with $I \subset \R$ and defined on a filtered probability space $(\Omega, \f, (\f_t)_{t \in [0,\infty)},\PR)$ satisfying the usual hypotheses. Generally we assume that  the behavior of $X$ in the interior $I^\circ$ of $I$ is governed by the stochastic differential equation
\begin{align}
    d X_t = \mu(X_t) dt + \sigma(X_t) dW_t , \quad X_0=x\in I^\circ\label{eq:dynamics}
\end{align} with an $(\f_t)_{t \in [0,\infty)}$-adapted, real valued, standard Brownian motion $W = (W_t)_{t \in [0,\infty)}$ and Lip\-schitz continuous coefficients $\mu: I \to \R$, $\sigma: I \to (0,\infty)$. A jointly continuous version of the local time process of $X$ at $y\in I$ will be denoted by $(l^y_t)_{t \in [0,\infty)}$. The existence of such a version follows e.g. from \cite[Theorem (44.2)]{rogers2000diffusions}. We set $\f^X_\infty := \sigma(X_t:t\ge0)$ and denote the canonical shift operator associated to $X$ by $\theta$. Let $E_1,E_2 \sim \Exp(1)$ be two fixed random variables on $(\Omega, \f,\PR)$ such that $E_1,E_2$ and $X$ are independent. If $Y$ is any random variable on $(\Omega, \f, (\f_t)_{t \in [0,\infty)},\PR)$ with values in some metric space, we denote its distribution by $\PR^Y$. As usual, we write $\PR_x$ for the conditional distribution of $\PR$ given $X_0=x$ and $\Ex_x$ for the corresponding expectation operator.

For locally compact, second countable metric spaces $\mathcal{X}$ we use the following notations for the spaces of measures, probability measures, and Radon measures on $\mathcal{X}$, respectively:
\begin{align*}
    \mathcal{M}(\mathcal{X})&:=\{\mu:\mu\ge0 \text{ measure on }\mathcal{X}\},\\
    \mathcal{M}^1(\mathcal{X})&:=\{\mu\in \mathcal{M}(\mathcal{X}): \mu(\mathcal{X}) =1\},\\
    \rm(\mathcal{X})&:= \{\mu \in \mathcal{M}(\mathcal{X}):\mu \text{ is locally finite} \}
\end{align*} 
Moreover, we regard $\rm(\varnothing)=\{0\}$, where $0$ denoted the zero measure.
We denote the first exit time from measurable $D\subset I$ by
\begin{align*}
    \tau^D &:= \inf \{ t \ge 0 : X_t \not\in D\}
\end{align*} and the first entry time into $y\in I$ or measurable $B\subset I$ by
\begin{align*}
    \eta_y:= \inf \{ t \ge 0 : X_t =y\}, \quad \eta^B:=\{t \ge 0: X_t\in B\},
\end{align*}respectively. Further, we denote the set of all  $(\f_t)_{t \in [0,\infty)}$-stopping times by $\mathcal{T}$.

We now come to the introduction of the general class of Markovian randomized stopping times, extending the notation of \eqref{gwebshery}. For this purpose, we consider stopping times that stop at the first exit from $D$ and in $D$ with a stopping rate described by the measure $\lambda$. This is specified as follows.
Let $D \subset I$ open (in $I$) and $\lambda \in \rm(D)$. We set the additive functional $A^{D,\lambda} = (A^{D,\lambda}_t)_{t \in [0,\infty)}$ generated by $X,D$ and $\lambda$ to be given by
\begin{align}
    A^{D,\lambda}_t(\omega):= \int_D l^y_t(\omega) \lambda(dy) + \infty \ind{\tau^D(\omega) \le t},\quad t \ge 0.\label{func}
\end{align} 

Other general notations used in the following are $\overline{D}$ for the closure of a set $D\subset \R$, $D^\circ$ for its interior, $\partial D$ for its boundary and $D^c$ for its complement $I\setminus D$. Given $\mathcal{X}, \mathcal{Y}$ and $\mathcal{X}'\subset\mathcal{X}$ and a function $f\colon \mathcal{X}\to \mathcal{Y}$ we denote the restriction of $f$ to $\mathcal{X}'$ by $f\vert_{\mathcal{X}'}$.

\begin{definition}(Markovian randomized stopping time)\\
Let $D \subset I$ open (in $I$), $\lambda \in \rm(D)$, and $E\sim \Exp(1)$ a random variable on $(\Omega, \f,\PR)$ that is independent of $X$.  We define the \textit{Markovian randomized (stopping) time} generated by $D$, $\lambda$ and $E$ as
\begin{align*}
    \tau^{D,\lambda} := \inf\{ t \ge 0: A_t^{D,\lambda} \ge E\}.
\end{align*} The space of all Markovian randomized times based on the random variable $E$ is denoted by
\begin{align*}
    \mathcal{Z}:= \mathcal{Z}(E):=\{ \tau^{D,\lambda}: D\subset I \text{ open (in } I),  \lambda\in \rm(D)\}
\end{align*} 
and is a subset of the set of all randomized stopping times $\ran:= \ran(E)$.
For $\tau = \tau^{D,\lambda} \in \mathcal{Z}$, we set $A^{\tau}:=A^{D,\lambda}$. For open or closed $U\subset I$ we set
\begin{align*}
    &\mathcal{Z}^U:=\mathcal{Z}^U(E):=\{ \tau^{D,\lambda}\in\mathcal{Z}(E) : D \subset U\},\\
    &\mathcal{Z}_U:=\mathcal{Z}_U(E):=\{ \tau^{D,\lambda}\in\mathcal{Z}(E) : I \setminus D \subset U, \,\supp(\lambda)\subset U \}.
\end{align*} 
In the following we specify the generating random variable $E$ only if necessary.
\end{definition}

One way to make the idea that $\tau\in \mathcal Z$ is Markovian precise is that with a natural extension of the shift operator $\theta$ it holds that
\begin{align*}
    \ind{\tau \ge \sigma} \tau = \ind{\tau\ge\sigma} \theta_\sigma \circ \tau + \sigma
\end{align*} in distribution for all $\f^X_\infty$-measurable $\sigma$, see \cite{christensen2023time}. 
More detailed discussions can be found in \cite[Section 3]{decamps2022mixed}, see also \cite[I.II.4]{Borodin} for the connection to killing times. 
In the following, player 1 will use the random device $E=E_1$ introduced above for randomization and player 2 will use $E_2$.

We consider a Dynkin game, i.e.\, a two player game of stopping based on the reward functionals 
\begin{align}
    &J^1:I \times \ran(E_1) \times \ran(E_2) \to \R, \notag\\
    &J^1(x,\tau_1, \tau_2):= \Ex_x[e^{-r_1 \tau_1} g_1(X_{\tau_1}) \ind{\tau_1 \le \tau_2}+e^{-r_1 \tau_2} f_1(X_{\tau_2}) \ind{\tau_1 > \tau_2}],\notag\\
    &V^1(x,\tau_2) := \sup_{\tau_1'\in \ran(E_1)} J^1(x,\tau_1', \tau_2)\label{stoppr}
\end{align}
\begin{align*}
    &J^2:I \times \ran(E_1) \times \ran(E_2) \to \R, \\
    &J^2(x, \tau_1,\tau_2):= \Ex_x[e^{-r_2 \tau_2}g_2(X_{\tau_2}) \ind{\tau_2 \le \tau_1}+e^{-r_2 \tau_1}f_2(X_{\tau_1}) \ind{\tau_2 > \tau_1}],\\
    &V^2(x,\tau_1) := \sup_{\tau_2'\in \ran(E_2)} J^1(x,\tau_1, \tau_2')
\end{align*} where $g_1,g_2,f_1,f_2:I \to \R$ are continuous and $r_1,r_2 \ge 0$. Here, player $i\in\{1,2\}$ tries to maximize the function $J^i$ by choosing $\tau_i \in \ran(E_i)$, while the other arguments $x\in I$ and $\tau^j\in \ran(E_j)$, $j\neq i$ are perceived as fixed. If necessary, we use the convention $g_i(X_\infty):=\lim_{t\to\infty}g_i(X_t)$, $i=1,2$.


\begin{definition}(equilibrium)\label{defequi} \\
    Let $x \in I$. We call a pair $(\tau_1,\tau_2) \in \ran(E_1)\times \ran(E_2)$ \textit{equilibrium at $x$} if
    \begin{align}
        &J^1(x,\tau_1, \tau_2) \ge J^1(x,\tau_1', \tau_2) \quad \text{and} \label{cond1}\\
        &J^2(x, \tau_1,\tau_2) \ge J^2(x, \tau_1,\tau_2') \label{cond2}
    \end{align} for all $\tau_1'\in\ran(E_1)$ and all $\tau_2' \in \ran(E_2)$. If $(\tau_1,\tau_2)$ is an equilibrium at $x$ for all $x\in I$, we call $(\tau_1,\tau_2)$ \textit{equilibrium}. An equilibrium $(\tau_1,\tau_2)\in \mathcal{Z}(E_1)\times \mathcal{Z}(E_2)$ is called \textit{Markov perfect equilibrium}.
\end{definition}

Even though we will look for Markov-perfect equilibria in the much smaller and more manageable class $\mathcal Z\subseteq \ran$, it should be noted that these must satisfy the equilibrium conditions \eqref{cond1} and \eqref{cond2} for general challengers $\tau'\in \ran$. 
\begin{assumption}\;
        \begin{enumerate}
        \renewcommand{\labelenumi}{(\theenumi)}
        \renewcommand{\theenumi}{A}
        \item \label{A} $g_1\le f_1$ and $g_2\le f_2$.
        \end{enumerate}
        
        \begin{enumerate}
        \renewcommand{\labelenumi}{(\theenumi)}
        \renewcommand{\theenumi}{B\arabic{enumi}}
            \item \label{B1} $X$ lives on a compact interval $I\subset \R$ with absorbing boundary points $y\in\partial I$.
            \item \label{B2} $f_1(y)=g_1(y)$ and $f_2(y)=g_2(y)$ for all $y\in\partial I$.
        \end{enumerate}
        \begin{enumerate}
        \renewcommand{\labelenumi}{(\theenumi)}
        \renewcommand{\theenumi}{C\arabic{enumi}}
        \item \label{C1} $X$ lives on an open interval $I\subset \R$ with natural boundary $\partial I$.
        \item \label{C2}$\max\{\Ex_x[\sup_{t\in [0,\infty)}e^{-r_it} \vert g_i(X_t)\vert ],\Ex_x[\sup_{t\in [0,\infty)} e^{-r_it}\vert f_i(X_t)\vert ]\}<\infty$ for $i=1,2$ and all $x\in I$.
        \item \label{C3} $\lim_{t\to \infty} e^{-r_it} g_i(X_t)=\lim_{t\to \infty} e^{-r_it} f_i(X_t)=0$ $\PR_x$-a.s.\ for $i=1,2$ and all $x\in I$.
        \end{enumerate}

\end{assumption}


\begin{remark}
Assumptions \eqref{A} corresponds to the general situation of a war of attrition, where both players prefer to stop second, so that a randomization strategy seems appropriate.
\eqref{B1} and \eqref{C1} specify the boundary behavior of $X$, while \eqref{B2} and \eqref{C2}, \eqref{C3} are technical assumptions associated to the corresponding boundary condition.
In particular in the case of an absorbing boundary of $X$ it is easily seen that the equality  $f_i(y)=g_i(y)$ for $y\in\partial I$, for one player $i$ is actually necessary for the existence of equilibria.  
    
\end{remark}


Based on the assumptions stated above our main theorems, with their proofs given in Sections \ref{sec:final_proof} and \ref{sec:naturalboundary} respectively, are the following:

\begin{theorem}\label{absbdryexists}
    Under the assumptions \eqref{A}, \eqref{B1} and \eqref{B2} there exists a Markov perfect equilibrium for \eqref{stoppr}.
\end{theorem}

\begin{theorem}\label{natbdryexists}
    Under the assumptions \eqref{A}, \eqref{C1}, \eqref{C2} and \eqref{C3} there exists a Markov perfect equilibrium for \eqref{stoppr}.
\end{theorem}

\begin{remark}
    Given Assumption \eqref{A} and \eqref{C2}, the technique from Section \ref{sec:naturalboundary} can be adapted to treat half-open interval state spaces $I$ with one absorbing and one natural boundary point, provided that \eqref{B2} holds at the absorbing boundary and \eqref{C3} holds whenever the absorbing boundary is not hit in finite time. However, in order not to inflate the already somewhat extensive notation, we have decided to limit ourselves to the two previous statements in this paper.
\end{remark}


\section{Topologization of Markovian randomized times} \label{sectopo}
Our approach to proving the existence theorem is to apply a suitable fixed point theorem of Kakutani type (see, e.g., \cite[Lemma 20.1]{osborne_rubinstein}) to the best response mapping. 
The prerequisites for the application of such fixed point theorems are always variants of the following:
\begin{itemize}
    \item compactness of the underlying space
    \item closed graph of the best response mapping
\item convexity of the best response sets
\end{itemize}
It can be seen that the choice of topology is essential and that there is no canonical topology for the space of Markovian randomized stopping times. 
We will actually consider two different topologies for our line of argument. One is based on the description of $\mathcal Z$ by a representing measure, the other on the distribution of the stopped process.
We would like to emphasize that the following results are not based on the specific stopping game we look at, but are of general interest. In particular the assumptions \eqref{A}, \eqref{B1}, \eqref{B2}, \eqref{C1}, \eqref{C2}, \eqref{C3} are not needed in this entire section.

\subsection{$\mathcal{M}$-topology}
First, we consider a topology that we will use later to prove fixed point theorems in the discretized game. It is based on the fact that discrete measures can be described by their counting densities. More precisely, in the following definition it is considered which local-time stopping rates $\lambda^D(\{x\})\in[0,\infty]$ a stopping time $\tau^{D,\lambda}$ uses on a discrete set $U$.
Now, the $\mathcal M$-topology is built so that $\tau,\tau_1,\tau_2,...\in\mathcal{Z}_U$ converge to $\tau$ if and only if the stopping rates converge for all $x\in U$. More precisely:

\begin{definition}($\mathcal{M}$-topology)\\
Let $U\subset I$ be closed in $I$. Let $\iota_1: \mathcal{Z}_U \to \mathcal{M}(U)$ be given by
\begin{align*}
    \tau^{D,\lambda} \mapsto \lambda^D, \quad \lambda^D(A):= \begin{cases}
        \lambda (A), & \text{if }A \subset D, \\
        \infty, &\text{if } A\setminus D \neq\varnothing,
    \end{cases} \quad A \in \mathcal{B}(U).
\end{align*} Moreover, for finite or countable $U$ let $\iota_2: \mathcal{M}(U) \to [0,1]^{U} $ be given by
\begin{align*}
    \mu \mapsto \left(\frac{\mu(\{x\})}{1+\mu(\{x\})}\right)_{x\in U}, 
\end{align*} 
where we write $\frac{\infty}{1+\infty}:=1$.
For finite or countable $U$, we equip $\mathcal{M}(U)$ with the pullback topology under $\iota_2$ and we define the \textit{$\mathcal{M}$-topology} as the pullback topology on $\mathcal{Z}(U)$ under the mapping $\iota_1$ or $\iota:=\iota_2 \circ \iota_1$ respectively. \\
For $\tau \in \mathcal{Z}_U$ we call $\iota_1(\tau)$ \textit{stopping rate (of $\tau$)} and for $x \in U$ we call $\iota_1(\tau)(\{x\})$ \textit{stopping rate (of $\tau$) at $x$}. In particular, this extends to the case of $U=I$, where $\mathcal{Z}=\mathcal{Z}_U$.
\end{definition}

\begin{remark}\label{remM} 
    \begin{enumerate}[(i)]
        \item \label{remM3} For our (generally non-symmetric) Dynkin games, we will encounter the product space $\mathcal{Z}_U(E_1) \times \mathcal{Z}_U(E_2)$ that encodes the strategies of both players. The $\mathcal{M}$-topology on the product space will be defined as the product topology of the $\mathcal{M}$-topologies on both components. Equivalently, we could define an embedding $\overline{\iota}:\mathcal{Z}(E_1) \times \mathcal{Z}(E_2) \to [0,1]^U \times [0,1]^U$, $\overline{\iota}:= (\iota,\iota)$ and define the $\mathcal{M}$-topology on $\mathcal{Z}_U(E_1) \times \mathcal{Z}_U(E_2)$ as the pullback topology under $\overline{\iota}$.
        \item For finite or countable $U \subset I$, the mappings $\iota_1$ and $\iota_2$ (and thus the composition $\iota$) are one-to-one and onto, so the spaces $ [0,1]^{U}, \mathcal{M}(U)$ and $\mathcal{Z}_U$ with the $\mathcal{M}$-topology are homeomorphic. In the case of finite $U$, we will see in the next section that this allows to use the compactness and convexity of the space $ [0,1]^{U} \subset \R^{U}$ in order to construct a best response mapping that allows for the application of Kakutani's fixed point theorem \cite[Lemma 20.1]{osborne_rubinstein}. A similar approach was used for discrete processes in \cite{christensen2023markovian,ChristensenLindensjoe_meanvariance}.
        \item For uncountable $U\subset I$, it is well-known that measures on infinite spaces are not characterized by the one-point evaluations. In our terminology, this means that the mapping $\iota_2$ is not one-to-one in this case. This prohibits the same approach for such $U$, especially for $U=I$. This makes it impossible to use the topology $\mathcal M$ fruitfully for the general case. 
    \end{enumerate}
\end{remark}

\begin{lemma}\label{mcont} Let $U\subset I$ finite, $\tau,\tau_1,\tau_2,\tau_3,...\in \mathcal{Z}_U$ with $\tau_n \stackrel{n \to \infty}{\longrightarrow} \tau$ in the $\mathcal{M}$-topology. 
\begin{enumerate}[(i)]
    \item \label{mcont1}For all $T\in[0,\infty)$ we have $$A^{\tau_n}_t(\omega) \stackrel{n \to \infty}{\longrightarrow} A^{\tau}_t(\omega)$$ uniformly in $t\in[0,T]$ for all $\omega\in \{ A^{\tau}_T< \infty\}$.
    \item \label{mcont2} For all $x\in I$ there is a $\PR_x$-nullset $N_1$ such that for all $\eps>0$, $t\in[0,\infty)$ and all $\omega\in \{A^{\tau}_t= \infty\}\setminus N_1$ we have 
    \begin{align*}
        A^{\tau_n}_{t+\eps}(\omega) \stackrel{n \to \infty}{\longrightarrow} A^{\tau}_{t+\eps}(\omega) = A^{\tau}_{t}(\omega)=\infty.
    \end{align*}
    \item \label{mcont3} $\tau_n \stackrel{n \to \infty}{\longrightarrow} \tau$ $\PR_x$-a.s.\ for all $x\in I$.
    \item \label{mcont4} We have $X_{\tau_n}\stackrel{n \to \infty}{\longrightarrow} X_{\tau}$ $\PR_x$-a.s.\ for all $x\in I$.
\end{enumerate}    
\end{lemma}

\begin{proof}
    \eqref{mcont1} By choosing $N\in \N$ such that $\vert \iota(\tau_n)-\iota(\tau)\vert < 1 - \max\{ \iota(\tau)(x):\iota(\tau)(x)<1\}$ for all $n \ge N$ we get
    \begin{align*}
        \tau_n^\infty:=\{x \in U: \iota_1(\tau_n)(\{x\})= \infty \} \subset \tau^\infty:=\{x \in U: \iota_1(\tau)(\{x\})= \infty \}
    \end{align*}for all $n \ge N$.

    Let $T\in [0,\infty), t \in [0,T]$ and $\omega \in \{A^{\tau}_T<\infty\}$. Now
    \begin{align*}
        \infty> A^{\tau}_T(\omega) 
        = \int l^y_T(\omega) \iota_1(\tau)(dy) + \infty\inda{[\eta^{\tau^\infty}(\omega),\infty)}(T)
        = \sum_{y\in U\setminus \tau^{\infty}} l^y_T(\omega) \iota_1(\tau)(\{z\})+ \infty\inda{[\eta^{\tau^\infty}(\omega),\infty)}(T)
    \end{align*} We read off
    \begin{align*}
        0 \le l^y_t(\omega) \le l^y_T(\omega) = 0 \quad \forall y \in \tau^\infty,\quad  \eta^{\tau^\infty}(\omega) > T.
    \end{align*} For $n \ge N$ this yields
    \begin{align*}
        &\vert A^{\tau_n}_t(\omega) - A^\tau_t(\omega) \vert \\
        =&\left\vert \int_{U\setminus \tau^{\infty}_n} l^y_t(\omega)  \iota_1(\tau_n)(dy) + \infty\inda{[\eta^{\tau^\infty}(\omega),\infty)}(t) -\int_{U\setminus \tau^{\infty}} l^y_t(\omega)  \iota_1(\tau)(dy) - \infty\inda{[\eta^{\tau^\infty}(\omega),\infty)}(t) \right\vert\\
        \le& \sum_{z \in U \setminus \tau^\infty} l^y_t(\omega) \vert \iota_1(\tau_n)(\{z\}) -\iota_1(\tau)(\{z\}) \vert 
        \le \sum_{z \in U \setminus \tau^\infty} l^y_T(\omega) \vert \iota_1(\tau_n)(\{z\}) -\iota_1(\tau)(\{z\}) \vert \stackrel{n \to \infty}{\longrightarrow} 0.
    \end{align*}
    
    \eqref{mcont2} Let $t\ge 0$ and $\omega \in \{A^\tau_t= \infty\}$. We have 
    \begin{align*}
        A^{\tau_n}_{t+\eps}(\omega) \ge \int_{U} l^y_{t+\eps}(\omega) \iota_1(\tau_n)(dy)+ \infty\inda{[\eta^{\tau^\infty_n}(\omega),\infty)}(t+\eps),
    \end{align*} while
    \begin{align*}
        \infty=A^\tau_t(\omega) = \int_{U \setminus \tau^\infty} l^y_t(\omega) \iota_1(\tau)(dy) + \infty\inda{[\eta^{\tau^\infty}(\omega),\infty)}(t+\eps) 
    \end{align*} i.e.\ $\eta^{\tau^\infty}(\omega) \le t$. Thus $\eta^{\tau^\infty}(\omega) < t +\eps$. Choose a $\PR_x$-nullset $N_1$ such that $l^z_{t+\eps}(\omega)>0$ for all $\omega \in \Omega \setminus N_1$ and some $z\in \tau^\infty$. Note that $N_1$ can be chosen independently from $t,\eps$ and $z$ since $\tau^\infty$ is finite. The monotonicity of $l^y_{\bigcdot}$ allows to consider only rational $t,\eps$ and then collect all the corresponding exception nullsets. Applying $\iota_1(\tau)(\{z\})=\infty$, this leads to
    \begin{align*}
        A^{\tau_n}_{t+\eps}(\omega) 
        \ge \int_U l^z_{t+\eps}(\omega) \iota_1(\tau_n)(dy) 
        \ge l^z_{t+\eps}(\omega) \iota_1(\tau_n)(\{z\})
        \stackrel{n \to \infty}{\longrightarrow} l^z_{t+\eps}(\omega) \iota_1(\tau)(\{z\}) =\infty
    \end{align*} for all $\omega\in \{A^\tau_t= \infty\}\setminus N_1$.
    
    \eqref{mcont3} Let $\eps>0$. We start by showing $\liminf_{n \to \infty} \tau_n \ge \tau-\eps$. For that let $t\le \tau-\eps$. By definition $A^{\tau_n}_t < E$. By \eqref{mcont1}
    \begin{align*}
        A^{\tau_n}_t \le A^{\tau_n}_{t-\eps} \stackrel{n \to \infty}{\longrightarrow} A^{\tau}_{t-\eps} < E
    \end{align*} and thus $\tau_n >t$ for sufficiently large $n\in \N$. Since $t$ was arbitrary, this finishes the first part. Next, we show the inequality $\limsup_{n \to \infty} \tau_n \le \tau+\eps$. On $\{ A^\tau_{\tau} = \infty\}$, by \eqref{mcont2}, we get $A^{\tau_n}_{\tau+\eps} \stackrel{n \to \infty}{\longrightarrow} \infty \ge E$ $\PR_x$-a.s.\ and thus $\tau_n \le \tau+ \eps$ $\PR_x$-a.s.\ for sufficiently large $n$. On $\{ E \le A^{\tau}_{\tau} < \infty\}$ we have $\eta^{\tau^\infty}> \tau$. Since $\int_0^\infty \ind{X_s \neq y} l^y_{\bigcdot}(ds) = 0$  $\PR_x$-a.s., we infer
    \begin{align*}
        \int^\tau_0 \ind{X_s \not \in \supp(\iota_1(\tau))} dA^\tau_s
        = \sum_{z \in \supp(\iota_1(\tau))} \int^\tau_0 \ind{X_s \not \in \supp(\iota_1(\tau))} l^z_{\bigcdot}(ds) \cdot \iota_1(\tau)(\{z\})
    \end{align*} $\PR_x$-a.s.\ on $\{ E \le A^{\tau}_{\tau} < \infty\}$. This, however, implies $X_\tau \in \supp(\iota_1(\tau))$ $\PR_x$-a.s.\ on $\{ E \le A^{\tau}_{\tau} < \infty\}$. Putting these thing together we get
    \begin{align}
        &A^{\tau_n}_{\tau+\eps} 
        = \theta_\tau \circ A^{\tau_n}_\eps + A^{\tau_n}_\tau 
        = \int \theta_\tau \circ l^z_\eps \iota_1(\tau_n)(dz) + A^{\tau_n}_\tau  \notag\\
        =& \sum_{z \in \supp(\iota_1(\tau_n))} \theta_\tau \circ l^z_\eps \iota_1(\tau_n)(dz)+ A^{\tau_n}_\tau
        \stackrel{n \to \infty}{\longrightarrow} \sum_{z \in \supp(\iota_1(\tau))} \theta_\tau \circ l^z_\eps \iota_1(\tau)(dz)+ A^{\tau}_\tau \label{limp}
    \end{align} $\PR_x$-a.s.\ on $\{ E \le A^{\tau}_{\tau} < \infty\}$. Using $\PR_z(l^z_\eps>0)=1$ for all $z \in I$ and $\PR_x(X_\tau \in \supp(\iota_1(\tau)) \vert E \le A^{\tau}_{\tau} < \infty)=1$ one sees that
    \begin{align*}
        &\PR_x \left(\sum_{z \in \supp(\iota_1(\tau))} \theta_\tau \circ l^z_\eps \iota_1(\tau)(dz)>0 \bigg\vert E \le A^{\tau}_{\tau} < \infty\right)\\
        =& \Ex_x\left[  \Ex_{x} \left[\ind{\sum_{z \in \supp(\iota_1(\tau))} \theta_\tau \circ l^z_\eps \iota_1(\tau)(dz)>0} \big\vert \f_\tau \right] \bigg\vert E \le A^{\tau}_{\tau} < \infty\right]\\
        =& \Ex_x\left[  \Ex_{X_\tau} \left[\ind{\sum_{z \in \supp(\iota_1(\tau))} l^z_\eps \iota_1(\tau)(dz)>0}  \right]\bigg\vert E \le A^{\tau}_{\tau} < \infty \right]
        =1.
    \end{align*} Combined with \eqref{limp} we get
    \begin{align*}
        A^{\tau_n}_{\tau + \eps} > A^\tau_\tau \ge E
    \end{align*} for $\PR_x$-a.a.\ $\omega \in \{E \le A^{\tau}_{\tau} < \infty\}$ and sufficiently large $n$. This implies $\tau_n \le \tau +\eps$ $\PR_x$-a.s.\ on $\{E \le A^{\tau}_{\tau} < \infty\}$. Sending $\eps \searrow 0$ yields the claim. 

    \eqref{mcont4} is immediate by \eqref{mcont3} and continuity of $X$ paths.
\end{proof}

\subsection{$\mathcal{M}^1_x$-topology}
The second topology is based on the idea that on $\mathcal Z$ we should choose a notion of convergence such that convergence of the stopping times just corresponds to the convergence of the distributions of the stopped process generated by it. As a first step, we start with the following definition.

\begin{definition}($\mathcal{M}^1_x$-topology)\\
    Let $x\in I$ and $U\subset I$ open. Further, we equip the space $\mathcal{M}^1(([0,\infty]\times \overline{U})^2)$ with the Prokhorov metric. Let $\kappa_x: \mathcal{Z}^U(E_1)\times \mathcal{Z}^U(E_2) \to \mathcal{M}^1(([0,\infty]\times \overline{U})^2)$ be given by 
    \begin{align*}
        (\tau^{D_1,\lambda_1}, \tau^{D_2,\lambda_2}) \mapsto \PR_x^{(\tau^{D_1,\lambda_1},X_{\tau^{D_1,\lambda_1}},\tau^{D_2,\lambda_2},X_{\tau^{D_2,\lambda_2}})}.
    \end{align*} We define the \textit{$\mathcal{M}^1_x$-topology} as the pullback topology on $\mathcal{Z}^U(E_1)\times \mathcal{Z}^U(E_2)$ under the mapping $\kappa_x$.
\end{definition}

The rest of the subsection will essentially be about showing the compactness properties of this space. The result is the following:

\begin{proposition}\label{comp}Let $U \subset I$ be open with compact closure $\overline{U}\subset I$ (closure taken in $\R$) and $x\in U$. Then, $\mathcal{Z}^U(E_1)\times \mathcal{Z}^U(E_2)$ equipped with the $\mathcal{M}^1_x$-topology is a compact topological space, i.e., the set $\kappa_x(\mathcal{Z}^U(E_1)\times \mathcal{Z}^U(E_2)) \subset \mathcal{M}^1(([0,\infty]\times \overline{U})^2)$ is weakly compact.    
\end{proposition}

    Since $([0,\infty]\times \overline{U})^2$ is compact, $\kappa_x(\mathcal{Z}^U(E_1)\times \mathcal{Z}^U(E_2))$ is relatively weakly compact by Prokhorov's theorem.
As usual the closedness will be much more challenging. We will divide the proof into several lemmas. We begin by introducing some notation and explaining the general approach.

$\mathcal{M}^1(([0,\infty]\times \overline{U})^2)$ is metric, so it suffices to show sequential closedness in order to complete the compactness proof. For this, let $((\tau^{D^1_n,\lambda^1_n},\tau^{D^2_n,\lambda^2_n}))_{n\in \N} \in (\mathcal{Z}^U(E_1)\times \mathcal{Z}^U(E_2))^{\N}$ be a sequence of tuples of Markovian randomized times such that $\kappa_x(\tau^{D^1_n,\lambda^1_n},\tau^{D^2_n,\lambda^2_n})$ converges to a probability measure $\mu \in \mathcal{M}^1(([0,\infty]\times \overline{U})^2) $. We denote the marginals of $\mu$ by $\mu_1,\mu_2,\mu_3,\mu_4$. Closedness now means that the limiting distribution $\mu$ of the sequence is generated by Markovian randomized times. That is, we have to find $(\tau^1,\tau^2) :=(\tau^{D^1,\lambda^1},\tau^{D^2,\lambda^2})\in\mathcal{Z}^U(E_1)\times \mathcal{Z}^U(E_2)$ such that $\kappa_x(\tau^1,\tau^2) = \mu$. To abbreviate we set $(\tau^1_n,\tau^2_n):=(\tau^{D^1_n,\lambda^1_n},\tau^{D^2_n,\lambda^2_n})$ for each $n\in\N$. All these objects and notations are fixed throughout this section. In addition, all the notions introduced in the Lemmas \ref{aux1}, \ref{aux2}, \ref{aux3}, and \ref{aux4} are fixed for the entire section as well.
         
In Lemma \ref{aux1} we construct the candidate for the continuation regions $D_1:=(L^1,R^1)$, $D_2:=(L^2, R^2)$ of $\tau^1,\tau^2$ around $x$ via the supports of the measures $\mu_2,\mu_4$ and derive the order of $L^i,R^i, i\in\{1,2\}$ and the limits of the supports of the measures $\PR_x^{X_{\tau^1_n}},\PR_x^{X_{\tau^2_n}}$. In Lemma \ref{aux2} we show a consistency condition that justifies to construct the continuation regions of $\tau^1,\tau^2$ as the supports of the measures $\mu_2,\mu_4$. In Lemma \ref{aux3} we construct $\lambda^i$, $i\in\{1,2\}$ as the vague limit of a sequence of measures related to the sequence $(\lambda^i_n)_{n\in\N}$. The main issue to overcome is that the measures $\lambda^i_n \in \rm((L^i_n,R^i_n))$ may have different domains. Here we use the ordering \eqref{consistency1} to find a common domain. In Lemma \ref{aux4} we will establish $\tau^i_n\to \tau^i$ $\PR_x$-a.s.\ along suitably chosen subsequences. This is obtained by applying Lemma \ref{aux2} and identifying local times stopped at certain (first entry) times for fixed $\omega$ as test functions for vague convergence to obtain suitable convergence of the additive functionals defining the stopping times $\tau^1,\tau^2$. Based on this, we show the remaining closedness statement of the proof of Proposition \ref{comp}.

\begin{lemma}\label{aux1}
    Set 
    \begin{align*}
        &L^1:=\sup \left\{y \in I\cap(-\infty,x]:  \mu_2(I \cap (-\infty, y))=0 \right\},\\
        &R^1:=\inf \left\{y \in I\cap[x,\infty):  \mu_2(I \cap (y,\infty))=0 \right\},\\
        &L^2:=\sup \left\{y \in I\cap(-\infty,x]:  \mu_4(I \cap (-\infty, y))=0 \right\},\\
        &R^2:=\inf \left\{y \in I\cap[x,\infty):  \mu_4(I \cap (y,\infty))=0 \right\}
    \end{align*} using the conventions $\sup \varnothing:= \inf I $, $\inf \varnothing:=\sup I$. With that, we define $D^1:=(L^1,R^1)$, $D^2:=(L^2, R^2)$. Note that this includes the option $D^i=(x,x)=\varnothing$ for $i\in\{1,2\}$.
    Let
    \begin{align*}
        &L^1_n:= \sup (I \setminus D^1_n) \cap (-\infty,x],
        && R^1_n:= \inf (I \setminus D^1_n) \cap [x,\infty),\\
        &L^2_n:= \sup (I \setminus D^2_n) \cap (-\infty,x],
        && R^2_n:= \inf (I \setminus D^2_n) \cap [x,\infty),
    \end{align*} and for $i\in\{1,2\}$ set
    \begin{align}
        \Tilde{L}^i:=\limsup_{n\to \infty} L^i_n, \quad \Tilde{R}^i:=\liminf_{n\to \infty} R^i_n. \label{lr}
    \end{align} Then, for $i\in\{1,2\}$ we have
    \begin{align}
        -\infty <\inf U \le  \Tilde{L}^i\le L^i \le x\le R^i \le\Tilde{R}^i  \le \sup U < \infty. \label{consistency1}
    \end{align}
\end{lemma}
    
\begin{proof}
    Note that $I \setminus D^i_n\supset \partial U$ for each $n$ and $(-\infty,x]\cap \partial U \neq \varnothing \neq [x,\infty) \cap \partial U$, so $\inf U \le \Tilde{L}^i$ and $\Tilde{R}^i \le \sup U$. $-\infty <\inf U$ and $\sup U < \infty$ are a consequence of compactness of $\overline{U}$. By choice of $L^i_n, R^i_n$, we have
    \begin{align*}
        \PR_x^{X_{\tau^i_n}} ( I\cap(  (-\infty,L^i_n)  \cup  (R^i_n,\infty))) =0
    \end{align*} for all $n\in \N$. Since for all $\eps > 0$ and all $N\in \N$, there is some $n\ge N$ such that $L^i_n > \Tilde{L}^i-\eps$ or $R^i_n< \Tilde{R}^i+\eps$, respectively, this implies
    \begin{align*}
        \liminf_{n \to \infty}\PR_x^{X_{\tau^i_{n}}}(I \cap (-\infty, \Tilde{L}^i-\eps))= 0 \quad \text{and}\quad
        \liminf_{n \to \infty}\PR_x^{X_{\tau^i_{n}}}(I \cap (\Tilde{R}^i+\eps, \infty))=0.
    \end{align*} By the Portmanteau theorem we get
    \begin{align}
        &\mu_{2i}(I \cap( (-\infty,\Tilde{L}^i)\cup(\Tilde{R}^i,\infty))) =\mu_{2i}(  I \cap((-\infty,\Tilde{L}^i))) +\mu_{2i}(I \cap( (\Tilde{R}^i,\infty))) \notag \\
        =& \lim_{\eps \searrow 0} \big(\mu_{2i}(I \cap((-\infty,\Tilde{L}^i-\eps))) +\mu_{2i}( I \cap((\Tilde{R}^i +\eps,\infty)))\big) \notag\\
        \le& \lim_{\eps \searrow 0}\big( \liminf_{n \to \infty} \PR_x^{ X_{\tau_{n}}}( I \cap((-\infty,\Tilde{L}^i-\eps))) +\liminf_{n \to \infty}\PR_x^{X_{\tau_{n}}} ( I \cap( ( \Tilde{R}^i+\eps,\infty)))\big)= 0. \notag
    \end{align} With the definition of $L^i,R^i$, we obtain $\Tilde{L}^i\le L^i$ and $R^i\le \Tilde{R}^i$. $L^i\le x\le R^i$ is trivial.
    \end{proof}
    
    \begin{lemma}\label{aux2}
        Let $i \in \{1,2\}$ and $\Tilde{L}^i, \Tilde{R}^i$ as in \eqref{lr}.
    \begin{enumerate}[(i)]
        \item \label{aux2_1}If $L^i=x=R^i$ then $\liminf_{n\to \infty} \lambda^i_n((x-\eps,x+\eps)\cap D^i_n)=\infty$ for all $\eps>0$ or $\Tilde{L}^i=L^i$ or $\Tilde{R}^i = R^i$.
        \item \label{aux2_2}If $L^i=x<R^i$ then $\liminf_{n\to \infty} \lambda^i_n((x-\eps,x+\eps)\cap D^i_n)=\infty$ for all $\eps>0$ or $\Tilde{L}^i=L^i$.
        \item \label{aux2_3}If $L^i<x= R^i$ then $\liminf_{n\to \infty} \lambda^i_n((x-\eps,x+\eps)\cap D^i_n) = \infty$ for all $\eps>0$ or $\Tilde{R}^i = R^i$.
        \item \label{aux2_4}If $L^i<x< R^i$ then $\liminf_{n\to \infty} \lambda^i_n((L^i-\eps,L^i+\eps)\cap D^i_n)=\infty$ for all $\eps>0$ or $\Tilde{L}^i=L^i$ as well as $\liminf_{n\to \infty} \lambda^i_n((R^i-\eps,R^i+\eps)\cap D^i_n) = \infty$ for all $\eps>0$ or $\Tilde{R}^i = R^i$. 
    \end{enumerate}  
\end{lemma}

\begin{remark}
        It will turn out, that the cases the \eqref{aux2_2} and \eqref{aux2_3} of Lemma \ref{aux2} are in fact empty, cf. Lemma \ref{aux3} \eqref{1}.
\end{remark}
    
    \begin{proof}
         If any of \eqref{aux2_1}, \eqref{aux2_2}, \eqref{aux2_3} or \eqref{aux2_4} were false, using that $\Tilde{L}^i \neq L^i$ implies $\Tilde{L}^i < L^i$ and $\Tilde{R}^i \neq R^i$ implies $\Tilde{R}^i < R^i$ by Lemma \ref{aux1} \eqref{consistency1}, we could find some $\eps>0$ such that 
    \begin{align*}
        K:=\limsup_{n \to \infty} \lambda^i_n ([L^i-\eps,x+\eps] ) <\infty \quad \text{or}\quad \limsup_{n \to \infty}\lambda^i_n([x-\eps, R^i+\eps] ) <\infty
    \end{align*} with $x+\eps\in (L^i,R^i)$, $\Tilde{L}^i<L^i-\eps$ or $x-\eps \in (L^i, R^i)$, $\Tilde{R}^i> R^i+ \eps$ respectively (in case of \eqref{aux2_4} we get both). W.l.o.g. we only argue for $L^i$. Note that $\tau^i_n$ stops immediately upon hitting $L^i_n$, so if $X$ reaches $L^i-\eps$ without being stopped by $\tau^i_n$, and then hits $L^i_n$ before reentering $L^i-\frac{\eps}{2}$, the process will be stopped in $I \setminus (L^i-\frac{\eps}{2},\infty)$. This can be rephrased as $X_{\tau^i_n}\in I \setminus (L^i-\frac{\eps}{2},\infty)$ $\PR_x$-a.s.\ on $\{ \eta_{L^i-\eps}< \tau^i_n\} \cap\{\theta_{\eta_{L^i-\eps}}\circ \eta_{L^i_n-\eps}<\theta_{\eta_{L^i-\eps}}\circ \eta_{L^i-\frac{\eps}{2}}\}$. Combining this with the fact that $\Ex_{z}[\ind{\eta_a<\eta_b}]$, $a\le z \le b$ is a continuous function of $a$, \cite[p. 119]{ito1974diffusion} and $\Tilde{L}^i = \limsup_{n \to \infty} L^i_n$, we obtain the following estimate, where the right hand side is positive since the exponential term is positive $\PR_x$-a.s.\ and $\Ex_{z}[\ind{\eta_a<\eta_b}]$ is positive for all $a<z<b$:
    \begin{align*}
         &\limsup_{n\to \infty} \PR_x(X_{\tau^i_n}\in I\setminus (L^i-\frac{\eps}{2},\infty))\\
         \ge& \limsup_{n \to \infty} \PR_x(\{ \eta_{L^i-\eps}< \tau^i_n\} \cap \{\theta_{\eta_{L^i}}\circ \eta_{L^i_n-\eps}<\theta_{\eta_{L^i-\eps}}\circ \eta_{L^i-\frac{\eps}{2}}\})\\
         =&  \limsup_{n \to \infty}  \Ex_x[\ind{\eta_{L^i-\eps}< \tau^i_n} \Ex_{X_{\eta_{L^i-\eps}}}[\ind{\eta_{L^i_n}<\eta_{L^i-\frac{\eps}{2}}}]]\\
         =&\limsup_{n \to \infty}  \Ex_x[\ind{\eta_{L^i-\eps}< \tau^i_n} ]\Ex_{L^i-\eps}[\ind{\eta_{L^i_n}<\eta_{L^i-\frac{\eps}{2}}}]\\
         \ge&\Ex_{L^i-\eps}[\ind{\eta_{\Tilde{L}^i}<\eta_{L^i-\frac{\eps}{2}}}]\limsup_{n \to \infty}  \Ex_x[\ind{\eta_{L^i-\eps}< \tau^i_n} ]\\
         =&\Ex_{L^i-\eps}[\ind{\eta_{\Tilde{L}^i}<\eta_{L^i-\frac{\eps}{2}}}]\limsup_{n \to \infty}\PR_x(\eta_{L^i-\eps}< \tau^i_n) \\
         \ge&\Ex_{L^i-\eps}[\ind{\eta_{\Tilde{L}^i}<\eta_{L^i-\frac{\eps}{2}}}] \limsup_{n \to \infty} \PR_x(E_i > A^{D^i_n,\lambda^i_n}_{\eta_{L^i-\eps}}) \\
         \ge& \Ex_{L^i-\eps}[\ind{\eta_{\Tilde{L}^i}<\eta_{L^i-\frac{\eps}{2}}}]\limsup_{n \to \infty} \PR_x\left(E_i > \int l^y_{\eta_{L^i-\eps}}  \lambda^i_n (dy) \right)\\
         \ge& \Ex_{L^i-\eps}[\ind{\eta_{\Tilde{L}^i}<\eta_{L^i-\frac{\eps}{2}}}]\limsup_{n \to \infty} \Ex_x\left[e^{-\int l^y_{\eta_{L^i-\eps}} d \lambda^i_n}\right]\\
         \ge& \Ex_{L^i-\eps}[\ind{\eta_{\Tilde{L}^i}<\eta_{L^i-\frac{\eps}{2}}}]\limsup_{n \to \infty} \Ex_x[\ind{l^y_{\eta_{L^i-\eps}}=0 \forall y \in I\setminus[L^i-\eps,x+\eps]} e^{-\sup_{y\in [L^i-\eps,x+\eps]}l^y_{\eta_{L^i-\eps}}(K+\eps)}]\\
         \ge&\Ex_{L^i-\eps}[\ind{\eta_{\Tilde{L}^i}<\eta_{L^i-\frac{\eps}{2}}}] \Ex_x[\ind{\eta_{L^i-\eps}<\eta_{x+\eps}}e^{-\sup_{y\in [L^i-\eps,x+\eps]}l^y_{\eta_{L^i-\eps}}(K+\eps)}]>0.
    \end{align*} Now by definition of $L^i$ and Portmanteau theorem we get a contradiction via
    \begin{align*}
        0=\mu_{2i}(I\setminus (L^i-\frac{\eps}{2},\infty)) \ge \limsup_{n\to \infty} \PR_x(X_{\tau^i_n}\in I\setminus (L^i-\frac{\eps}{2},\infty))>0.
    \end{align*}
    \end{proof}

\begin{lemma}\label{aux3} 
    \begin{enumerate}[(i)]
        \item \label{1} Let $i\in \{1,2\}$. If $L^i=x$ or $R^i=x$, then there is a subsequence $(\tau_{n_k}^i)_{k\in \N}$ of $(\tau^i_n)_{n \in \N}$ such that $\tau^i_{n_k}\stackrel{k \to \infty}{\longrightarrow}0$ $\PR_x$-a.s.\ and thus $R^i=x$ and $L^i=x$.
        \item Suppose $L^i<x<R^i$. Let $(\delta_m)_{m \in \N}\in ((0,(R^i-x)\wedge(x-L^i)))^{\N}$ with $\delta_m \searrow 0$ and 
            $$d\lambda_{n,m}^i:= \inda{[L^i+ \delta_m, R^i-\delta_m]}  d \lambda^i_n \in \rm((L^i_n,R^i_n)). $$
        \begin{enumerate}
            \item \label{1a}For each fixed $m\in\N$, there is some $N_m\in \N$ such that the measure $\lambda_{n,m}^i\vert_{[L^i+ \delta_m, R^i-\delta_m]}\in \rm([L^i+ \delta_m, R^i-\delta_m])$ is well defined for all $n\ge N_m$.
            \item \label{1b} For all $m\in \N$ and $n\ge N_m$ we extend $\lambda_{n,m}^i\vert_{[L^i+ \delta_m, R^i-\delta_m]}$ by 0 to a measure on $(L^i,R^i)$ and denote the extension by $\Tilde{\lambda}^i_{n,m}$. For each $m \in \N$ we have $$\sup_{n\ge N_m}\Tilde{\lambda}^i_{n,m}((L^i, R^i))<\infty.$$ 
            \item \label{constr} Let $(n_k)_{k\in \N}\in \N^{\N}$ be strictly increasing. There is a subsequence $(n_{k_m})_{m\in \N}$ such that $n_{k_m}\ge N_m$ for all $m\in \N$ and
                \begin{align*}
                    \Tilde{\lambda}^i_{n_{k_m},m} \stackrel{m \to \infty}{\longrightarrow}\Tilde{\lambda}^i 
                \end{align*} vaguely for some $\Tilde{\lambda}^i\in \rm((L^i,R^i))$.                 
        \end{enumerate}
            \item \label{2} Based on this, we can construct $\lambda^1,\lambda^2$ as follows:
            \begin{enumerate}
                \item \label{2a}If $L^1=x,L^2=x$ we set $\lambda^1:= \lambda^2:=0\in \rm(\varnothing)$.
                \item \label{2b}If $L^i=x$ and $L^j<x<R^j$ for $i\neq j \in \{1,2\}$, we set $\lambda^i:=0$ as before and apply \eqref{constr} with $(n_k)_{k\in \N}$ being such that $\tau^i_{n_k}\stackrel{k \to \infty}{\longrightarrow} 0$ $\PR_x$-a.s.\ to find a subsequence $(n_{k_m})_{m\in \N}$ and $\Tilde{\lambda}^j\in \rm((L^j,R^j))$ such that $\Tilde{\lambda}^j_{n_{k_m},m} \stackrel{m \to \infty}{\longrightarrow}\Tilde{\lambda}^j$ vaguely. Then we set $\lambda^j:=\Tilde{\lambda}^j$.
                \item \label{2c}If $L^i<x<R^i$ for all $i\in \{1,2\}$, we apply \eqref{constr} with $(n)_{n\in \N}$ in place of $(n_k)_{k\in \N}$ to find a subsequence $(n_{k})_{k\in \N}$ and $\Tilde{\lambda}^1\in \rm((L^1,R^1))$ such that $\Tilde{\lambda}^1_{n_{k},k} \stackrel{k \to \infty}{\longrightarrow}\Tilde{\lambda}^1$ vaguely.  Then we apply \eqref{constr} to the subsequence $(n_k)_{k\in \N}$ we obtained in the last step to find a further subsequence $(n_{k_m})_{m\in \N}$ and $\Tilde{\lambda}^2\in \rm((L^2,R^2))$ such that $\Tilde{\lambda}^2_{n_{k_m},m} \stackrel{m \to \infty}{\longrightarrow}\Tilde{\lambda}^2 $ vaguely. We now set $\lambda^1:=\Tilde{\lambda}^1$ and $\lambda^2:= \Tilde{\lambda}^2$.
            \end{enumerate}  
        \item \label{3} There is a subsequence $(n_k)_{k\in \N}$ such that for each $i\in \{1,2\}$ either $\tau^i_{n_k}\stackrel{k \to \infty}{\longrightarrow}0$ $\PR_x$-a.s.\ or $\Tilde{\lambda}^i_{n_{k},k} \stackrel{k \to \infty}{\longrightarrow}\lambda^i $ vaguely on $\rm((L^i,R^i))$ with $L^i<x<R^i$.          
    \end{enumerate}
\end{lemma}

    \begin{proof}
        \eqref{1} By Lemma \ref{aux2}, $x=L^i = \Tilde{L}^i$, $x=R^i = \Tilde{R}^i$ or $\liminf_{n \to \infty} \lambda^i_n((x-\eps,x+\eps)\cap D^i_n)=\infty$ for all $\eps>0$. We start by treating the case $x=L^i= \Tilde{L}^i$. We choose $(n_k)_{k \in \N}$ such that $L^i_{n_k} \stackrel{k \to \infty}{\longrightarrow} \Tilde{L}^i = x$. Then
        \begin{align*}
            0 \le \tau^i_{n_k} \le \eta_{L^i_{n_k}} \stackrel{k\to \infty}{\longrightarrow} 0\quad \PR_x\text{-a.s.}
        \end{align*} The case $x = R^i= \Tilde{R}^i$ case be treated analogously.
        Next we assume $\liminf_{n \to \infty} \lambda^i_n((x-\eps,x+\eps)\cap D^i_n)=\infty$ for all $\eps>0$. Let $\eps>0$. The assumption yields
        \begin{align*}
            A^{D^i_n,\lambda^i_n}_{\tau^{(x-\eps,x+ \eps)}} \ge \int l^y_{\tau^{(x-\eps,x+ \eps)}}  \lambda^i_n(dy) \stackrel{n \to \infty}{\longrightarrow}\infty
        \end{align*} $\PR_x$-a.s.\ since $y\mapsto l^y_{\tau^{(x-\eps,x+ \eps)}}$ is continuous with $\PR_x(l^y_{\tau^{(x-\eps,x+ \eps)}}>0)=1$. From this we infer
        \begin{align*}
            \PR_x(\tau^i_n > \tau^{(x-\eps,x+ \eps)}) 
            = \PR_x(E_i> A^{D^i_n,\lambda^i_n}_{\tau^{(x-\eps,x+ \eps)}})
            \ge \PR_x\left( E_i > \int l^y_{\tau^{(x-\eps,x+ \eps)}} \lambda^i_n(dy)\right) 
            \stackrel{n \to \infty}{\longrightarrow} 0.
        \end{align*} Since $\eps>0$ was arbitrary, we can choose a subsequence $(\tau^i_{n_k})_{k \in \N}$ of $(\tau^i_n)_{n \in \N}$ such that $\sum_{k\in \N}\PR_x(\tau^i_{n_k} > \tau^{(x-\frac{1}{k},x + \frac{1}{k})})<\infty$. By Borel Cantelli's lemma $\PR_x(\limsup_{k \to \infty}\{\tau^i_{n_k} > \tau^{(x-\frac{1}{k},x + \frac{1}{k})}\})=0$. This implies
        \begin{align*}
            0\le \limsup_{k \to \infty} \tau^i_{n_k} \le \lim_{k \to \infty} \tau^{(x-\frac{1}{k},x + \frac{1}{k})} = 0 \quad \PR_x\text{-a.s.}
        \end{align*}
        Since $X$ has continuous paths, we infer $X_{\tau^i_{n_k}} \to x$ $\PR_x$-a.s.\ and thus $\PR_x^{X_{\tau_{n_k}^i}} \to \delta_x$, where $\delta_x$ denotes the Dirac measure in $x$. By uniqueness of weak limits, this yields $\mu_{2i}= \delta_x$. With this, $L^i=x= R^i$ follows from the definition.
        
        \eqref{1a} By Lemma \ref{aux1} \eqref{consistency1} given $\delta_m\in(0,(R^i-x)\wedge(x-L^i)),m\in \N$ there is some $N_m\in \N$ such that for all $n \ge N_m$ we have $L^i_n \le \limsup_{n \to \infty} L^i_n +\delta_m = \Tilde{L}^i+\delta_m \le L^i+\delta_m<R^i-\delta_m \le \Tilde{R}^i-\delta_m = \liminf_{n \to \infty} R^i_n - \delta_m \le R^i_n$. I.e.\ $(L^i_n,R^i_n)\supset [L^i+\delta_m,R^i-\delta_m]$ and thus $\lambda_{n,m}^i\vert_{[L^i+ \delta_m, R^i-\delta_m]}\in \rm([L^i+ \delta_m, R^i-\delta_m])$ is well defined.
        
        \eqref{1b} By definition $\Tilde{\lambda}^i_{n,m}((L^i, R^i)) = \lambda^i_{n,m}([L^i+ \delta_m, R^i-\delta_m])$. Since Radon measures are finite on compacts, we have $\lambda^i_{n,m}([L^i+ \delta_m, R^i-\delta_m])<\infty$ for all fixed $n\ge N_m$. Thus it suffices to show $\limsup_{n\to \infty}\lambda^i_{n,m}([L^i+ \delta_m, R^i-\delta_m])<\infty$. We argue by contradiction so assume this were false. Since
        \begin{align*}
            \lambda^i_{n,m}([L^i+ \delta_m, R^i-\delta_m]).
            =& \lambda_n^i(\{x\}) +  \lambda_n^i([L^i+\delta_m,x)) + \lambda_n^i((x,R^i-\delta_m]),
        \end{align*} this implies one of the following:
        \begin{enumerate}
            \item \label{cc1} $\lambda_{n_k}^i(\{x\}) \to \infty$ for $k \to \infty$ or
            \item \label{cc2} $\lambda_{n_k}^i([L^i+\delta_m,x)) \to \infty$ for $k \to \infty$ or
            \item \label{cc3} $\lambda_{n_k}^i((x,R^i-\delta_m]) \to \infty$ for $k \to \infty$.
        \end{enumerate} If \eqref{cc1} using $\PR_x(l^x_{\tau^{(L^i+\delta_m, R^i-\delta_m)}}>0)=1$ we obtain
    \begin{align*}
         \PR_x^{X_{\tau^i_{n_k}}}(I \setminus [L^i+\delta_m, R^i- \delta_m]) \le \PR_x(l^x_{\tau^{(L^i+\delta_m, R^i-\delta_m)}} \lambda^i_{n_k,m}(\{x\})<E_i) \stackrel{k\to \infty}{\longrightarrow}0.
    \end{align*} 
    With Portmanteau's theorem, this implies 
    \begin{align*}
        &\mu_{2i}(I \setminus [L^i+\delta_m, R^i- \delta_m])
        \le\liminf_{n \to \infty} \PR_x^{X_{\tau^i_{n}}}(I \cap (-\infty,L^i+\delta_m)) \\
        &\le \liminf_{n \to \infty}\PR_x^{X_{\tau^i_{n}}}(I \setminus [L^i+\delta_m, R^i- \delta_m]) 
        =0,
    \end{align*} contradicting the infimal choice of $L^i$. We finish by treating the case \eqref{cc2}, \eqref{cc3} being analogous. First note that for all $y\in [L^i+\delta_m,x] $ we have $l^y_{\eta_{L^i-\delta_{m+1}}} > 0$ $\PR_x$-a.s. Using the joint continuity of the local time, we infer
    \begin{align*}
         \inf_{y\in [L^i+\delta_m,x]}l^y_{\eta_{L^i-\delta_{m+1}}}>0 \quad \PR_x\text{-a.s.}
    \end{align*} That yields
    \begin{align*}
         &\mu_{2i}(I \setminus (-\infty,L^i+\delta_{m+1}))
         \le \lim_{k\to \infty}\PR_x^{X_{\tau^i_{n_k}}}(I \setminus [L^i+\delta_{m+1}, \infty)) \\
         \le&  \lim_{k\to \infty}\PR_x\left( \int_{[L^i+\delta_m,x) } l^y_{\eta_{L^i-\delta_{m+1}}} \lambda^i_{n_k}(dy)<E_i \right) \\
         \le& \lim_{k\to \infty}\PR_x \left( \left(\inf_{y\in [L^i+\delta_m,x]}l^y_{\eta_{L^i-\delta_{m+1}}} \right) \lambda^i_{n_k}([L^i+\delta_m,x)) \right)
         =0,
    \end{align*} once again contradicting the choice of $L^i$. 

    \eqref{constr} Since the following arguments do not depend on $(n_k)_{k\in \N}$, we assume we assume $(n_k)_{k\in \N} = (n)_{n\in\N}$ w.l.o.g. By \cite[Korollar 31.3]{Bauer1992} for an open or closed interval $\chi\subset\R$ and each $c>0$ the set $\{\mu \in \rm(\chi):\mu(\chi)\le c\}$ is vaguely compact. Thus by \eqref{1b} the sequence $(\Tilde{\lambda}^i_{n,m})_{n\ge N_m}$ has a subsequence $(\Tilde{\lambda}^i_{n^1_k,1})_{k\in\N}$ with $n^1_1\ge N_1$ that converges vaguely to a limit $\Tilde{\lambda}^i_1\in \rm((L^i,R^i))$. Now in the same way, for each $m \ge 2$ we iteratively choose subsequences $(\Tilde{\lambda}^i_{n^m_k,m})_{k \in \N}$ of $(\Tilde{\lambda}^i_{n^{m-1}_k,m-1})_{k \in \N}$ with $n^m_1\ge N_m$ that converge vaguely to limit points $\Tilde{\lambda}^i_m\in \rm((L^i,R^i))$. 
    
    Next we show $\Tilde{\lambda}^i_m\vert_{[L^i+\delta_{j},R^i-\delta_{j}]} = \Tilde{\lambda}^i_{m'} \vert_{[L^i+\delta_{j},R^i-\delta_{j}]}$ for all $m'\ge m >j\in \N$ and $n \ge N$. For that let $f:[L^i+\delta_{j}, R^i-\delta_{j}] \to \R$ be continuous and bounded. $f$ can be extended to a continuous bounded function $\overline{f}:(L^i,R^i)\to \R$ such that $\overline{f}\vert_{(L^i,L^i+\delta_{j+1})\cup (R^i-\delta_{j+1},R^i)}\equiv 0$. By definition $\lambda^i_{n,m}\vert_{[L^i+\delta_m,R^i-\delta_m]} =\lambda^i_{n,m'}\vert_{[L^i+\delta_m,R^i-\delta_m]}$ for all $m\in\N$ and all $n\ge N$. From $\supp(f) \subset [L^i+\delta_m,R^i-\delta_m]$ we infer
    \begin{align*}
        \int_{(L^i,R^i)}\overline{f} d\Tilde{\lambda}^i_m= \lim_{k \to \infty}\int_{(L^i,R^i)} \overline{f} d\lambda^i_{n^{m'}_k,m}= \lim_{k \to \infty}\int_{(L^i,R^i)} \overline{f} d\lambda^i_{n^{m'}_k,m'}= \int_{(L^i,R^i)} \overline{f} d\Tilde{\lambda}^i_{m'}
    \end{align*} which proves the above claim. 
    With that we are able to consistently define the measure $\Tilde{\lambda}^i\in\rm((L^i,R^i))$ via
    \begin{align*}
        \Tilde{\lambda}^i(K):=\Tilde{\lambda}^i_j(K) 
    \end{align*} for any compact $K\subset(L^i,R^i)$, where $j\ge 2$ is chosen such that $K\subset[L^i+\delta_{j-1},R^i-\delta_{j-1}]$.
    Now we show $\lambda^i_{n^{m}_m,m}\to \lambda^i$ vaguely for $m \to \infty$. For this, let $f\in \ccc_{\cpt}((L^i,R^i))$, i.e., a continuous function with compact support. We choose $j\in \N$ such that $\supp(f)\subset [L^i+\delta_j,R^i-\delta_j]$. Now as above and by construction of $\lambda^i$ we obtain
    \begin{align*}
        \lim_{m\to \infty} \int_{(L^i,R^i)} f d\lambda^i_{n^{m}_m,m} =\lim_{m\to \infty} \int_{(L^i,R^i)} f d\lambda^i_{n^{m}_m,j+1} = \int_{(L^i,R^i)} f d \Tilde{\lambda}^i_{j+1} = \int_{(L^i,R^i)} f d \lambda^i
    \end{align*}as claimed.

    \eqref{2} For \eqref{2a} and \eqref{2c} there is nothing to prove and for \eqref{2b} we find the subsequence $(n_k)_{k\in \N}$ by \eqref{1}.    

    \eqref{3} In case of \eqref{2a} by \eqref{1} we find $(n_k)_{k \in \N}$ and $(n_k')_{k\in \N}$ such that $\tau^1_{n_k} \stackrel{k \to \infty}{\longrightarrow} 0$ and $\tau^2_{n_k'} \stackrel{ k \to \infty}{\longrightarrow} 0$ $\PR_x$-a.s. This implies $X_{\tau^1_{n_k}} \stackrel{k \to \infty}{\longrightarrow} x$ and $X_{\tau^2_{n_k'}} \stackrel{k \to \infty}{\longrightarrow} x$ $\PR_x$-a.s. By uniqueness of weak limits we obtain $\mu_1 = \delta_0, \mu_2 = \delta_x, \mu_3 = \delta_0, \mu_4 = \delta_x$, i.e.\ $\mu = \delta_0 \otimes \delta_x \otimes_0 \otimes \delta_x$. Thus $(\tau^1_n,X_{\tau^1_n},\tau^2_n,X_{\tau^2_n}) \to \mu $ in distribution already yields $(\tau^1_n,X_{\tau^1_n},\tau^2_n,X_{\tau^2_n}) \to(0,x,0,x)$ in probability. Thus there is $(n_k'')_{k\in \N}$ such that $(\tau^1_{n_k''},X_{\tau^1_{n_k''}},\tau^2_{n_k''},X_{\tau^2_{n_k''}})\stackrel{k \to \infty}{\longrightarrow} (0,x,0,x)$ $\PR_x$-a.s. In the cases \eqref{2b} and \eqref{2c} it is sufficient to choose $(n_{k_m})_{m\in \N}$.  
\end{proof}

\begin{lemma}\label{aux4}
    Let $(n_k)_{k\in \N}$ be the sequence provided by Lemma \ref{aux3} \eqref{3}. We consider $i\in\{1,2\}$ such that $\Tilde{\lambda}^i_{n_{k},k} \stackrel{k \to \infty}{\longrightarrow}\lambda^i $ vaguely on $\rm((L^i,R^i))$ with $L^i<x<R^i$.
    \begin{enumerate}[(i)]
        \item \label{aux4_1}There is a $\PR_x$-nullset $N_1$ such that for all $t\ge 0$ and all $\omega\in \{ t < \tau^{D^i}\}\setminus N_1$
        \begin{align*}
            A^{D^i_{n_k},\lambda^i_{n_k}}_t(\omega)\stackrel{k \to \infty}{\longrightarrow} A^{D^i,\lambda^i}_t(\omega) \le  A^{D^i,\lambda^i}_T(\omega)<\infty.
        \end{align*}
        \item \label{aux4_2}We have $A^{D^i,\lambda^i}_{\tau^i+\eps}>E_i$ $\PR_x$-a.s.\ for every $\eps >0$.
        \item \label{aux4_3}There is a $\PR_x$-nullset $N_2$ such that for all $t\ge 0$ and all $\omega\in \{ t > \tau^{D^i}\}\setminus N_2$
        \begin{align*}
            A^{D^i_{n_k},\lambda^i_{n_k}}_t(\omega) \stackrel{k\to \infty}{\longrightarrow} \infty.
        \end{align*}
        \item \label{aux4_4} We have $\tau^i_{n_k} \stackrel{k\to \infty}{\longrightarrow} \tau^i$ and $X_{\tau^i_{n_k}}\stackrel{k\to \infty}{\longrightarrow} X_{\tau^i}$ $\PR_x$-a.s.
    \end{enumerate}
\end{lemma}
   
\begin{proof}
    \eqref{aux4_1} Let $N_1'$ be a nullset such that $\tau^{(L^i+\eps, R^i-\eps)}(\omega) \stackrel{\eps \searrow 0}{\longrightarrow} \tau^{(L^i,R^i)}(\omega)= \tau^{D^i}(\omega)$. This yields $A^{D^i,\lambda^i}_{t\wedge\tau^{(L^i+\eps, R^i-\eps)}(\omega)}(\omega) \stackrel{\eps \searrow 0}{\longrightarrow} A^{D^i,\lambda^i}_{t}(\omega)$ and $A^{D^i_{n_k},\lambda^i_{n_k}}_{t\wedge\tau^{(L^i+\eps, R^i-\eps)}(\omega)}(\omega) \stackrel{\eps \searrow 0}{\longrightarrow} A^{D^i_{n_k},\lambda^i_{n_k}}_{t}(\omega)$ uniformly in $k\in \N$ for each $\omega \in \Omega\setminus N_2$ and all $t<\tau^{D^i}(\omega)$, which justifies the interchanging of limits in the calculation to follow.
    In the next line we will use that for $D\subset I$ open and $\lambda\in \rm(D)$ by definition we have $A^{D,\lambda}_{t}(\omega) = \int l^y_t(\omega) \lambda(dy)$ for $t<\tau^D(\omega)$.
    Then we invoke the fact that $\lambda^i_{n_k}\vert_{[L^i+\eps, R^i-\eps]} = \Tilde{\lambda}^i_{n_k,k} \vert_{[L^i+\eps, R^i-\eps]} $ for sufficiently large $k$ by definition.
    To find the limit in $k$, let $l^{\bigcdot}_{t \wedge \tau^{(L^i+\eps, R^i-\eps)}}(\omega) :(L^i,R^i)\to \R, y \mapsto l^y_{t \wedge \tau^{(L^i+\eps, R^i-\eps)}}(\omega)$ and note that $l^y_{t \wedge \tau^{(L^i+\eps, R^i-\eps)}}(\omega) = 0$ for all $y\not\in(L^i+\eps,R^i-\eps)$ and all $\omega\in \Omega$. Thus there is a nullset $N_1''$ such that $l^{\bigcdot}_{t \wedge \tau^{(L^i+\eps, R^i-\eps)}}(\omega)\in \ccc_{\cpt}((L^i,R^i))$ for all $\omega\in \Omega \setminus N_1''$ and all $t<\tau^{D^i}(\omega)$, so we can use vague convergence. 
    Proceeding as explained yields
    \begin{align*}
        &\lim_{k \to \infty} A^{D^i_{n_k},\lambda^i_{n_k}}_{t}(\omega)
        =\lim_{k \to \infty} \lim_{\eps \searrow 0 } A^{D^i_{n_k},\lambda^i_{n_k}}_{t\wedge\tau^{(L^i+\eps, R^i-\eps)}(\omega)}(\omega) \\
        =&\lim_{\eps \searrow 0 } \lim_{k\to \infty} A^{D^i_{n_k},\lambda^i_{n_k}}_{t\wedge\tau^{(L^i+\eps, R^i-\eps)}(\omega)}(\omega) 
        = \lim_{\eps \searrow 0 }\lim_{k\to \infty} \int_{D^i_{n_k}} l^y_{t\wedge\tau^{(L^i+\eps, R^i-\eps)}(\omega)} (\omega) d \lambda^i_{n_k}  \\
        =&\lim_{\eps \searrow 0 }\lim_{k \to \infty} \int_{D^i}  l^y_{t\wedge\tau^{(L^i+\eps, R^i-\eps)}(\omega)} (\omega) d \Tilde{\lambda}^i_{n_k,k}
        = \lim_{\eps \searrow 0 }\int_{D^i} l^y_{t\wedge\tau^{(L^i+\eps, R^i-\eps)}(\omega)} (\omega) d \lambda^i\\
        =&\lim_{\eps \searrow 0 } A^{D^i,\lambda^i}_{t\wedge\tau^{(L^i+\eps, R^i-\eps)}(\omega)}(\omega)
        = A^{D^i,\lambda^i}_{t}(\omega)
    \end{align*} for all $\omega \in \Omega\setminus N_1$ with $N_1:=N_1'\cup N_1''$.

    \eqref{aux4_2} If $A^{D^i,\lambda^i}_{\tau^i}= \infty$ the claim is trivial, so we are left with the set $ \{A^{D^i,\lambda^i}_{\tau^i}<\infty \}$. First note that $ \{A^{D^i,\lambda^i}_{\tau^i}<\infty \}=  \{A^{D^i,\lambda^i}_{\tau^i} = E_i\}$ by continuity of $[0,\tau^{D^i}) \ni t\mapsto A^{D^i,\lambda^i}_{t}$ and $\tau^{D^i}= \inf\{t \ge 0: A^{D^i,\lambda^i}_t = \infty\}$. Thus $X_{\tau^i} \in D^i$. Next we show $X_{\tau^i}\in \supp(\lambda^i)$. Fix some $\omega \in \{X_{\tau^i}\not\in \supp(\lambda^i)\}$. Now, since $\supp(\lambda^i)$ is closed, there is some $\eps'>0$ such that $X_{\tau^i-t}(\omega)\not \in \supp(\lambda^i)$ for all $t\in [0,\eps']$. This means $\theta_{\tau^i-\eps'}\circ l^y_{\eps'}(\omega)=0$ for all $y\in \supp(\lambda^i)$. This yields
    \begin{align*}
        &E_i(\omega)=A^{D^i,\lambda^i}_{\tau^i(\omega)}(\omega)
        = \int l^y_{\tau^i(\omega)}(\omega)\lambda^i(dy) 
        =  \int l^y_{\tau^i(\omega)-\eps'}(\omega)+ \theta_{\tau^i-\eps'}\circ l^y_{\eps'}(\omega) \lambda^i(dy) \\
        =& \int l^y_{\tau^i(\omega)-\eps'} \lambda^i(dy) + \int \theta_{\tau^i-\eps'}\circ l^y_{\eps'}(\omega) \lambda^i(dy)
        = \int l^y_{\tau^i(\omega)-\eps'} \lambda^i(dy) 
        = A^{D^i,\lambda^i}_{\tau^i(\omega)-\eps'}(\omega),
    \end{align*} which contradicts $\tau^i= \inf\{t \ge 0 :A^{D^i,\lambda^i}_t \ge E_i\}$. Since $y\mapsto\theta_{\tau^i} \circ l^y_{\eps}$ is $\PR_x$-a.s.\ bounded away from $0$ in every neighborhood of $X_{\tau^i}\in\supp(\lambda^i)$ and $\lambda^i$ puts positive weight on a neighborhood of very such point, we get
    \begin{align*}
        A^{D^i,\lambda^i}_{\tau^i+\eps} 
        \ge \int  l^y_{\tau^i+\eps}\lambda^i(dy)  
        =  \int  l^y_{\tau^i} \lambda^i(dy) + \int\inf \theta_{\tau^i} \circ l^y_{\eps} d \lambda^i(dy)
        >\int  l^y_{\tau^i} \lambda^i(dy) 
        = A^{D^i,\lambda^i}_{\tau^i} = E_i
    \end{align*}$\PR_x$-a.s.\ as claimed.

    \eqref{aux4_3} W.l.o.g.\ we show the claim for $\PR_x$-a.a.\ $\{t>\tau^{D^i}\} \cap \{ \eta_{L^i} < \eta_{R^i}\}$. The set $\{t>\tau^{D^i}\} \cap \{ \eta_{L^i} < \eta_{R^i}\}$ can be treated analogously. For every $\eps>0$ we have 
    \begin{align*}
        c_{k}:=\max \left\{\vert L^i_{n_k}-L^i \vert , \frac{1}{\lambda^i_{n_k}((L^i-\eps, L^i+\eps) \cap D_{n_k}^i)} \right\} \stackrel{k \to \infty}{\longrightarrow}0.
    \end{align*} If this were not the case we could find a subsequence $(c_{k_j})_{j \in \N}$ that is bounded away from 0. Now we could repeat the arguments from Lemma \ref{aux2} for the subsequence $(\lambda^i_{n_{k_j}})_{j \in \N}$ in place of the original sequence $(\lambda^i_n)_{n\in \N}$ to reach a contradiction. Keep in mind, that for this argument $L^i$ would remain the same as it only depends on $\mu$ (which is unchanged when passing on to a subsequence), whereas $\Tilde{L}^i$ would have to be replaced by $\limsup_{j\to\infty} L^i_{n_{k_j}}<L^i$.
    
    Now we choose some nullset $N_2'$ such that for all $\omega\in (\{t>\tau^{D^i}\} \cap \{ \eta_{L^i} < \eta_{R^i}\})\setminus N_2'$ there is some $\eps>0$ with $\inf_{y\in (L^i-\eps,L^i+\eps)} \theta_{\tau^{D^i}}\circ l^y_{t-\tau^{D^i}(\omega)}(\omega)>0$. Additionally, we choose a nullset $N_2''$ such that $\eta_{y}(\omega) \to \eta_{L^i}(\omega)$ for all $\omega\in \Omega \setminus N_2''$ and set $_2:= N_2'\cup N_2''$. Now for all $\omega\in(\{t>\tau^{D^i}\} \cap \{ \eta_{L^i} < \eta_{R^i}\})\setminus N_2( \supset \{ \tau^{D^i}=\eta_{L^i}\})$ there is an $\eps >0$ such that
    \begin{align*}
        &\max \left\{  (\tau^{D^i_{n_k}}(\omega)-\tau^{D^i}(\omega))_+  , \frac{1}{\int \theta_{\tau^{D^i}}\circ l^y_{t-\tau^{D^i}(\omega)}(\omega)d \lambda^i_{n_k}}  \right\}\\ 
        \le& \max \left\{ \vert \eta_{L^i_{n_k}}(\omega)-\eta_{L^i}(\omega) \vert , \frac{1}{\inf_{y\in (L^i-\eps,L^i+\eps)} \theta_{\tau^{D^i}}\circ l^y_{t-\tau^{D^i}(\omega)}(\omega) \lambda^i_{n_k}((L^i-\eps,L^i+\eps)\cap D^i_{n_k})} \right\}\\
        &\stackrel{k \to \infty}{\longrightarrow}0.
    \end{align*} This means, given any $C>0$, putting these things together there is some $K\in \N$ such that $(\tau^{D^i_{n_k}}(\omega)-\tau^{D^i}(\omega))_+<t-\tau^{D^i}(\omega)$ (i.e.\ $t\ge \tau^{D^i_{n_k}}(\omega)$) or $\int \theta_{\tau^{D^i}}\circ l^y_{t-\tau^{D^i}(\omega)}(\omega)d \lambda^i_{n_k}\ge C$ for all $k \ge K$. Thus, for the same $\omega$ and $k\ge K$
    \begin{align*}
         &A^{D^i_{n_k},\lambda^i_{n_k}}_t(\omega) 
         = \int l^y_{t}(\omega) \lambda^i_{n_k}(dy) + \infty \ind{ t \ge \tau^{D^i_{n_k}}(\omega)}\\
         =& \int \theta_{\tau^{D^i}}\circ l^y_{t-\tau^{D^i}(\omega)}(\omega)d \lambda^i_{n_k} + \int l^y_{\tau^{D^i(\omega)}} (\omega)\lambda^i_{n_k}(dy) + \infty \ind{ t \ge \tau^{D^i_{n_k}}(\omega)}\\
         \ge&\int l^y_{\tau^{D^i(\omega)}} (\omega)\lambda^i_{n_k}(dy) + \infty \ind{ t \ge \tau^{D^i_{n_k}}(\omega)} \ge C.
    \end{align*} 
    
    \eqref{aux4_4} Let $\eps>0$. By definition, $A^{D^i,\lambda^i}_{(\tau^i(\omega)-\eps)\vee 0}(\omega) < E_i(\omega)$ for all $\omega \in \Omega\setminus N'$ with a $\PR_x$-nullset $N'$. By \eqref{aux4_1}, $A^{D^i_{n_k},\lambda^i_{n_k}}_{\tau^i(\omega)-\eps}<E_i(\omega)$ for all $\omega\in \Omega\setminus(N'\cup N_1)$ and sufficiently large $k\in \N$, so $\tau^i_{n_k}(\omega)>\tau^i(\omega)-\eps$ for these $k$ and $\omega$. By \eqref{aux4_2}, there is a $\PR_x$-nullset $N''$ such that $A^{D^i,\lambda^i}_{\tau^i(\omega)+ \eps}(\omega) > E_i(\omega)$ for all $\omega\in \Omega \setminus N''$. By \eqref{aux4_1} $A^{D^i_{n_k},\lambda^i_{n_k}}_{\tau^i(\omega)-\eps}>E_i(\omega)$ for all $\omega\in \{\tau^i+\eps < \tau^{D^i}\} \setminus(N''\cup N_1)$ and sufficiently large $k\in \N$, so $\tau^i_{n_k}(\omega)\le \tau^i(\omega)+\eps$ for these $k$ and $\omega$. By \eqref{aux4_3} we have the same conclusion for all $\omega \in \{\tau^i+\eps > \tau^{D^i}\} \setminus(N''\cup N_2)$. If $\omega \in \{\tau^i+\eps = \tau^{D^i}\}$ we replace $\eps$ by any $\eps'<\eps$. Sending $\eps \searrow 0$ along any fixed sequence gives $\tau^i_{n_k} \to \tau^i$ outside the union of all the countably many exception sets associated with the chosen sequence.
\end{proof}
    
    \begin{proof}(of Proposition \ref{comp})\\
    As described above, weak compactness holds by Prokhorov's theorem. Relative closeness holds by putting the previous pieces together:
        By Lemma \ref{aux3} \eqref{3} and Lemma \ref{aux4} \eqref{aux4_4} there is an indexing sequence $(n_k)_{k\in\N}$ such that $(\tau^1_{n_k}, \tau^2_{n_k}) \stackrel{k\to\infty}{\longrightarrow}(\tau^1,\tau^2)$ $\PR_x$-a.s. By path continuity of $X$ we obtain $(\tau^1_{n_k},X_{\tau^1_{n_k}},\tau^2_{n_k},X_{\tau^2_{n_k}})  \stackrel{k\to\infty}{\longrightarrow} (\tau^1,X_{\tau^1},\tau^2,X_{\tau^2}) $ $\PR_x$-a.s. In particular, $\PR_x^{(\tau^1_{n_k},X_{\tau^1_{n_k}},\tau^2_{n_k},X_{\tau^2_{n_k}})} \stackrel{k\to\infty}{\longrightarrow} \PR_x^{(\tau^1,X_{\tau^1},\tau^2,X_{\tau^2})}$. By uniqueness of weak limits we infer $\mu =  \PR_x^{(\tau^1,X_{\tau^1},\tau^2,X_{\tau^2})}$.
    \end{proof}

\subsection{$\mathcal{M}^1$-product topology}
    The $\mathcal{M}^1_x$-topology already has good properties. However, since we are interested in Markov-perfect equilibria, the consideration for a fixed $x$ is not sufficient. We must therefore find a suitable notions of convergence that takes into account all possible initial points at once.

\begin{definition}($\mathcal{M}^1$-product topology)\\
    Let $U\subset I$ open. Let $\kappa: \mathcal{Z}^U(E_1)\times \mathcal{Z}^U(E_2) \to \bigtimes_{x\in U\cap \mathbb{Q}} \mathcal{M}^1(([0,\infty]\times \overline{U})^2)$ be given by
    \begin{align*}
        (\tau^{D_1,\lambda_1}, \tau^{D_2,\lambda_2}) \mapsto (\kappa_x(\tau^{D_1,\lambda_1}, \tau^{D_2,\lambda_2}))_{x\in U \cap \mathbb Q}.
    \end{align*} We define the \textit{$\mathcal{M}^1$-product topology} as the pullback topology on $\mathcal{Z}^U(E_1)\times \mathcal{Z}^U(E_2)$ under the mapping $\kappa$.
\end{definition}

The consideration of rational starting points has the technical background that we thus obtain a countable product space. The following two results ensure that this is not a significant limitation. 

\begin{lemma}\label{equicont}
    Let $U \subset I$ open with compact closure $\overline{U}\subset I$ (closure taken in $\R$), $ ((\tau^1_n,\tau^2_n))_{n\in \N}=((\tau^{D^1_n,\lambda^1_n},\tau^{D^2_n,\lambda^2_n}))_{n\in \N}$ a sequence in $\mathcal{Z}^U(E_1)\times \mathcal{Z}^U(E_2)$, $(\tau^1,\tau^2)=(\tau^{D^1,\lambda^1},\tau^{D^2,\lambda^2})\in\mathcal{Z}^U(E_1)\times \mathcal{Z}^U(E_2)$ such that $(\tau^1_n,\tau^2_n)\stackrel{n \to \infty}{\longrightarrow}(\tau^1,\tau^2) $ in the $\mathcal{M}^1$-product topology and let $f:[0,\infty]\times I\times[0,\infty]\times I\to\R$ be bounded and uniformly continuous. Then, the sequence $\varphi = (\varphi_n)_{n\in \N}$ of mappings 
    \begin{align*}
       \varphi_n:\,U\to \R,\;\; x\mapsto \Ex_x[f(\tau^1_n,X_{\tau^1_n},\tau^2_n,X_{\tau^2_n})] = \int f d \PR_x^{(\tau^1_n,X_{\tau^1_n},\tau^2_n,X_{\tau^2_n})}
    \end{align*} is equicontinuous.
\end{lemma}

\begin{proof}
    Let $x\in U$, $\eps>0$ and $\delta>0$ such that $\vert f(t_1,x_1,t_2,x_2)- f(s_1,y_1,s_2,y_2)\vert <\eps$ for all $(t_1,x_1,t_2,x_2),(s_1,y_1,s_2,y_2)\in[0,\infty]\times I\times[0,\infty]\times I $ with $\vert t_1-s_1\vert,\vert t_2-s_2\vert,\vert x_1-y_1\vert,\vert x_2-y_2 \vert\le\delta$. Set $B=(x-\frac{\delta}{2},x+\frac{\delta}{2})$ and $S:=\sup \vert f\vert $.

    By our standing assumptions on $I$, the point $x$ is non-singular in the sense of \cite[Section 3.4.]{ito1974diffusion}. By \cite[Section 4.4.]{ito1974diffusion} $\PR_y( \eta_x<\infty) \ge\PR_y( \eta_x< \tau^B)\stackrel{y\to x}{\longrightarrow}1$. Thus by \cite[Section 3.3.(10c)]{ito1974diffusion} $\PR_y( \eta_x< \delta)\stackrel{y\to x}{\longrightarrow}1$. For all $\tau^{D,\lambda}\in \mathcal{Z}^U$ with $x\not \in D$ this yields
    \begin{align}
        \PR_y(\tau^{D,\lambda}\le \tau^B\wedge \delta)\ge \PR_y(\eta_x< \tau^B\wedge \delta) \stackrel{y\to x}{\longrightarrow}1.\label{hr1}
    \end{align}
    In case $x\not\in D^1\cup D^2$ we infer
    \begin{align*}
        &\vert \Ex_x[f(\tau^1,X_{\tau^1},\tau^2,X_{\tau^2})] -\Ex_y[f(\tau^1,X_{\tau^1},\tau^2,X_{\tau^2})] \vert\\
        \le& \vert f(0,x,0,x) -\Ex_y[ \ind{\tau^1\vee \tau^2 \le \tau^B\wedge \delta} f(\tau^1,X_{\tau^1},\tau^2,X_{\tau^2})] \vert + S \PR_y(\tau^1\vee \tau^2 > \tau^B\wedge \delta)\\
        \le& \vert f(0,x,0,x)\vert (1 -\PR_y(\tau^1\vee \tau^2 > \tau^B\wedge \delta)) + \eps + S \PR_y(\tau^1\vee \tau^2 > \tau^B\wedge \delta) \stackrel{y\to x}{\longrightarrow}\eps
    \end{align*} independent from $\tau^1,\tau^2$. I.e.\ provided $x\not\in  D^1_n\cup D^2_n$ the same estimate can be repeated for $\tau^1_n,\tau^2_n$ for every $n$.
    
    Now assume $x\in D^1\cap D^2$. By Lemma \ref{aux3} \eqref{1b} we may assume that there is neighborhood $B'\subset D^1\cap D^2$ of $x$ such that $B'\subset \bigcap_{n\in \N}( D^1_n\cap D^2_n)$ and $\Lambda:=\sup_{i\in \{1,2\}, n\in \N} \lambda^i_n(B')<\infty$. W.l.o.g.\ let $B'=B$. Now for $i=1,2$
    \begin{align*}
        A^{\tau^i_n}_{\eta_x\wedge \tau^B \wedge \delta} 
        =A^{D^i_n,\lambda^i_n}_{\eta_x\wedge \tau^B \wedge \delta} 
        = \int l^{y}_{\eta_x\wedge \tau^B \wedge \delta} d\lambda^i_n(y) 
        \le \int \eta_x\wedge \tau^B \wedge \delta d\lambda^i_n(y) 
        \le \Lambda (\eta_x\wedge \tau^B \wedge \delta).
    \end{align*} and so
    \begin{align}
        &\PR_y(\tau^1_n\wedge \tau^2_n > \eta_x\wedge \tau^B\wedge \delta) 
        = \PR_y (A^{\tau^1_n}_{\eta_x\wedge \tau^B \wedge \delta} < E_1 \,\text{ and }\, A^{\tau^2_n}_{\eta_x\wedge \tau^B \wedge \delta} <E_2)\notag\\
        \ge & \PR_y(\Lambda (\eta_x\wedge \tau^B \wedge \delta) < E_1\wedge E_2)
        = \Ex_y(e^{-2\Lambda (\eta_x\wedge \tau^B \wedge \delta)})\stackrel{y\to x}{\to}1\label{hr2}
    \end{align} independent from $n$. To abbreviate, let $\sigma:=  \eta_x\wedge \tau^B\wedge\delta$. Using that $\tau^1_n$ and $\tau^2_n$ are Markovian, we obtain
    \begin{align*}
        &\vert \Ex_x[f(\tau^1_n,X_{\tau^1_n},\tau^2_n,X_{\tau^2_n})]- \Ex_y[f(\tau^1_n,X_{\tau^1_n},\tau^2_n,X_{\tau^2_n})] \vert \\
        \le& \vert \Ex_x[f(\tau^1_n,X_{\tau^1_n},\tau^2_n,X_{\tau^2_n})] -\Ex_y[\ind{\tau^1_n\wedge \tau^2_n > \sigma}f(\tau^1_n,X_{\tau^1_n},\tau^2_n,X_{\tau^2_n}) ]\vert + S\Ex_y(\ind{\tau^1_n\wedge \tau^2_n \le \sigma})\\
        \le& \vert \Ex_x[f(\tau^1_n,X_{\tau^1_n},\tau^2_n,X_{\tau^2_n})] - \Ex_y[\ind{\tau^1_n\wedge \tau^2_n > \sigma}f(\theta_\sigma\circ\tau^1_n+\sigma,X_{\theta_\sigma\circ\tau^1_n+\sigma},\theta_\sigma\circ\tau^2_n+\sigma,X_{\theta_\sigma\circ\tau^2_n+\sigma}) ]\vert \\
        &+ S (1-\Ex_y(e^{-2\Lambda\sigma}))\\
        \le& \vert \Ex_x[f(\tau^1_n,X_{\tau^1_n},\tau^2_n,X_{\tau^2_n})] - \Ex_y[\ind{\tau^1_n\wedge \tau^2_n > \sigma}f(\theta_\sigma\circ\tau^1_n,X_{\theta_\sigma\circ\tau^1_n+\sigma},\theta_\sigma\circ\tau^2_n,X_{\theta_\sigma\circ\tau^2_n+\sigma}) ]\vert + \eps \\
        &+S (1-\Ex_y(e^{-2\Lambda\sigma}))\\
        =&\vert \Ex_x[f(\tau^1_n,X_{\tau^1_n},\tau^2_n,X_{\tau^2_n})] - \Ex_y[\ind{\tau^1_n\wedge \tau^2_n > \sigma}\Ex_{X_\sigma}[f(\tau^1_n,X_{\tau^1_n},\tau^2_n,X_{\tau^2_n})]]\vert + \eps +S (1-\Ex_y(e^{-2\Lambda\sigma}))\\
        =&\vert \Ex_x[f(\tau^1_n,X_{\tau^1_n},\tau^2_n,X_{\tau^2_n})] - \Ex_y[\ind{\tau^1_n\wedge \tau^2_n > \sigma}\ind{\eta_x=\sigma}]\Ex_{x}[f(\tau^1_n,X_{\tau^1_n},\tau^2_n,X_{\tau^2_n})]\vert + S\PR_y(\eta_x>\sigma)+ \eps \\
        &+S (1-\Ex_y(e^{-2\Lambda\sigma}))\\
        \le& \vert \Ex_x[f(\tau^1_n,X_{\tau^1_n},\tau^2_n,X_{\tau^2_n})]\vert (1-\Ex_y[\ind{\eta_x=\sigma}e^{-2\Lambda\sigma}]) + S\PR_y(\eta_x>\sigma)+ \eps +S (1-\Ex_y(e^{-2\Lambda\sigma}))\\
        &\stackrel{y\to x}{\to}\eps
    \end{align*} independent from $n$ by \eqref{hr1} and \eqref{hr2}.
For the cases $x\in D_1\setminus D^2$ and $x\in D^2\setminus D^1$ we can make essentially the same estimate combining the arguments from the cases $x\not \in D^1\cup D^2$ and $x\in D^1\cap D^2$. Sending $\eps \searrow 0$ we infer the claim.
\end{proof}

\begin{proposition}\label{weakconv}
    Let $U \subset I$ open with compact closure $\overline{U}\subset I$ (closure taken in $\R$), $((\tau^1_n,\tau^2_n))_{n\in \N}$ a sequence in $ \mathcal{Z}^U(E_1)\times \mathcal{Z}^U(E_2)$ and $(\tau^1,\tau^2)\in\mathcal{Z}^U(E_1)\times \mathcal{Z}^U(E_2)$ such that $(\tau^1_n,\tau^2_n)\stackrel{n \to \infty}{\longrightarrow}(\tau^1,\tau^2) $ in the $\mathcal{M}^1$-product topology. Then, 
\begin{align*}
    \PR_x^{(\tau^1_n,X_{\tau^1_n},\tau^2_n,X_{\tau^2_n})}\stackrel{n \to \infty}{\longrightarrow} \PR_x^{(\tau^1,X_{\tau^1},\tau^2,X_{\tau^2})}
\end{align*} weakly for all $x\in I$.
\end{proposition}

\begin{proof}
    For $x \in U\cap \mathbb Q$, the claim is immediate by definition. For $x\in I \setminus U$ 
    \begin{align*}
        \PR_x^{(\tau^1_n,X_{\tau^1_n},\tau^2_n,X_{\tau^2_n})} =\delta_0\otimes \delta_x\otimes \delta_0 \otimes \delta_x= \PR_x^{(\tau^1,X_{\tau^1},\tau^2,X_{\tau^2})}
    \end{align*} so the claim is once again trivial. Let $x\in U\setminus \mathbb Q$ and $f:[0,\infty]\times I\times[0,\infty]\times I\to\R$ bounded and uniformly continuous. For all $y\in U\cap \mathbb Q$ we have
    \begin{align*}
        &\bigg\vert \int f d\PR_x^{(\tau^1_n,X_{\tau^1_n},\tau^2_n,X_{\tau^2_n})} - \int f d \PR_x^{(\tau^1,X_{\tau^1},\tau^2,X_{\tau^2})}\bigg\vert \\
        \le &\bigg\vert \int f d\PR_x^{(\tau^1_n,X_{\tau^1_n},\tau^2_n,X_{\tau^2_n})} - \int f d\PR_y^{(\tau^1_n,X_{\tau^1_n},\tau^2_n,X_{\tau^2_n})} \bigg\vert \\
        &+ \bigg\vert  \int f d\PR_y^{(\tau^1_n,X_{\tau^1_n},\tau^2_n,X_{\tau^2_n})} - \int f d \PR_y^{(\tau^1,X_{\tau^1},\tau^2,X_{\tau^2})}\bigg\vert \\
        &+ \bigg\vert \int f d \PR_y^{(\tau^1,X_{\tau^1},\tau^2,X_{\tau^2})}- \int f d \PR_x^{(\tau^1,X_{\tau^1},\tau^2,X_{\tau^2})}\bigg\vert.
    \end{align*} By the above the middle term on the right hand side converges to 0 for $n\to \infty$. By Lemma \ref{equicont} the first and third term converge to 0 for $y\to x$ independent from $n$. By \cite[Theorem 1.1.1]{stroock2007multidimensional} this proves the claim.
\end{proof}

We come to the main result of this section.
\begin{theorem}\label{comp2}
    Let $U \subset I$ open with compact closure $\overline{U}\subset I$ (closure taken in $\R$). Then, $\mathcal{Z}^U(E_1)\times \mathcal{Z}^U(E_2)$ equipped with the $\mathcal{M}^1$-product topology is a compact topological space.
\end{theorem}

\begin{remark}
More explicitly, the previous theorem states that the set $$\kappa(\mathcal{Z}^U(E_1)\times \mathcal{Z}^U(E_2)) \subset \bigtimes_{x\in U\cap \mathbb{Q}}\mathcal{M}^1(([0,\infty]\times \overline{U})^2)$$ is compact in the product topology on $\bigtimes_{x\in U\cap \mathbb{Q}}\mathcal{M}^1(([0,\infty]\times \overline{U})^2)$.
Note that this result is in fact not a direct consequence of Tychonoff's theorem. The point is that, for a given $x\in I$, the $\mathcal{M}^1_x$ topology identifies all Markovian randomized times $\tau^{D,\lambda},\tau^{D',\lambda'}$ with the property that the connected components $D_x$ and $D_x'$ of $x$ in $D$, $D'$ respectively, coincide, as well as the generating measures restricted to these connected components, i.e.\ $\lambda\vert_{D_x}= \lambda'\vert_{D_x'}$. Thus, although every sequence $(\kappa(\tau^1_n,\tau^2_n))_{n \in \N}$ has a subsequence that converges to a family of measures $(\mu_x)_{x\in U \cap \mathbb Q}$, it is not immediately obvious that there is a single pair of Markovian randomized times $(\tau^1,\tau^2)$ such that $\kappa(\tau^1,\tau^2)=(\mu_x)_{x\in U \cap \mathbb Q}$. 
\end{remark}

\begin{proof}

    In the proof of Proposition \ref{comp}, we have already shown that $\mathcal{M}^1_x(([0,\infty]\times \overline{U})^2)$, $x\in I,$ is  a compact space. Thus, by Tychonoff's theorem the product space $\bigtimes_{x\in U\cap \mathbb{Q}}\mathcal{M}^1_x(([0,\infty]\times \overline{U})^2)$ is compact as well. This yields relative compactness of $\kappa(\mathcal{Z}^U(E_1)\times \mathcal{Z}^U(E_2))$.

    As a countable product of metric spaces, $\bigtimes_{x\in U\cap \mathbb{Q}}\mathcal{M}^1_x(([0,\infty]\times \overline{U})^2)$ is metric, so it suffices to show sequential closedness in order to finish the proof of compactness. To this end, let $((\tau^{D^1_n,\lambda^1_n},\tau^{D^2_n,\lambda^2_n}))_{n\in \N} \in (\mathcal{Z}^U(E_1)\times \mathcal{Z}^U(E_2))^{\N}$ be a sequence of tuples of Markovian randomized times such that $\kappa(\tau^{D^1_n,\lambda^1_n},\tau^{D^2_n,\lambda^2_n})$ converges in the product topology, i.e.\ pointwise in $x\in U\cap \mathbb Q$, to a family of measures $(\mu_x)_{x\in U \cap \mathbb{Q}} \in \bigtimes_{x\in U\cap \mathbb{Q}} \mathcal{M}^1_x(([0,\infty]\times \overline{U})^2) $. We have to find $(\tau^1,\tau^2) :=(\tau^{D^1,\lambda^1},\tau^{D^2,\lambda^2})\in\mathcal{Z}^U(E_1)\times \mathcal{Z}^U(E_2)$ such that $\kappa(\tau^1,\tau^2) = (\mu_x)_{x\in U \cap \mathbb{Q}}$. To abbreviate, we set $(\tau^1_n,\tau^2_n):=(\tau^{D^1_n,\lambda^1_n},\tau^{D^2_n,\lambda^2_n})$ for each $n\in\N$.

    Given $x\in I \cap \mathbb Q$, we write 
    \begin{align*}
        &L^i_x := L^i, 
        && R^i_x:= R^i,\\
        &\Tilde{L}^i_x:= \Tilde{L}^i,
        &&\Tilde{R}^i:= \Tilde{R}^i,\\
        &D^i_x:=(L^i_x,R^i_x),
        &&\lambda^i_x:=\lambda^i\in \rm(D^i_x),
    \end{align*} where $L^i,R^i, \Tilde{L}^i,\Tilde{R}^i$ are the objects defined in Lemma \ref{aux1} corresponding to $x$ and $\lambda^i$ being the corresponding measure constructed in Lemma \ref{aux3}.

    By definition
    \begin{align}
        \PR_x^{ (\tau^{D^1_x,\lambda^1_x},X_{\tau^{D^1_x,\lambda^1_x}},\tau^{D^2_x,\lambda^2_x},X_{\tau^{D^2_x,\lambda^2_x}})} =\PR_x^{(\tau^{D^1,\lambda^1},X_{\tau^{D^1,\lambda^1}},\tau^{D^2,\lambda^2},X_{\tau^{D^2,\lambda^2}})} \label{consistency2}
    \end{align}
    for all open $D^1_x\subset D^1$, $D^2_x\subset D^2$, $\lambda^1\in \rm(D^1), \lambda^2 \in \rm(D^2)$ such that $L^1_x,R^1_x \not \in D^1$, $L^2,R^2\not \in D^2$ and $\lambda^1\vert_{D^1_x} = \lambda^1_x$, $\lambda^2\vert_{D^2_x} = \lambda^1_x$. 

    We now define our candidate for the limit point via 
    \begin{align}
        &D^i:= \bigcup_{x \in U \cap \mathbb Q} D^i_x, &&\lambda^i\vert_{D^i_x}:= \lambda^i_x, \quad i\in\{1,2\}.\label{go}
    \end{align} Next we will now show that $D^i_x,\lambda^i_x, i\in \{1,2\}$ do not depend on the limiting distribution $\mu_x$, but only on the characteristics of the sequence $((\tau^1_n,\tau^2_n))_{n\in \N}$. For that we define
    \begin{align*}
        &\hat L^i_x:= \sup \left\{y \le x : \lim_{n \to \infty}\lambda^i_n((y-\eps, y+ \eps)) = \infty\,\,\forall \eps >0 \right\},\\
        &\hat R^i_x:= \sup \left\{y \ge x : \lim_{n \to \infty}\lambda^i_n((y-\eps, y+ \eps)) = \infty\,\,\forall \eps >0 \right\}
    \end{align*}and show 
    \begin{align}
        L^i_x = \max\{ \Tilde{L}^i_x, \hat L^i_x\}, \quad R^i_x = \min\{\Tilde{R}^i_x, \hat R^i_x\}. \label{LRgl}
    \end{align} W.l.o.g.\ we only argue for the case of $L^i_x$, the one for $R^i_x$ being analogous.
    $L^i_x \ge \Tilde{L}^i_x$ has been shown in Lemma \ref{aux1} \eqref{consistency1}. By Lemma \ref{aux3} \eqref{3} and \eqref{1b}, using the notation from the lemma, we get
    \begin{align*}
        \lambda^i_x((y-\eps,y+\eps)) \le \liminf_{k \to \infty} \Tilde{\lambda}^i_{n_k,k}((y-\eps,y+\eps)) \le \sup_{n \ge N_m}\Tilde{\lambda}^i_{n,m}((L^i_x,R^i_x))<\infty
    \end{align*} whenever $(y-\eps,y+\eps)\subset (L^i_x,R^i_x)$. This implies $L^i_x \ge \hat L^i_x$. Finally if $L^i_x>\Tilde{L}^i_x$, then $\lim_{n \to \infty} \lambda^i_n([L^i_x-\eps,L^i_x+\eps])= \infty$ for all $\eps>0$ by Lemma \ref{aux2}, so $L^i_x = \hat L^i_x$. With that \eqref{LRgl} is shown.

    By definition the right hand side of \eqref{LRgl} only depends on the characteristics of the sequence in the sense that
    \begin{align}
        L^i_x=\max\{\Tilde{L}^i_x, \hat L^i_x\} = \sup\left\{ y \le x : \liminf_{n \to \infty} \infty \ind{D^i_n \cap (y-\eps,y+\eps) \neq \varnothing} + \lambda^i_n((y-\eps,y+\eps)) = \infty \,\, \forall \eps >0\right\}.\label{6fin}
    \end{align}
    Repeating the argument for $R^i_x$ we infer 
    \begin{align*}
        R^i_x=\min\{\Tilde{R}^i_x, \hat R^i_x\} = \inf\left\{ y \ge x : \liminf_{n \to \infty} \infty \ind{D^i_n \cap (y-\eps,y+\eps) \neq \varnothing} + \lambda^i_n((y-\eps,y+\eps)) = \infty \,\, \forall \eps >0\right\}.
    \end{align*} Together that yields
    \begin{align*}
        y\in (L^i_x,R^i_x) &\Longrightarrow L^i_y = L^i_x, R^i_y = R^i_x,\\
        y = L^i_x &\Longrightarrow L^i_y = L^i_x, R^i_y = L^i_x,\\
        y = R^i_x &\Longrightarrow L^i_y = R^i_x, R^i_y = R^i_x.
    \end{align*} Additionally note that $\lambda^i_x$ does only depend on $D^i_x= (L^i_x, R^i_x)$. This shows that $D^i,\lambda^i$ are well defined by \eqref{go} and that this choice satisfies \eqref{consistency2}. By Proposition \ref{comp} and Lemma \ref{aux3} 
    \begin{align*}
        \kappa_x(\tau^1_n,\tau^2_n) 
        \stackrel{n \to \infty}{\longrightarrow}\PR_x^{ (\tau^{D^1_x,\lambda^1_x},X_{\tau^{D^1_x,\lambda^1_x}},\tau^{D^2_x,\lambda^2_x},X_{\tau^{D^2_x,\lambda^2_x}})} 
        \left(=\PR_x^{(\tau^{D^1,\lambda^1},X_{\tau^{D^1,\lambda^1}},\tau^{D^2,\lambda^2},X_{\tau^{D^2,\lambda^2}})}\right)
    \end{align*} weakly for all $x\in U$, i.e.\ 
    \begin{align*}
        \kappa(\tau^1_n,\tau^2_n) \stackrel{n \to \infty}{\longrightarrow}\PR_x^{(\tau^{D^1,\lambda^1},X_{\tau^{D^1,\lambda^1}},\tau^{D^2,\lambda^2},X_{\tau^{D^2,\lambda^2}})}
    \end{align*} in $\bigtimes_{x\in U\cap \mathbb{Q}}\mathcal{M}^1(([0,\infty]\times \overline{U})^2)$.  
\end{proof}

The rest of this section is dedicated to establishing some technical properties of sequences that converge in the $\mathcal{M}^1$-product topology. The construction of the limit points in Lemma \ref{aux3} \eqref{3} and the proof of Theorem \ref{comp2} evolves around the convergence of the connected components of the continuation region and the convergence of the stopping rate measures on these along subsequences. Together, the results can be interpreted as meaning that for any sequence that converges in the $\mathcal{M}^1$-product topology, it is possible to find a subsequence such that the continuation regions and the stopping rate measures converge in a suitable sense. We first establish that the continuation region of the $\mathcal{M}^1$-product limit is unique. In particular, it is not affected by passing over to subsequences, since the $\mathcal{M}^1$-product topology is metrizable. 

\begin{lemma}\label{lemma:eineUmg}
    Let $U \subset I$ open with compact closure $\overline{U}\subset I$ (closure taken in $\R$) and $(\tau^{D^1,\lambda^1},\tau^{D^2,\lambda^2})$, $(\tau^{\tilde D^1,\tilde \lambda^1},\tau^{\tilde D^2,\tilde \lambda^2})\in \mathcal{Z}^U(E_1)\times \mathcal{Z}^U(E_2)$ with $D^i\neq \tilde D^i$ for some $i\in \{1,2\}$. Then, there are disjoint open neighborhoods $V$ of $(\tau^{D^1,\lambda^1},\tau^{D^2,\lambda^2})$ and $\tilde V$ of $(\tau^{\tilde D^1,\tilde \lambda^1},\tau^{\tilde D^2,\tilde \lambda^2})$ in the $\mathcal{M}^1$-product topology.
\end{lemma}

\begin{proof}
    W.l.o.g.\ assume $D^1 \neq \tilde D^1$. Write $D^1_x$ ($\tilde D^1_x$) to denote the connected component of $D^1$ ($\tilde D^1$) that contains $x$, relying on the convention $D^1_x:= \varnothing$ ($\tilde D^1_x:=\varnothing$) for $x\not \in D^1$ ($x\not \in \tilde D^1$). By assumption, there is some $x\in U\cap \mathbb{Q}$ such that $D^1_x \neq \tilde D^1_x$. $D^1_x$, $\tilde D^1_x$ are open intervals, so we set $D^1_x=(l,r)$, $\tilde D^1_x=(\tilde l,\tilde r)$ with $l=x=r$ ($\tilde l = x =\tilde r$) if $D^1_x=\varnothing$ ($\tilde D^1_x=\varnothing$). W.l.o.g.\ we will only argue for the case $l<\tilde l$. Let $f:I\to \R$ be bounded, continuous and non-negative with $f(y)\ge 1$ for all $y \le \frac{l+ \tilde l}{2}$ and $f(y)=0$ for all $y\ge \tilde l$. We extend $f$ to a function $\overline{f}: [0,\infty]\times I \times [0,\infty]\times I\to \R$ by setting $\overline{f}(t_1,y_1,t_2,y_2):= f(y_1)$ for all $(t_1,y_1,t_2,y_2)\in [0,\infty]\times I \times [0,\infty]\times I$. To compress the notation we set $\sigma_1:=\eta_{\frac{3l + \tilde l }{4}}$ and $\sigma_2:=\eta_{\frac{l + \tilde l}{2}}$. Now, along the lines of the proof of Lemma \ref{aux2},
    \begin{align*}
        &\int \overline{f} d \PR_x^{(\tau^{D^1,\lambda^1}, X_{\tau^{D^1,\lambda^1}} , \tau^{D^2,\lambda^2},X_{\tau^{D^2,\lambda^2}})} 
        = \int f d \PR_x^{X_{\tau^{D^1,\lambda^1}}} 
        \ge \Ex_x \left[\ind{X_{\tau^{D^1,\lambda^1}} \le  \frac{l+ \tilde l}{2}}\right]\\
        \ge& \Ex_x\left[ \ind{\sigma_1< \tau^{D^1,\lambda^1} } \ind{\theta_{\sigma_1} \circ \eta_l < \theta_{\sigma_1} \circ \sigma_2 } \right]
        = \PR_{\frac{3l + \tilde l }{4}}(\eta_l < \sigma_2) ] \Ex_x\left[e^{-A^{D^1,\lambda^1}_{\sigma_1 }}\right]=:c>0.
    \end{align*} On the other hand since $\supp \PR_x^{X_{\tau^{\tilde D^1,\tilde \lambda^1}}} \subset \{f=0\}$
    \begin{align*}
        \int \overline{f} d \PR_x^{(\tau^{\tilde D^1,\tilde \lambda^1}, X_{\tau^{\tilde D^1,\tilde \lambda^1}} , \tau^{\tilde D^2,\tilde \lambda^2},X_{\tau^{\tilde D^2,\tilde \lambda^2}})} 
        = \int f d \PR_x^{X_{\tau^{\tilde D^1,\tilde \lambda^1}}}  
        = 0.
    \end{align*} Thus the desired neighborhoods are 
    \begin{align*}
        V&:= \left\{(\tau^{\hat D^1,\hat \lambda^1}, \tau^{\hat D^2,\hat \lambda^2}) \in \mathcal{Z}^U(E_1)\times \mathcal{Z}^U(E_2): \int \overline{f} d \PR_x^{(\tau^{\hat D^1,\hat \lambda^1}, X_{\tau^{\hat D^1,\hat \lambda^1}} , \tau^{\hat D^2,\hat \lambda^2},X_{\tau^{D^2,\lambda^2}})} >\frac{c}{2}\right\}\quad \text{and} \\
        \tilde V&:= \left\{(\tau^{\hat D^1,\hat \lambda^1}, \tau^{\hat D^2,\hat \lambda^2}) \in \mathcal{Z}^U(E_1)\times \mathcal{Z}^U(E_2): \int \overline{f} d \PR_x^{(\tau^{\hat D^1,\hat \lambda^1}, X_{\tau^{\hat D^1,\hat \lambda^1}} , \tau^{\hat D^2,\hat \lambda^2},X_{\tau^{\hat D^2,\hat \lambda^2}})} <\frac{c}{2}\right\}.
    \end{align*}
    
\end{proof}

The following corollary captures the connection between the continuation regions and stopping rates of a sequence and the continuation region of its limit in the $\mathcal{M}^1$-product topology.

\begin{corollary}\label{corollary:LR}
    Let $U \subset I$ open with compact closure $\overline{U}\subset I$ (closure taken in $\R$), $((\tau^1_n,\tau^2_n))_{n\in \N}=((\tau^{D^1_n,\lambda^1_n},\tau^{D^2_n,\lambda^2_n}))_{n\in \N}$ a sequence in $ \mathcal{Z}^U(E_1)\times \mathcal{Z}^U(E_2)$ and $(\tau^1,\tau^2)=(\tau^{\tilde D^1,\tilde \lambda^1},\tau^{\tilde D^2,\tilde \lambda^2})\in\mathcal{Z}^U(E_1)\times \mathcal{Z}^U(E_2)$ such that $(\tau^1_n,\tau^2_n)\stackrel{n \to \infty}{\longrightarrow}(\tau^1,\tau^2) $ in the $\mathcal{M}^1$-product topology. Then for each $x\in \tilde D^i$, $i=1,2$, the corresponding connected component $ \tilde D^i_x$ of $\tilde D^i$ is of the form $\tilde D^i_x=(L^i_x,R^i_x)$ with $L^i_x$ and $R^i_x$ given by \eqref{LRgl}.
\end{corollary}

\begin{proof}
    In the proof of Theorem \ref{comp2} we show that $((\tau^1_n,\tau^2_n))_{n\in \N}$ converges to a limit $(\tau^{D^1,\lambda^1},\tau^{D^2,\lambda^2})$ with $D^i_x=(L^i_x,R^i_x)$ with $L^i_x$ and $R^i_x$ given by \eqref{LRgl}, $i=1,2$. By Lemma \ref{lemma:eineUmg}, the continuation regions of the limit of a sequence that converges in the $\mathcal{M}^1$-product topology are uniquely determined. Thus, $D^i=\tilde D^i$ for $i=1,2$ and the claim follows.
\end{proof}

The next lemma deals with the construction of a subsequence with vaguely convergent stopping rates.

\begin{lemma}\label{lemma:ratesconverge}
    Let $U \subset I$ open with compact closure $\overline{U}\subset I$ (closure taken in $\R$), $((\tau^1_n,\tau^2_n))_{n\in \N}=((\tau^{D^1_n,\lambda^1_n},\tau^{D^2_n,\lambda^2_n}))_{n\in \N}$ a sequence in $ \mathcal{Z}^U(E_1)\times \mathcal{Z}^U(E_2)$ and $(\tau^1,\tau^2)=(\tau^{D^1,\lambda^1},\tau^{D^2,\lambda^2})\in\mathcal{Z}^U(E_1)\times \mathcal{Z}^U(E_2)$ such that $(\tau^1_n,\tau^2_n)\stackrel{n \to \infty}{\longrightarrow}(\tau^1,\tau^2) $ in the $\mathcal{M}^1$-product topology. Then there is a subsequence $((\tau^1_{n_k},\tau^2_{n_k}))_{k\in \N}$ and $(\tilde \tau^1,\tilde \tau^2)=(\tau^{\tilde D^1,\tilde \lambda^1},\tau^{\tilde D^2,\tilde \lambda^2})\in \mathcal{Z}^U(E_1)\times \mathcal{Z}^U(E_2)$ with $\PR_x^{(\tilde \tau^1,X_{\tilde\tau^1},\tilde \tau^2, X_{\tilde \tau^2})}=\PR_x^{( \tau^1,X_{\tau^1},\tau^2, X_{\tau^2})}$ for all $x\in U\cap \mathbb{Q}$ such that
    \begin{align*}
        \iota_1(\tau^i_{n_k}) \stackrel{k \to \infty}{\longrightarrow} \iota_1(\tilde \tau^i)
    \end{align*} vaguely on each connected components of $D^i (=\tilde D^i)$ for all $i\in \{1,2\}$.
\end{lemma}

\begin{proof}
    In the proof of Theorem \ref{comp2} we construct $\tilde D^1, \tilde \lambda^1, \tilde D^2, \tilde \lambda^2$ such that $\PR_x^{(\tilde \tau^1,X_{\tilde\tau^1},\tilde \tau^2, X_{\tilde \tau^2})}=\PR_x^{( \tau^1,X_{\tau^1},\tau^2, X_{\tau^2})}$ for all $x\in U\cap \mathbb{Q}$ for $(\tilde \tau^1,\tilde \tau^2):=(\tau^{\tilde D^1,\tilde \lambda^1},\tau^{\tilde D^2,\tilde \lambda^2})$.
    Lemma \ref{lemma:eineUmg} shows that $\tilde D^1=D^1$ and $\tilde D^2= D^2$. For each $i\in \{1,2\}$ let $D^i=\bigcup_{n\in \N} (l^i_n,r^i_n)$ with disjoint, nonempty intervals $(l^i_n,r^i_n)$, $n\in \N$. 
    By Corollary \ref{corollary:LR} we have $(l^i_n,r^i_n)=(L^i_x,R^i_x)$ with $L^i_x,R^i_x$ given by \eqref{LRgl} for all $n\in \N$, both $i\in \{1,2\}$ and any $x\in (l^i_n,r^i_n)$. 

    Fix $i\in \{1,2\}$, $n\in \N$, and $x\in (l^i_n,r^i_n)$. The construction of $\tilde \lambda^i$ in Lemma \ref{aux3} \eqref{3} provides a subsequence $((\tau^1_{n_k},\tau^2_{n_k}))_{k\in \N}$ such that $\tilde \lambda^i_{n_k,k} \stackrel{k \to \infty}{\longrightarrow} \lambda^i$ vaguely on $(l^i_n,r^i_n)=(L^i_x,R^i_x)$, where $\tilde \lambda^i_{n_k,k}$ extends $\inda{[L^i_x+\delta_k,R^i_x-\delta_k]}\lambda^i_{n_k}$ by 0 to $(L^i_x,R^i_x)$ and $(\delta_k)_{k\in \N}\in (0,\infty)^{\N}$ is a sequence such that $\delta_k \searrow 0$ for $k\to \infty$.
    
    Consider an arbitrary test function $\varphi: (l^i_n,r^i_n)\to \R$ which is $\ccc^\infty$ with compact support in $(l^i_n,r^i_n)$, i.e., $\dist(\supp \varphi, \partial(l^i_n,r^i_n))>0 $, where $\dist(A,B):=\inf\{\vert y-z\vert : y\in A,z\in B\}$ for $A,B\subset I$. For all $k$ with $\delta_k<\dist(\supp \varphi, \partial(l^i_n,r^i_n))$ we have 
    \begin{align*}
        \int \varphi d\iota_1(\tau^i_{n_k})=\int \varphi d\lambda^i_{n_k}=\int \varphi d\tilde\lambda^i_{n_k,k} \stackrel{k \to \infty}{\longrightarrow} \int \varphi d\tilde \lambda^i
    \end{align*} i.e.\ by definition of vague convergence $\iota_1(\tau^i_{n_k})\stackrel{k \to \infty}{\longrightarrow}\tilde \lambda^i=\iota_1(\tilde \tau^i)$ vaguely on $(l^i_n,r^i_n)$.
    
    We start by choosing the subsequence with this property corresponding to $i=1$, $n=1$. Of that subsequence, we choose another subsequence that has the property for $i=2$, $n=1$. Then of the second subsequence we choose yet another subsequence that has the property for $i=1$, $n=2$. Proceeding in the same fashion, we construct a sequence of nested subsequences. Now, the diagonal sequence has the property for all $i=1,2$ and all $n\in \N$ as desired.   
\end{proof}

\section{Existence of Markovian equilibria in discretized games} \label{secdisc}

In order to prove our first main Theorem \ref{absbdryexists}, we introduce a discretized versions of the Dynkin game under consideration. We start by establishing existence of equilibria in the discretized game by applying a fixed point theorem based on the $\mathcal{M}$-topology. 

The central Assumption \eqref{A} imposes to the war of attrition type of the game and is necessary to work out certain properties of the best response problem, see Lemma \ref{indiff}. In order to determine the $\mathcal{M}$-continuity points of $J^1,J^2$ in Lemma \ref{MC} we use the structural Assumptions \eqref{B1} or \eqref{C1}, \eqref{C2}, \eqref{C3}. In that lemma Assumption \eqref{B1} or \eqref{C1} clarify the boundary behavior of $X$, since other behaviors would require different treatment. In case of \eqref{C1} we need to add the technical assumptions \eqref{C2} and \eqref{C3} in order pass over to time $\infty$ or the boundary of the state space respectively.

In the proof of Lemma \ref{lemma_closed} the Assumption \eqref{B2} complements the requirements \eqref{nonstop}, \eqref{nonstop2} for the best response mapping, when Lemma \ref{MC} is applied in the case described in Remark \ref{remMC} to ensure closedness of the graph of the best response. By working out versions of Lemma \ref{qi} and Lemma \ref{MC} specifically for $J^1, J^2$ at $\partial I$, Assumption \eqref{B2} could be generalized. Together this leads to the formulation of the main Theorems \eqref{ex1} of this section under the combined assumptions \eqref{A}, \eqref{B1}, \eqref{B2}.

\begin{definition}
For $U\subset I$ and $E\sim \Exp(1)$, let 
\begin{align*}
    \mathcal{T}_U&:= \{\tau \in \mathcal{T}: X_\tau\in U \},\\
    \ran_U(E)&:=\{\tau \in \ran(E): X_\tau\in U \}.
\end{align*} We call a pair $(\tau_1,\tau_2) \in \ran_U(E_1)\times \ran_U(E_2)$ \textit{$U$-equilibrium at $x$} if
    \begin{align}
        &J^1(x,\tau_1, \tau_2) \ge J^1(x,\tau_1', \tau_2) \quad \text{and}\label{Uequi1} \\
        &J^2(x, \tau_1,\tau_2) \ge J^2(x, \tau_1,\tau_2')\label{Uequi2}
    \end{align} for all $\tau_1'\in\ran_U(E_1)$ and all $\tau_2' \in \ran_U(E_2)$. If $(\tau_1,\tau_2)$ is a $U$-equilibrium at $x$ for all $x\in I$, we call $(\tau_1,\tau_2)$ \textit{$U$-equilibrium}. A $U$-equilibrium $(\tau_1,\tau_2)\in \mathcal{Z}_U(E_1)\times \mathcal{Z}_U(E_2)$ is called \textit{Markov perfect $U$-equilibrium}.
\end{definition}

First, we rewrite the best response functional for the players when the opponent has chosen a Markovian randomized time. Lemma \ref{indiff1} and Lemma \ref{indiff} formally only cover the case where player 1 reacts to a strategy of player 2, but clearly, by symmetry, analogous results hold if the roles of the players are interchanged.

\begin{lemma}\label{indiff1}
    Fix $\tau_2=\tau^{D,\lambda}\in\mathcal{Z}(E_2)$ and set $A:=A^{\tau_2}$, $A_{t-}:=\lim_{s\nearrow t} A_s$ and $$F:I\to \R,\;x\mapsto \Ex_x\left[\int_{[0,\tau^D)} e^{-A_s-r_1s}f_1(X_s)dA_s\right].$$ Then,
            \begin{align*}
                J^1(x,\tau,\tau_2) =&F(x) + \Ex_x[\ind{\tau<\tau^D}e^{-A_\tau-r_1\tau}(g_1(X_\tau)-F(X_\tau))]\\
                &+\Ex_x[\ind{\tau \ge \tau^D}(e^{-A_{\tau-}-r_1\tau}g_1(X_\tau)+ \ind{\tau>\tau^D}e^{-A_{\tau^D-}-r_1 \tau^D}f_1(X_{\tau^D}))]\\
                =&F(x) +\Ex_x[\ind{\tau<\tau^D} e^{-A_{\tau}-r_1\tau} (g_1(X_\tau)-F(X_\tau))] \\
                &+ \Ex_x[\ind{\tau=\tau^D}e^{-A_{\tau^D-}-r_1\tau^D}g_1(X_{\tau^D})+ \ind{\tau>\tau^D}e^{-A_{\tau^D-}-r_1 \tau^D}f_1(X_{\tau^D})]
            \end{align*} for all $\tau \in \ran(E_1)$ and all $x\in I$. In accordance with \eqref{stoppr} we use the convention $g_1(X_\infty):=\lim_{t\to \infty} g_1(X_t)$.
\end{lemma}

\begin{proof}

    We first calculate the conditional distribution of $\tau_2$ under $\PR_x$ given $\f^X_\infty,\sigma(E_1)$. For $t\ge 0$ we have
    \begin{align*}
        \PR_x(\tau_2>t \vert\f^X_\infty, \sigma(E_1)) = \PR_x(A_t<E_2\vert \f^X_\infty, \sigma(E_1)) = e^{-A_t}.
    \end{align*} Thus $d(-e^{-A_{\bigcdot}})$ (i.e.\ the Markov kernel $\Omega\times \mathcal{B}([0,\infty])\ni (\omega, (a,b]) \mapsto e^{-A_a(\omega)}-e^{-A_b(\omega)} $) is a regular version of the conditional distribution of $\tau_2$ given $\f^X_\infty,\sigma(E_1)$. Applying a change of variables for finite variation processes, cf. \cite[p. 42]{Protter2005}, and using that $A$ only jumps in $\tau^D$, we find
    \begin{align*}
        e^{-A_a(\omega)}-e^{-A_b(\omega)} = \int_{(a,b\wedge \tau^D)} e^{A_{s-}}dA_{\bigcdot}(s)  + e^{A_{\tau^D-}} - e^{A_{\tau^D}}.
    \end{align*}
    This yields
    \begin{align*}
        &J^1(x,\tau,\tau_2) 
        = \Ex_x[\ind{\tau\le \tau_2}e^{-r_1\tau}g_1(X_\tau)+\ind{\tau >\tau_2}e^{-r_1\tau_2}f_1(X_{\tau_2})]\\
        =& \Ex_x[ e^{-r_1\tau}g_1(X_\tau) \Ex_x[  \ind{\tau\le \tau^2}\vert \f^X_\infty,\sigma(E_1)] +\Ex_x[\ind{\tau >\tau_2}e^{-r_1\tau_2}f_1(X_{\tau_2}) \vert \f^X_\infty,\sigma(E_1)]] \\
        =& \Ex_x\left[ e^{-r_1\tau}g_1(X_\tau) \int_{[\tau,\infty)}  d (-e^{-A_{\bigcdot}}) + \int_{[0,\tau)} e^{-r_1 s}f_1(X_s) d(-e^{-A_{\bigcdot}})(s)\right]\\
        =& \Ex_x\bigg[e^{-r_1\tau}g_1(X_\tau) e^{-A_{\tau-}} +\int_{[0,\tau\wedge\tau^D)} e^{-r_1 s}f_1(X_s)e^{-A_{s-}}dA_{\bigcdot}(s) \\
        &+ \ind{\tau>\tau^D}e^{-r_1 \tau^D}f_1(X_{\tau^D}) (e^{-A_{\tau^D-}} - e^{-A_{\tau^D}}) \bigg]\\
        =& \Ex_x\left[e^{-A_{\tau-}-r_1\tau}g_1(X_\tau) +\int_{[0,\tau\wedge\tau^D)} e^{-A_{s-}-r_1 s}f_1(X_s)dA_{\bigcdot}(s) + \ind{\tau>\tau^D}e^{-A_{\tau^D-}-r_1 \tau^D}f_1(X_{\tau^D}) \right]\\
        =& \Ex_x\left[e^{-A_{\tau-}-r_1\tau}g_1(X_\tau) + F(x) - \int_{[\tau\wedge\tau^D,\tau^D)}e^{-A_{s}-r_1 s}f_1(X_s)dA_{\bigcdot}(s) + \ind{\tau>\tau^D}e^{-A_{\tau^D-}-r_1 \tau^D}f_1(X_{\tau^D}) \right]\\
        =&F(x)+\Ex_x\bigg[\ind{\tau<\tau^D} e^{-A_{\tau}-r_1\tau}\bigg(g_1(X_\tau)+\Ex_x\bigg[\theta_\tau\circ\int_{[0,\tau^D)}e^{-A_{s}-r_1 s}f_1(X_s)dA_{\bigcdot}(s) \bigg\vert \f_\tau \bigg]\bigg) \bigg]\\
        &+\Ex_x\big[\ind{\tau \ge \tau^D}(e^{-A_{\tau-}-r_1\tau}g_1(X_\tau)+ \ind{\tau>\tau^D}e^{-A_{\tau^D-}-r_1 \tau^D}f_1(X_{\tau^D}))\big]\\
        =&F(x) +\Ex_x[\ind{\tau<\tau^D} e^{-A_{\tau}-r_1\tau} (g_1(X_\tau)-F(X_\tau))] \\
        &+\Ex_x[\ind{\tau \ge \tau^D}(e^{-A_{\tau-}-r_1\tau}g_1(X_\tau)+ \ind{\tau>\tau^D}e^{-A_{\tau^D-}-r_1 \tau^D}f_1(X_{\tau^D}))]\\
        =&F(x) +\Ex_x[\ind{\tau<\tau^D} e^{-A_{\tau}-r_1\tau} (g_1(X_\tau)-F(X_\tau))] \\
        &+ \Ex_x[\ind{\tau=\tau^D}e^{-A_{\tau^D-}-r_1\tau^D}g_1(X_{\tau^D})+ \ind{\tau>\tau^D}e^{-A_{\tau^D-}-r_1 \tau^D}f_1(X_{\tau^D})].
    \end{align*}
\end{proof}

The following lemma deals with the properties of the stopping problem faced by player 1 when player 2 has chosen a fixed strategy. The first statement is that under Assumption \eqref{A} randomization can be an optimal strategy only at points where stopping is also optimal. From then on we work under the additional assumptions \eqref{B1} and \eqref{B2}. In the second part of the following, relying on the previous lemma which guarantees that the best-response problem is a Markovian stopping problem, we reformulate the problem so that general theory provides the existence of Markovian best responses. The third part establishes a connection between the value functions of the problem on $I$ and those of discretized versions of the problem. Finally, we show what is sometimes called the \textit{principle of indifference}, which says that randomization is optimal if and only if both stopping and continuing are optimal.

\begin{lemma}\label{indiff}
    Assume \eqref{A} to hold. Let $U\subset I$, fix $\tau_2=\tau^{D,\lambda}\in\mathcal{Z}(E_2)$, and consider the best response problem
        \begin{align}
        V^1_U(x,\tau_2) = \sup_{\tau \in \mathcal{R}_U(E_1)} J^1(x,\tau,\tau_2), \quad x\in U.\label{pls1}
    \end{align} We call $\tau^* \in \ran_U(E_1)$ a \textit{best response (to $\tau_2$)} or \textit{optimal stopping time (in \eqref{pls1}/in the stopping problem associated with $V^1_U(x,\tau_2)$)} if $V^1_U(x,\tau_2)=J^1(x,\tau^*,\tau_2)$ for all $x\in U$. 
    
    \begin{enumerate}[(i)]
        \item \label{weakindiff} If $\tau^*=\tau^{D^*,\lambda^*}\in \mathcal{Z}_U$ is a best response, then $V^1_U(y,\tau_2)=g_1(y)$ for all $y\in (\supp(\lambda^*)\cap D)\cup (D^*)^c$. And if $\supp(\lambda^*)\cap D^c \cap \{y\in U : f_1(y)>g_1(y)\}=\varnothing$, then $V^1_U(y,\tau_2)=g_1(y)$ for all $y\in \supp(\lambda^*)\cup (D^*)^c$, in particular $\PR_x(X_{\tau^*}\in \{y\in U: V^1_U(y,\tau_2)=g_1(y)\}\vert \tau^*<\infty)=1$ for all $x\in I$.

        \item \label{indiff2} Assume \eqref{B1} and \eqref{B2}. If $U$ is finite or $U=I$, then there exists a best response $\tau^*\in \mathcal{Z}_U(E_1)$. 
        
        \item \label{indiff15} Assume \eqref{B1} and \eqref{B2}. If $U_1\subseteq U_2\subseteq \dots \subseteq I$ are finite and $\bigcup_mU_m$ is dense in $I$, then
        \[V_I^1(x)=\sup_{m\in\N}\sup_{\tau\in \mathcal Z_{U_m}(E_1)}J^1(x,\tau,\tau_2)\;\;\mbox{for all $x\in U$.}\]       
        
        \item \label{weakindiff2} Assume \eqref{B1} and \eqref{B2} and let $U$ be finite. We extend $V^1_U$ to $I\setminus U$ via 
        \begin{align*}
            V^1_U(y,\tau_2)&:=\begin{cases}
                g_1(y), &r_1=0,\\
                0, &r_1>0
            \end{cases}\quad \text{for } \,y\in \partial I\setminus U 
        \end{align*} and write $V^1_U(X_\infty,\tau_2):= \lim_{t \to \infty}V^1_U(X_t,\tau_2)$. With this we set 
        \begin{align*}
            \sigma_y:=&\eta^{(U\cup \partial I)\setminus\{y\}}, \quad y\in I,\\
            S_1:=&\{y\in U : V^1_U(y,\tau_2)=g_1(y)\},\\
            S_2:=&\{y\in S_1: V^1_U(y,\tau_2)> \Ex_y[\ind{\sigma_y> \tau_2} e^{-r_1\tau_2 }f_1(X_{\tau_2}) + \ind{\sigma_y\le \tau_2}e^{-r_1\sigma_y} V^1_U(X_{\sigma_y},\tau_2)]\},\\
            T_1:=&\{\tau^{D^*,\lambda^*}\in \mathcal{Z}_U: V^1_U(y,\tau_2)=J^1(y,\tau^{D^*,\lambda^*},\tau_2) \,\forall y\in U,\\
            &\supp(\lambda^*)\cap D^c \cap \{y\in U : f_1(y)>g_1(y)\}=\varnothing\},\\
            T_2:=&\{\tau^{D^*,\lambda^*}\in \mathcal{Z}_U: S_2 \subset (D^*)^c\subset S_1, \supp(\lambda^*)\subset S_1,\\
            &\supp(\lambda^*)\cap D^c \cap \{y\in U : f_1(y)>g_1(y)\}=\varnothing\}.
        \end{align*} Then $T_1=T_2$.


    \end{enumerate}
\end{lemma}

\begin{proof}
\eqref{weakindiff} Clearly, it is sufficient to show the first claim for $x\in \supp(\lambda^*)\cap D$. Set $\tau^*(z):=\inf\{t\ge 0: A^{\tau^*}_t\ge z\}\in \mathcal{T}_U$. This means $\tau^*=\tau^*(E_1)$. Since $X, E_2$ and $E_1$ are independent, Fubini's theorem yields
\begin{align*}
    &V^1_U(x,\tau_2)
    = J^1(x,\tau^*,\tau_2)=\Ex_x[\ind{\tau^*(E_1)\le \tau_2}e^{-r_1\tau^*(E_1)}g_1(X_{\tau^*(z)})+\ind{\tau^*(E_1) >\tau_2}e^{-r_1\tau_2}f_1(X_{\tau_2})] \\
    =& \int_{[0,\infty)} \Ex_x[\ind{\tau^*(z)\le \tau_2}e^{-r_1\tau^*(z)}g_1(X_{\tau^*(z)})+\ind{\tau^*(z) >\tau_2}e^{-r_1\tau_2}f_1(X_{\tau_2})]d\PR_x^{E_1}(z)\\
    =& \int_{[0,\infty)} J^1(x,\tau^*(z),\tau_2) d\PR_x^{E_1}(z).
\end{align*} Since $J^1(x,\tau^*(z),\tau_2)\le V^1_U(x,\tau_2)$ for all $t\ge 0$, we infer $J^1(x,\tau^*(z),\tau_2)= V^1_U(x,\tau_2)$ for $\PR_x^{E_1}$-a.a. $z\in [0,\infty)$. Note that $E_1\sim \Exp(1)$ implies that the notions $\PR_x^{E_1}$-a.a. $z\in [0,\infty)$ and Lebesgue almost all $z\in [0,\infty)$ are identical. By continuity there is a $\PR_x$-null set $N\in \f_\infty^X$ such that for all $t>0$ and all $\omega\in \Omega\setminus N$ there is a neighborhood $B_{t,\omega}$ of $x$ with $\inf_{y\in B} l^y_t(\omega)>0$. Since $x\in \supp(\lambda^*)$, $\lambda^*(B)>0$ for every neighborhood $B$ of $x$ and thus 
\begin{align*}
    &\PR_x(A^{\tau^*}_t>0 \,\forall t>0)
    \ge \PR_x\left(\int_{B_{t,\omega}\cap D^*} l^y_t d\lambda^*(y)>0\,\forall t>0\right)\\
    \ge& \PR_x\left(\inf_{y\in B_{t,\omega}} l^y_t(\omega) \lambda^*(B_{t,\omega}\cap D^*)>0\forall t>0\right)=1
\end{align*} Form this we infer $\tau^*(z)\to 0$ $\PR_x$-a.s.\ for $z\searrow 0$. Moreover, $t\mapsto A^{\tau_2}_t$ is continuous on $[0,\tau^D)$ with $A^{\tau_2}_0=0$. As $\tau^D>0$ $\PR_x$-a.s., this implies $\ind{\tau^* > \tau_2}=\ind{A^{\tau_2}_{\tau^*(z)}> E_2}\to 0$ $\PR_x$-a.s. Choosing $(z_n)_{n\in \N}$ such that $J^1(x,\tau^*(z_n),\tau_2)= V^1_U(x,\tau_2)$ and $z_n\searrow 0$, dominated convergence provides
\begin{align*}
    &V^1_U(x,\tau_2)=\lim_{n\to \infty}J^1(x,\tau^*(z_n),\tau_2) \\
    =&  \Ex_x\left[\lim_{n\to \infty} \ind{\tau^*(z_n)\le \tau_2}e^{-r_1\tau^*(z_n)}g_1(X_{\tau^*(z_n)})+\ind{\tau^*(z_n) >\tau_2}e^{-r_1\tau_2}g_1(X_{\tau_2})\right] =g_1(x).
\end{align*}

\eqref{indiff2} We start with the first claim. In order to employ classical existence results for stopping problems, we need to consider non negative payoffs. To achieve this, we use a change of measure to obtain an equivalent problem without discounting, so that the payoffs can be shifted appropriately without changing the solution. The first technicality that has to deal with is the fact that we can only eliminate the discounting up to the time of absorption of $X$, so we start by reducing the problem to that time frame.

If $r_1>0$, it is optimal for player 1 to never stop in points in $\{g_1<0\}\cap \partial I\cap D$. Thus, the problem \eqref{pls1} is equivalent to the problem 
\begin{align}
    \tilde V^1_U(x,\tau_2)&= \sup_{\tau \in \mathcal{R}_U(E_1)} \tilde J^1(x,\tau,\tau_2), \quad x\in U,\label{pls2}\\
    \tilde J^1(x,\tau,\tau_2):&= \Ex_x[\ind{\tau\le \tau_2}e^{-r_1\tau}\tilde g_1(X_\tau)+\ind{\tau >\tau_2}e^{-r_1\tau_2}f_1(X_{\tau_2})],\notag\\
    \tilde g_1 :&= \max\{g_1, \inda{D^c\cup I^\circ} g_1+\ind{r_1=0}\inda{\partial I\cap D}g_1\}\notag
\end{align} in the sense that for every optimal stopping time $\tau^{D^*,\lambda^*}\in \mathcal{Z}_U(E_1)$ in \eqref{pls2} the stopping time $\tau^{D^*\cup (\{g_1<0\}\cap \partial I\cap D),\lambda^*}$ with $\lambda^*$ extended by 0 is optimal in \eqref{pls1}. By construction and due to assumption \eqref{B2}, $\tilde J^1(x,\tau,\tau_2)\le \tilde J^1(x,\tau\wedge\tau^{I^\circ},\tau_2)$ for all $\tau$. This implies
\begin{align*}
    \tilde V^1_U(x,\tau_2)=\sup_{\tau \in \mathcal{R}_U(E_1)} \tilde J^1(x,\tau\wedge\tau^{I^\circ},\tau_2).
\end{align*} By \cite[Lemma 2.1.]{dayanik2008optimal}, there are functions $\psi, \varphi:I\to [0,\infty)$ such that $e^{-r_1(t\wedge \eta_{\max I})}\psi(X_{t\wedge \eta_{\max I}})$ and $e^{-r_1(t\wedge \eta_{\min I})}\varphi(X_{t\wedge \eta_{\min I}})$ are martingales. By \citep[Lemma 2.2.]{dayanik2008optimal}, $\psi$ and $\varphi$ are continuous and $\psi+\varphi>0$. We now define new measures $Q_x$, $x\in I$ via $\frac{dQ_x}{d\PR_x}:= e^{-r_1(t\wedge \tau^{I^\circ})}\frac{\psi(X_{\tau\wedge\tau^{I^\circ}})+\varphi(X_{\tau\wedge\tau^{I^\circ}})}{\psi(x)+\varphi(x)}$ on $\f_t$ and denote the expectation under $Q_x$ by $\Ex_x^Q$. Under $(Q_x)_{x\in I}$ the process $X$ is still a regular diffusion on $I$ with absorbing boundary $\partial I$. Additionally the distribution of $E_1, E_2$ are not affected by the change of measure as $\frac{dQ_x}{d\PR_x}$ only depends on $X$ which is independent of $E_1, E_2$. With the convention $\frac{\tilde g_1}{\psi+\varphi}(X_\infty):=\lim_{t\to \infty}\frac{\tilde g_1}{\psi+\varphi}(X_t)$ and \eqref{B2} we find that
\begin{align*}
    &\tilde J^1(x,\tau\wedge\tau^{I^\circ},\tau_2)\\
    =&\Ex_x\left[e^{-r_1 \tau\wedge\tau^{I^\circ}\wedge \tau_2} \frac{ \psi(X_{\tau\wedge\tau^{I^\circ}\wedge \tau_2})+\varphi(X_{\tau\wedge\tau^{I^\circ}\wedge \tau_2})}{\psi(x)+\varphi(x)} \frac{\ind{\tau\wedge\tau^{I^\circ}\le \tau_2}\tilde g_1(X_{\tau\wedge\tau^{I^\circ}})+\ind{\tau\wedge\tau^{I^\circ} >\tau_2}f_1(X_{\tau_2})}{\psi(X_{\tau\wedge\tau^{I^\circ}\wedge \tau_2})+\varphi(X_{\tau\wedge\tau^{I^\circ}\wedge \tau_2})}\right](\psi(x)+\varphi(x))\\
    =&\Ex_x^Q\left[ \frac{\ind{\tau\wedge\tau^{I^\circ}\le \tau_2}\tilde g_1(X_{\tau\wedge\tau^{I^\circ}})+\ind{\tau\wedge\tau^{I^\circ} >\tau_2}f_1(X_{\tau_2})}{\psi(X_{\tau\wedge\tau^{I^\circ}\wedge \tau_2})+\varphi(X_{\tau\wedge\tau^{I^\circ}\wedge \tau_2})}\right] (\psi(x)+\varphi(x))\\
    =& \Ex_x^Q\left[ \ind{\tau\le \tau_2} \left(\frac{\tilde g_1}{\psi+\varphi}-\inf_{y\in I}\frac{\tilde g_1}{\psi+\varphi}(y)\right)(X_{\tau})+ \ind{\tau>\tau_2} \left(\frac{f_1}{\psi+\varphi}-\inf_{y\in I}\frac{\tilde g_1}{\psi+\varphi}(y)\right)(X_{\tau_2})  \right] (\psi(x)+\varphi(x))\\
    &+ \inf_{y\in I}\frac{\tilde g_1}{\psi+\varphi}(y)(\psi(x)+\varphi(x)).
\end{align*} If we set $\hat g_1:= \frac{\tilde g_1}{\psi+\varphi}-\inf_{y\in I}\frac{\tilde g_1}{\psi+\varphi}(y)$, $\hat f_1:=\frac{f_1}{\psi+\varphi}-\inf_{y\in I}\frac{\tilde g_1}{\psi+\varphi}(y)$, it is now sufficient to find a Markovian solution to the problem
\begin{align}
    \hat V^1_U(x,\tau_2)= \sup_{\tau \in \mathcal{R}_U(E_1)} \Ex_x^Q[ \ind{\tau\le \tau_2} \hat g_1(X_{\tau})+ \ind{\tau>\tau_2} \hat f_1(X_{\tau_2})]. \label{pls3}
\end{align} We set $Y=(Y_t)_{t\in [0,\infty)}:=(X_{t\wedge \tau^D})_{t \in [0,\infty)}$, $\hat F:I \to \R$, $x\mapsto \Ex_x[\int_{[0,\tau^D)} e^{-A_s}\hat f_1(X_s)dA_s]$, define the upper semicontinuous $h_U:I\to [0,\infty)$
\begin{align*}
    h_U(x):=\begin{cases}
        (\hat g_1(x)-\hat F(x))^+, &x\in D\cap U \cap I^\circ,\\
        \hat f_1(x), &x\in  D^c,\\
        \hat g_1(x),& x\in \partial I\setminus D^c,\\
        0, &x\in (D\cap I^\circ)\setminus U.
    \end{cases}
\end{align*} and
\begin{align}
    H_U(x) := \sup_{\tau \in \ran(E_1)} \Ex_x\left[e^{-A_{\tau-}}h_U(Y_\tau)\right], \quad x\in I,\label{pls}
\end{align} where $e^{-A_{\infty-}}h_U(Y_\infty):= \lim_{t \to \infty}e^{-A_{t-}}h_U(Y_t)$. Note that since $X$ is absorbed at $\partial I$ we have $\PR_z(A_t=0 \forall t <\tau^D)=\PR_z(l^y_t=0\forall t\ge0)=1$ for all $z\in \partial I$ and thus $F=0$ on $\partial I\setminus D^c$. Thus, by Lemma \ref{indiff1}, which can by applied under the measure $Q$ analogously,
\begin{align*}
    &\Ex_x^Q[ \ind{\tau\le \tau_2} \hat g_1(X_{\tau})+ \ind{\tau>\tau_2} \hat f_1(X_{\tau_2})]\\
    \le& \hat F(x) +\Ex_x^Q[\ind{\tau<\tau^D} e^{-A_{\tau}} (\hat g_1(X_\tau)-\hat F(X_\tau))^+] + \Ex_x^Q[\ind{\tau\ge\tau^D}e^{-A_{\tau^D-}}\hat f_1(X_{\tau^D})]\\
    =&\hat F(x)+ \Ex_x^Q[ e^{-A_{\tau\wedge\tau^D-}} h_U(Y_{\tau\wedge\tau^D})]
\end{align*} for all $\tau \in \ran(E_1)$ with equality if $\PR_x(X_\tau\in D^c\cup ((D\cap I^\circ)\setminus (U\setminus\{\hat g_1-\hat F<0\}) ))=0$ due to \eqref{B2}. 
A solution to such problems such as \eqref{pls} in the greatest generality is given in \cite{beibel2001optimal}. In fact, this provides the existence of a maximizing first exit time $\tau^*=\tau^{D^*}\in \mathcal{Z}$ in the class $\mathcal T$, since our function $h=h_U$ obviously satisfies that the relevant maxima in \cite[Section 3/6]{beibel2001optimal} are attained. By the characterization of randomized stopping times by the mixing of stopping times in $\mathcal T$  (\cite[Section 7]{touzi2002continuous}), it is immediately clear that $\tau^*$ is actually a maximum in the class of all randomized stopping times. 
Next we show that from $\tau^*$ we can construct $\hat \tau:=\tau^{\hat D}\in \mathcal{Z}_U$ with $\Ex_x\left[e^{-A_{\tau^*-}}h_U(Y_{\tau^*})\right]=\Ex_x\left[e^{-A_{\tau\wedge \tau^D-}}h_U(Y_{\tau\wedge \tau^D})\right]$, $x\in I$, and $\PR_x(X_{\hat \tau}\in D^c\cup ((D\cap I^\circ)\setminus (U\setminus\{\hat g_1-\hat F<0\}) ))=0$ in order to obtain a Markovian solution of \eqref{pls3}. We set $\hat D:= D^*\cup(\partial I\setminus U)\cup D^c\cup ((D\cap I^\circ)\setminus (U\setminus\{\hat g_1-\hat F<0\}))$. Observe that $\tau^*\wedge \tau^D$ is still optimal since $A_{t-}=\infty$ for all $t>\tau^D$ and so $e^{-A_{\tau^*-}}h_U(Y_{\tau^*})\le e^{-A_{\tau^*\wedge\tau^D-}}h_U(Y_{\tau^*\wedge\tau^D-})$.

By definition $I\setminus\hat D\subset U$, i.e.\ $\hat\tau \in \mathcal{Z}_U$. First $\hat g_1, \hat f_1 \ge 0$ implies $h_U\le 0$ on $((D\cap I^\circ)\setminus U)\cup(I^\circ\cap D\cap U\cap\{\hat g_1-\hat F<0\})$. Secondly since $X$ is absorbed in $\partial I$ and $A_{t}=A_{\tau^{I^\circ}}$ for all $t\in (\tau^{I^\circ},\tau^D)$, $e^{-A_t-}h_U(Y_t)= e^{-A_{\tau^{I^\circ}}-}h_U(Y_{\tau^{I^\circ}})$ on $\{Y_{\tau^{I^\circ}}\in D\cap I^
\circ\}$ for all $t\ge \tau^{I^\circ}$. Together
\begin{align*}
    H_U(x)=\Ex_x\left[e^{-A_{\tau^*-}}h_U(Y_{\tau^*}) \right] = \Ex_x\left[e^{-A_{ \tau^*\wedge\tau^D-}}h_U(Y_{\tau^*\wedge\tau^D}) \right]= \Ex_x\left[e^{-A_{\hat \tau\wedge\tau^D-}}h_U(Y_{\hat\tau\wedge\tau^D}) \right]
\end{align*} as required.

\eqref{indiff15} Here, too, the findings from \cite{beibel2001optimal} directly deliver the desired result, even in the class of first exit times from sets with boundaries in $U_m$: From the proofs in \cite[Section 5]{beibel2001optimal} you can immediately see that stopping at points whose values converge towards the maximum of the functions under consideration delivers the desired result.

\eqref{weakindiff2} We show $T_1\subset T_2$ first. Let $\tau^*=\tau^{D^*,\lambda^*}\in T_1$. By \eqref{weakindiff} it suffices to show $S_2\setminus (D^*)^c=\varnothing$. We assume the contrary, i.e.\ that there is some $x\in S_2\setminus (C^*)^c$. This implies $A^{\tau^*}_{\eta^{U\setminus\{x\}}-}<\infty$ $\PR_x$-almost surely. By definition $\tau^*(z)\ge \eta^{U\setminus\{x\}}$ for all $z > A^{\tau^*}_{\eta^{U\setminus\{x\}}-}$. By arguing as in the proof of \eqref{weakindiff} in the first and fourth step and using the Markov property in the third one we reach a contradiction to $\tau^*\in T_1$ via
\begin{align*}
    &V^1_U(x,\tau_2)= J^1(x,\tau^*(z),\tau_2)\\
    =&\Ex_x\big[ \ind{\tau^*(z)\wedge \tau_2 <\eta^{U\setminus \{x\}}} (\ind{\tau^*(z)\le \tau_2}e^{-r_1\tau^*(z)}g_1(x)+\ind{\tau^*(z) >\tau_2} e^{-r_1\tau_2}f_1(x))\\
    &+ \ind{\tau^*(z) \wedge \tau_2\ge \eta^{U\setminus \{x\}}}(\ind{\tau^*(z)\le \tau_2}e^{-r_1\tau^*(z)}g_1(X_{\tau^*(z)})+\ind{\tau^*(z) >\tau_2}e^{-r_1\tau_2}f_1(X_{\tau_2}))  \big]\\
    =&\Ex_x\big[ \ind{\tau^*(z)\wedge \tau_2<\eta^{U\setminus \{x\}}} (\ind{\tau^*(z)\le \tau_2}e^{-r_1\tau^*(z)}g_1(x)+\ind{\tau^*(z) >\tau_2}e^{-r_1\tau_2}f_1(x))\\
    &+ \ind{\tau^*(z)\wedge \tau_2 \ge \eta^{U\setminus \{x\}}} e^{-r_1 \eta^{U\setminus \{x\}}} V^1_U(X_{\eta^{U\setminus \{x\}}},\tau_2)\big]\\
    \stackrel{z\to \infty}{\longrightarrow}& \Ex_x \big[\ind{\eta^{U\setminus\{x\}}> \tau_2} e^{-r_1\tau_2 }f_1(x) + \ind{\eta^{U\setminus\{x\}}\le \tau_2}e^{-r_1\eta^{U\setminus\{x\} }} V^1_U(X_{\eta^{U\setminus\{x\}}},\tau_2)\big].
\end{align*}

Next, we prove $T_2\subset T_1$. We start by showing that $\tau^{S_2^c}$ is a best response. By \eqref{indiff2} there is some optimal pure first entry time $\tau^{D_1}\in \mathcal{T}_1$. By the above $S_2\subset D_1^c$, so we set $S_3:=D_1^c\setminus S_2$. Based on that, we define $\sigma^1:=\tau^{D_1}+ \ind{X_{\tau^{D_1}}\in S_3} \theta_{\eta^{S_3}} \circ \theta_{\sigma_{X_0}}\circ\tau^{D_1}$, i.e. if $\tau^{D_1}$ stops in the unwanted set $S_3$, $\sigma$ waits until the process $X$ hits another point in $U\cup \partial I$ and then goes by the rule $\tau^{D_1}$ again. Using the definitions, the fact $J(x,\tau^{D_1},\tau_2)=V(x,\tau_2)$ for all $x\in U\cup \partial I$ and the Markov property, we find
\begin{align*}
    &J^1(x,\sigma^1,\tau_2)\\
    =&\Ex_x[\ind{X_{\tau^{D_1}}\not \in S_3} (\ind{\tau^{D_1}\le \tau_2}e^{-r_1\tau^{D_1}} g_1(X_{\tau^{D_1}}) + \ind{ \tau^{D_1} >\tau_2} e^{-r_1\tau_2}f_1(X_{\tau_2}) ) \\
    &+  \ind{X_{\tau^D} \in S_3} e^{-r_1\tau^{D_1}} \Ex_{X_{\tau^{D_1}}} [ \ind{\theta_{\sigma_{X_0}}\circ\tau^{D_1}\le \tau_2}e^{-r_1\theta_{\sigma_{X_0}}\circ\tau^{D_1}} g_1(X_{\theta_{\sigma_{X_0}}\circ\tau^{D_1}}) + \ind{ \theta_{\sigma_{X_0}}\circ\tau^{D_1} >\tau_2} e^{-r_1\tau_2}f_1(X_{\tau_2}) ]]\\
    =&\Ex_x[\ind{X_{\tau^{D_1}}\not \in S_3} (\ind{\tau^{D_1}\le \tau_2}e^{-r_1\tau^{D_1}} g_1(X_{\tau^{D_1}}) + \ind{ \tau^{D_1} >\tau_2} e^{-r_1\tau_2}f_1(X_{\tau_2}) ) \\
    &+\ind{X_{\tau^D} \in S_3} e^{-r_1\tau^{D_1}} \Ex_{X_{\tau^{D_1}}} [\ind{\sigma_{X_0}>\tau_2} e^{-r_1\tau_2} f_1(X_{\tau_2}) + \ind{ \sigma_{X_0}\le \tau_2} e^{-r_1 \sigma_{X_0}}\Ex_{X_{\sigma_{X_0}}}[ \ind{\tau^{D_1}\le \tau_2}e^{-r_1\tau^{D_1}} g_1(X_{\tau^{D_1}}) \\
    &+ \ind{ \tau^{D_1} >\tau_2} e^{-r_1\tau_2}f_1(X_{\tau_2} )  ]  ]  ]\\
    =& \Ex_x[\ind{X_{\tau^{D_1}}\not \in S_3} (\ind{\tau^{D_1}\le \tau_2}e^{-r_1\tau^{D_1}} g_1(X_{\tau^{D_1}}) + \ind{ \tau^{D_1} >\tau_2} e^{-r_1\tau_2}f_1(X_{\tau_2}) ) \\
    &+\ind{X_{\tau^D} \in S_3} e^{-r_1\tau^{D_1}} \Ex_{X_{\tau^{D_1}}} [\ind{\sigma_{X_0}>\tau_2} e^{-r_1\tau_2} f_1(X_{\tau_2}) + \ind{ \sigma_{X_0}\le \tau_2} e^{-r_1 \sigma_{X_0}}V_U^1(X_{\sigma_{X_0}},\tau_2)] \\
    =& \Ex_x[\ind{X_{\tau^{D_1}}\not \in S_3} (\ind{\tau^{D_1}\le \tau_2}e^{-r_1\tau^{D_1}} V^1_U(X_{\tau^{D_1}},\tau_2) + \ind{ \tau^{D_1} >\tau_2} e^{-r_1\tau_2}f_1(X_{\tau_2}) ) \\
    &+\ind{X_{\tau^D} \in S_3} e^{-r_1\tau^{D_1}} V^1_U(X_{\tau^{D_1}},\tau_2)] \\
    =&J^1(x,\tau^{D_1},\tau_2)=V^1_U(x,\tau_2)
\end{align*} for all $x\in U$.
Recursively set $\sigma^n:=\sigma^{n-1}+ \ind{X_{\tau^{D_1}}\in S_3} \theta_{\eta^{S_3}} \circ \theta_{\sigma_{X_0}}\circ \sigma^{n-1}$. Analogous to the previous calculation we obtain $J^1(x,\sigma^n,\tau_2)=V^1_U(x,\tau_2)$. Since $U$ is assumed to be finite, Assumption \eqref{B1} yields $\PR_x(\exists N\in \N : \sigma_n =\tau^{S^c_2}\forall n\ge N)=1$. Thus, dominated convergence yields
\begin{align*}
    V^1_U(x,\tau_2)= \lim_{n\to \infty} J^1(x,\sigma^n,\tau_2) = J^1(x,\tau^{S^c_2},\tau_2)
\end{align*} for all $x\in U$. In classical Markovian stopping, given an optimal stopping time, optimality carries over to all smaller stopping times that stop exclusively on the set where payoff and value function coincide. We use the same principle to show that all $\tau'\in T_2$ are in fact optimal. By construction $\tau'\le \tau^{S^c_2}$ and thus $\tau'+\theta_{\tau'}\circ \tau^{S^c_2}=\tau^{S^c_2}$. Using that $V^1(X_{\tau'})=g_1(X_{\tau'})$ and $V^1(X_{\tau^{S^c_2}})=g_1(X_{\tau^{S^c_2}})$ as $\tau',\tau^{S^c_2}\in T_2$, by the Markov property we obtain
\begin{align*}
    &J^1(x,\tau',\tau_2)\\
    =&\Ex_x[\ind{\tau' > \tau_2}e^{-r_1\tau_2}f_1(X_{\tau_2})+\ind{\tau'\le \tau_2} e^{-r_1\tau'} g_1(X_{\tau'})]\\
    =&\Ex_x[\ind{\tau' > \tau_2}e^{-r_1\tau_2}f_1(X_{\tau_2})+\ind{\tau'\le \tau_2} e^{-r_1\tau'}V^1_U(X_{\tau'})]\\
    =&\Ex_x[\ind{\tau' > \tau_2}e^{-r_1\tau_2}f_1(X_{\tau_2})+\ind{\tau'\le \tau_2} e^{-r_1\tau'}\Ex_{X_{\tau'}}[\ind{\tau^{S^c_2} > \tau_2}e^{-r_1\tau_2}f_1(X_{\tau_2})+\ind{\tau^{S^c_2}\le \tau_2} e^{-r_1\tau^{S^c_2}} g_1(X_{\tau^{S^c_2}})]]\\
    =&\Ex_x[\ind{\tau' > \tau_2}e^{-r_1\tau_2}f_1(X_{\tau_2})+\ind{\tau'\le \tau_2} e^{-r_1\tau'}\\
    &\cdot\Ex_{x}[\ind{\theta_{\tau'}\circ\tau^{S^c_2} > \theta_{\tau'}\circ\tau_2}e^{-r_1\theta_{\tau'}\circ\tau_2} f_1(X_{\theta_{\tau'}\circ\tau_2}+\tau')+\ind{\theta_{\tau'} \circ\tau^{S^c_2} \le \theta_{\tau'}\circ \tau_2} e^{-r_1\theta_{\tau'}\circ\tau^{S^c_2}} g_1(X_{\theta_{\tau'}\circ\tau^{S^c_2}}+\tau') \vert \f_{\tau'} ]]\\
    =&\Ex_x[\ind{\tau' > \tau_2}e^{-r_1\tau_2}f_1(X_{\tau_2})+ \ind{\tau'\le \tau_2 < \theta_{\tau'}\circ\tau^{S^c_2}+\tau'} e^{-r_1 \tau'-r_1 \theta_{\tau'}\circ\tau_2 } f_1(X_{\theta_{\tau'}\circ\tau_2+\tau'}) \\
    &+ \ind{\tau'+\theta_{\tau'}\circ \tau^{S^c_2}\le \tau_2} e^{-r_1\tau' -r_1\theta_{\tau'}\circ \tau^{S^c_2} }g_1(X_{\tau'+\theta_{\tau'}\circ \tau^{S^c_2}})]\\
    =& J^1(x,\tau^{S^c_2},\tau_2)=V^1_U(x,\tau_2)
\end{align*} for all $x\in U$, i.e.\ $\tau'\in T_1$.

\end{proof}

The subsequent lemma captures the fact that two Markovian randomized strategies with independent randomization will, with probability 1, never stop at the same time while $X$ is outside the closure of the stopping region of either strategy.  This lemma turns out to be a useful tool for the characterization of the continuity points of the mappings $J^1, J^2$, see Lemma \ref{MC}.

In accordance with the notations introduced in Lemma \ref{aux1} and the proof of Theorem \ref{comp2}, we start by introducing a notation that will be used for the remainder of this paper. 
Given any $x \in I$ and a Markovian randomized time $\tau=\tau^{D,\lambda} \in \mathcal{Z}$ we set
\begin{align}
    L^\tau_x:=  \sup( (I \setminus D) \cap (-\infty,x]), \quad R^\tau_x:= \inf((I \setminus D) \cap [x,\infty)). \notag
\end{align} Here we rely on the conventions $\inf \varnothing := \infty, \sup \varnothing := -\infty$.

\begin{lemma}\label{qi}(quasi independence of Markovian randomized times)\\
Let $\tau_1:=\tau^{D_1,\lambda_1}\in \mathcal{Z}(E_1)$, $ \tau_2:= \tau^{D_2,\lambda_2} \in \mathcal{Z}(E_2)$ and $x \in I$. 
\begin{enumerate}[(i)]
    \item \label{qi1} If $L^{\tau_1}_x \neq L^{\tau_2}_x$ and $R^{\tau_1}_x \neq R^{\tau_2}_x$, then $\PR_x(\tau_1 = \tau_2) =0$.
    \item \label{qi2} If $L^{\tau_1}_x \neq L^{\tau_2}_x$, then $\PR_x(\eta^{R^{\tau_1}_x} > \tau_1,\tau_1 = \tau_2) =0$.
    \item \label{qi3} If $R^{\tau_1}_x \neq R^{\tau_2}_x$, then $\PR_x(\eta^{L^{\tau_1}_x} > \tau_1 ,\tau_1 = \tau_2) =0$.    
\end{enumerate}
\end{lemma}

\begin{proof}
    \eqref{qi1} Set $S:=(L^{\tau_1}_x, R^{\tau_1}_x)$, $T:=(L^{\tau_2}_x, R^{\tau_2}_x)$. We start by gathering some properties of Markovian randomized times to use later on in this proof.

    By definition we have
    \begin{align}
        &\{ \tau_1 < \tau^S \} \subset \{ A^{\tau_1}_{\tau_1} < \infty\} \quad \text{and}\notag\\
        &\{ \tau_2 < \tau^T \} \subset \{ A^{\tau_2}_{\tau_2} < \infty\}. \label{prop1}
    \end{align}
    Moreover,
    \begin{align}
        &\{\tau_1 =\tau_2\} \subset \{ A^{\tau_1}_{\tau_2} \ge E_1 \} \cap \{ A^{\tau_1}_{\tau_2-\eps}< E_1 \,\, \forall \eps>0 \} \quad \text{and}\notag\\
        & \{\tau_2 =\tau_1\} \subset \{ A^{\tau_2}_{\tau_1} \ge E_2 \} \cap \{ A^{\tau_2}_{\tau_1-\eps}< E_2 \,\, \forall \eps>0 \}. \label{prop2}
    \end{align}
    The mappings
    \begin{align*}
        &[0, \tau^S) \ni t \mapsto A^{\tau_1}_t \in [0,\infty) \quad \text{and} \quad
        [0, \tau^T) \ni t \mapsto A^{\tau_2}_t \in [0,\infty).
    \end{align*} are continuous, while 
    \begin{align*}
        \tau^S:= \inf\{ t \ge 0: A^{\tau_1}_t = \infty\} \quad \tau^T:= \inf\{ t \ge 0: A^{\tau_2}_t = \infty\}.
    \end{align*} We conclude
    \begin{align}
        &\{ \tau_1 =\tau_2 \} \cap \{ \tau_1 < \tau^S\} \cap \{ \tau_2 = \tau^T\} \notag\\
        \subset& \{ \tau_1 = \tau^T\} \cap \{ \tau_1< \tau^S\} \notag\\
        \stackrel{\eqref{prop1}}{\subset}& \{ \tau_{1} = \tau^T \} \cap \{ A^{\tau_1}_{\tau^T} =  A^{\tau_1}_{\tau_1}<\infty\}\notag\\
        \stackrel{\eqref{prop2}}{\subset}& \{ A^{\tau_1}_{\tau^T} \ge E_1\} \cap \{ A^{\tau_1}_{\tau^T- \eps} < E_1 \,\, \forall \eps >0 \} \cap \{ A^{\tau_1}_{\tau^T}< \infty \} \notag\\
        \subset& \{ A^{\tau_1}_{\tau^T} = E_1 \} \label{res1}
    \end{align} and analogously
    \begin{align}
        \{ \tau_1 =\tau_2 \} \cap \{ \tau_1 = \tau^S\} \cap \{ \tau_2 < \tau^T\} \subset \{ A^{\tau_2}_{\tau^S} = E_2 \}. \label{res2}
    \end{align} Similarly,
    \begin{align}
        &\{ \tau_1 =\tau_2 \} \cap \{ \tau_1 < \tau^S\} \cap \{ \tau_2 < \tau^T\} \notag\\
        \stackrel{\eqref{prop1}}{\subset}& \{ \tau_1 =\tau_2 \} \cap \{ A^{\tau_1}_{\tau_2} =  A^{\tau_1}_{\tau_1}< \infty\}\notag \\
        \stackrel{\eqref{prop2}}{\subset}& \{ A^{\tau_1}_{\tau_2} \ge E_1\} \cap \{ A^{\tau_1}_{\tau_2-\eps} < E_1\,\, \forall \eps>0\} \cap \{ A^{\tau_1}_{\tau_2} < \infty\}\notag\\
        \subset& \{ A^{\tau_1}_{\tau_2} = E_1\} \label{res3}
    \end{align} Since $X_{\tau^S} \in \{ L^{\tau_1}_x, R^{\tau_1}_x\}$ and $X_{\tau^T} \in \{ L^{\tau_2}_x, R^{\tau_2}_x\}$, the assumption yields
    \begin{align}
        \{ \tau_1 = \tau_2 \} \cap \{ \tau_1 = \tau^S \} \cap \{ \tau_2 = \tau^T\} \subset \{\tau^S = \tau^T \} = \varnothing. \label{res4}
    \end{align} For the final estimate, we start by applying \eqref{res1}, \eqref{res2}, \eqref{res3} and \eqref{res4}. Then we find that the conditional expectations vanish since $A^{\tau_1}_{\tau^T}$ and $A^{\tau_2}_{\tau^S}$ are $\f^X_\infty$-measurable and thus independent from $\sigma(E_1,E_2)$ and $A^{\tau_1}_{\tau_2}$ is $\sigma(\f^X_\infty \cup \sigma(E_2))$-measurable and thus independent from $\sigma(E_1)$ while $E_1,E_2$ put no weight in singletons.
    \begin{align}
        &\PR_x(\tau_1 = \tau_2 ) \notag\\
        = &\PR_x( \tau_1 =\tau_2, \tau_1 < \tau^S, \tau_2< \tau^T)+ \PR_x( \tau_1 =\tau_2, \tau_1 = \tau^S, \tau_2< \tau^T)\notag\\
        &+ \PR_x( \tau_1 =\tau_2, \tau_1 < \tau^S, \tau_2= \tau^T)+ \PR_x( \tau_1 =\tau_2, \tau_1 = \tau^S, \tau_2= \tau^T)\notag\\
        \le& \Ex_x[\ind{A^{\tau_1}_{\tau^T} = E_1 }]+ \Ex_x[\ind{A^{\tau_2}_{\tau^S} = E_2 }] + \Ex_x[\ind{A^{\tau_1}_{\tau^2} = E_1 }] \notag\\
        =&  \Ex_x[\Ex_x[\ind{A^{\tau_1}_{\tau^T} = E_1 }\vert \f^X_\infty] ] +  \Ex_x[\Ex_x[\ind{A^{\tau_2}_{\tau^S} = E_2 } \vert \f^X_\infty]] +  \Ex_x[\Ex_x[\ind{A^{\tau_1}_{\tau^2} = E_1 } \vert \sigma(\f^X_\infty \cup \sigma(E_2))]]\notag\\
        =&0 \label{res5}
    \end{align} To show \eqref{qi2} and \eqref{qi3}, we can proceed analogously replacing \eqref{res4} by
    \begin{align*}
        &\{\tau_1 = \tau_2\} \cap \{\tau_1 < \eta^{R^{\tau_1}_x}\} \cap \{ \tau_1 = \tau^S \} \cap \{ \tau_2 = \tau^T\} \subset \{\tau^S = \tau^T \} \cap \{\tau_1 < \eta^{L^{\tau_1}_x}\}= \varnothing \,\,\, \text{or}\\
        &\{\tau_1 = \tau_2\} \cap \{\tau_1 < \eta^{L^{\tau_1}_x}\} \cap \{ \tau_1 = \tau^S \} \cap \{ \tau_2 = \tau^T\} \subset \{\tau^S = \tau^T \} \cap \{\tau_1 < \eta^{L^{\tau_1}_x}\}= \varnothing
    \end{align*} respectively. Instead of \eqref{res5} we have
    \begin{align*}
        &\PR_x(\tau_1 = \tau_2 ,\tau_1 < \eta^{L^{\tau_1}_x}) \\
        \le &\PR_x( \tau_1 =\tau_2, \tau_1 < \tau^S, \tau_2< \tau^T)+ \PR_x( \tau_1 =\tau_2, \tau_1 = \tau^S, \tau_2< \tau^T)\\
        &+ \PR_x( \tau_1 =\tau_2, \tau_1 < \tau^S, \tau_2= \tau^T)+ \PR_x( \tau_1 =\tau_2, \tau_1 = \tau^S, \tau_2= \tau^T,\tau_1 < \eta^{L^{\tau_1}_x}) =...=0
    \end{align*}
    or
    \begin{align*}
        &\PR_x(\tau_1 = \tau_2 ,\tau_1 < \eta^{R^{\tau_1}_x}) \\
        \le &\PR_x( \tau_1 =\tau_2, \tau_1 < \tau^S, \tau_2< \tau^T)+ \PR_x( \tau_1 =\tau_2, \tau_1 = \tau^S, \tau_2< \tau^T)\\
        &+ \PR_x( \tau_1 =\tau_2, \tau_1 < \tau^S, \tau_2= \tau^T)+ \PR_x( \tau_1 =\tau_2, \tau_1 = \tau^S, \tau_2= \tau^T,\tau_1 < \eta^{R^{\tau_1}_x}) =...=0
    \end{align*} respectively to finish the proof.
\end{proof}

\begin{lemma}\label{MC}($\mathcal{M}$-continuity of $J$)\\
    Assume \eqref{A} and either \eqref{B1} or \eqref{C1}, \eqref{C2}, \eqref{C3}. Let $U \subset I$ finite, $x \in I$, $(\tau^1_n,\tau^2_n)_{n \in \N} \in (\mathcal{Z}_U(E_1) \times \mathcal{Z}_U(E_2))^{\N}$ and $(\tau^1,\tau^2) \in \mathcal{Z}_U(E_1) \times \mathcal{Z}_U(E_2)$ with $(\tau^1_n,\tau^2_n) \stackrel{n \to \infty}{\longrightarrow} (\tau^1,\tau^2)$ in the $\mathcal{M}$-topology. 
    \begin{enumerate}[(i)]
        \item \label{MC1} If $L^{\tau^1}_x \neq L^{\tau^2}_x$ whenever $f_1(L^{\tau^1}_x) > g_1(L^{\tau^1}_x)$ and $R^{\tau^1}_x \neq R^{\tau^2}_x$ whenever $f_1(R^{\tau^1}_x) > g_1(R^{\tau^1}_x)$ then $$J^1(x,\tau^1_n, \tau^2_n) \stackrel{n \to \infty}{\longrightarrow} J^1(x,\tau^1,\tau^2)$$ and $J^1(x,\bigcdot, \bigcdot)$ is continuous in $(\tau^1,\tau^2)$.
        \item \label{MC2}If $L^{\tau^1}_x \neq L^{\tau^2}_x$ whenever $f_2(L^{\tau^2}_x) > g_2(L^{\tau^2}_x)$ and $R^{\tau^1}_x \neq R^{\tau^2}_x$ whenever $f_2(R^{\tau^2}_x) > g_2(R^{\tau^2}_x)$ then $$J^2(x,\tau^1_n, \tau^2_n) \stackrel{n \to \infty}{\longrightarrow} J^2(x,\tau^1, \tau^2)$$ and $J^2(x,\bigcdot, \bigcdot)$ is continuous in $(\tau^1,\tau^2)$.
    \end{enumerate}  
\end{lemma}

\begin{remark}\label{remMC}
    Under assumption \eqref{B1} we set
    \begin{align*}
        &g_i(-\infty):=g^i(\inf I) &&g_i(\infty):=g^i(\sup I)\\
        &f_i(-\infty):=f^i(\inf I) &&f_i(\infty):=f^i(\sup I)
    \end{align*} and under \eqref{C1}, \eqref{C2}, \eqref{C3} we set $g_i(-\infty)=g_i(\infty)=f_i(-\infty)=f_i(\infty)=0$
    for $i=1,2$. Like that Lemma \ref{MC} allows for $L^{\tau^1}_x = L^{\tau^2}_x\not\in I$, $R^{\tau^1}_x = R^{\tau^2}_x\not \in I$ or $\PR_x( \tau^1 \wedge \tau^2= \infty) >0$. Additionally Assumption \eqref{A} could be dropped if we replace the $f_i(L^{\tau^1}_x) > g_i(L^{\tau^1}_x)$ and $f_i(R^{\tau^1}_x) > g_i(R^{\tau^1}_x)$, $i=1,2$ part by the more general formulations $f_i(L^{\tau^1}_x) \neq g_i(L^{\tau^1}_x)$ and $f_i(R^{\tau^1}_x) \neq g_i(R^{\tau^1}_x)$, $i=1,2$
\end{remark}

\begin{proof} (of Lemma \ref{MC})\\
    We only show \eqref{MC1} as \eqref{MC2} can be shown in exactly the same way. W.l.o.g.\ we only treat the case $L^{\tau^1}_x = L^{\tau^2}_x$, $R^{\tau^1}_x \neq R^{\tau^2}_x$. The other cases $L^{\tau^1}_x \neq L^{\tau^2}_x$, $R^{\tau^1}_x = R^{\tau^2}_x$ and $L^{\tau^1}_x = L^{\tau^2}_x$, $R^{\tau^1}_x = R^{\tau^2}_x$ are analogous using the appropriate parts of Lemma \ref{qi}. We have 
    \begin{align*}
        &J^1(x,\tau^1,\tau^2) 
        = \Ex_x[\ind{\tau^1 \le \tau^2} e^{-r_1 \tau^1} g(X_{\tau^1}) +\ind{\tau^1 > \tau^2} e^{-r_1 \tau^2} f_1(X_{\tau^2})]\\
        =&\Ex_x[(\ind{\tau^1 \le \tau^2} e^{-r_1 \tau^1} g(X_{\tau^1}) +\ind{\tau^1 > \tau^2} e^{-r_1 \tau^2} f_1(X_{\tau^2})) (\ind{\eta^{L^{\tau^1}_x} \le \tau^1 \wedge \tau^2} +\ind{\eta^{L^{\tau^1}_x} > \tau^1 \wedge \tau^2})],\\
        &J^1(x,\tau^1,\tau^2) 
        = \Ex_x[\ind{\tau^1_n \le \tau^2_n} e^{-r_1 \tau^1_n} g(X_{\tau^1_n}) +\ind{\tau^1_n > \tau^2_n} e^{-r_1 \tau^2_n} f_1(X_{\tau^2_n})]\\
        =&\Ex_x[(\ind{\tau^1_n \le \tau^2_n} e^{-r_1 \tau^1_n} g(X_{\tau^1_n}) +\ind{\tau^1_n > \tau^2_n} e^{-r_1 \tau^2_n} f_1(X_{\tau^2_n})) (\ind{\eta^{L^{\tau^1}_x} \le \tau^1 \wedge \tau^2} +\ind{\eta^{L^{\tau^1}_x} > \tau^1 \wedge \tau^2})].
    \end{align*} By Lemma \ref{qi} \eqref{qi3}, we have $\PR_x( \eta^{L^{\tau^1}_x} > \tau^1 \wedge \tau^2, \tau^1 = \tau^2) =0$. Together with Lemma \ref{mcont} \eqref{mcont3}, \eqref{mcont4}
    \begin{align*}
        &(\ind{\tau^1_n \le \tau^2_n} e^{-r_1 \tau^1_n} g(X_{\tau^1_n}) +\ind{\tau^1_n > \tau^2_n} e^{-r_1 \tau^2_n} f_1(X_{\tau^2_n})) \ind{\eta^{L^{\tau^1}_x} >  \tau^1 \wedge \tau^2} \\
        \stackrel{n \to \infty}{\longrightarrow}& (\ind{\tau^1 \le \tau^2} e^{-r_1 \tau^1} g(X_{\tau^1}) +\ind{\tau^1 > \tau^2} e^{-r_1 \tau^2} f_1(X_{\tau^2})) \ind{\eta^{L^{\tau^1}_x} >  \tau^1 \wedge \tau^2} \quad \PR_x\text{-a.s.}
    \end{align*} 
    The assumption yields
    \begin{align*}
        &\ind{\eta^{L^{\tau^1}_x} \le \tau^1 \wedge \tau^2 }\ind{\tau^1 = \tau^2 } e^{-r_1 \tau^1} g_1(X_{\tau^1})
        =\ind{\eta^{L^{\tau^1}_x} \le \tau^1 \wedge \tau^2 }\ind{\tau^1 = \tau^2 } e^{-r_1 \tau^1} g_1(L^{\tau^1}_x)\\
        =&\ind{\eta^{L^{\tau^1}_x} \le \tau^1 \wedge \tau^2 }\ind{\tau^1 = \tau^2 } e^{-r_1 \tau^2} f_1(L^{\tau^1}_x)
        =\ind{\eta^{L^{\tau^1}_x} \le \tau^1 \wedge \tau^2 }\ind{\tau^1 = \tau^2 } e^{-r_1 \tau^2} f_1(X_{\tau^2}).
    \end{align*} We infer
    \begin{align*}
        &(\ind{\tau^1_n \le \tau^2_n} e^{-r_1 \tau^1_n} g_1(X_{\tau^1_n}) +\ind{\tau^1_n > \tau^2_n} e^{-r_1 \tau^2_n} f_1(X_{\tau^2_n})) \ind{\eta^{L^{\tau^1}_x} \le  \tau^1 \wedge \tau^2} \\
        \stackrel{n \to \infty}{\longrightarrow}& (\ind{\tau^1 \le \tau^2} e^{-r_1 \tau^1} g_1(X_{\tau^1}) +\ind{\tau^1 > \tau^2} e^{-r_1 \tau^2} f_1(X_{\tau^2})) \ind{\eta^{L^{\tau^1}_x} \le  \tau^1 \wedge \tau^2} \quad \PR_x\text{-a.s.}
    \end{align*} The claim follows via dominated convergence using that $U$ is assumed finite, so $f_1\vert_U$ and $g_1\vert_U$ are bounded. Since the $\mathcal{M}$-topology is metric this means $J^1(x, \bigcdot, \bigcdot)$ is continuous in $(\tau^1,\tau^2)$.
\end{proof}

\begin{definition}(best response mapping)\\
    Let $U \subset I$ finite. We call the mapping $\Tilde{\Phi}:=\Tilde{\Phi}_U: \mathcal{Z}_U(E_1)\times \mathcal{Z}_U(E_2)\to \mathcal{P}(\mathcal{Z}_U(E_1)\times \mathcal{Z}_U(E_2))$, where $\mathcal{P}$ denotes the power set, be given by
    \begin{align}
        (\tau_1,\tau_2) \mapsto \bigg\{&(\tau_1',\tau_2')\in  \mathcal{Z}_U(E_1)\times\mathcal{Z}_U(E_2) \bigg\vert \notag\\
        &J^1(x,\tau_1',\tau_2) =V^1_U(x,\tau_2) := \sup_{\tau_1''\in\mathcal{T}_U\cup \mathcal{Z}_U(E_1) } J^1(x,\tau_1'',\tau_2) \,\,\forall x\in U, \label{equi}\\
        &J^2(x,\tau_1,\tau_2') =V^2_U(x,\tau_1):= \sup_{\tau_2''\in\mathcal{T}_U\cup \mathcal{Z}_U(E_2) } J^2(x,\tau_1,\tau_2'')\,\,\forall x\in U, \label{equi2}\\
        &\forall x \in U: f_1(x)\neq g_1(x), \iota_1(\tau_2)(\{x\})=\infty \Longrightarrow \iota_1(\tau_1')(\{x\})=0 ,\label{nonstop}\\
        &\forall x \in U: f_2(x)\neq g_2(x), \iota_1(\tau_1)(\{x\})=\infty \Longrightarrow \iota_1(\tau_2')(\{x\})=0\bigg\}\label{nonstop2}
    \end{align} \textit{best response (mapping)}. Based on the notation from Remark \ref{remM} \eqref{remM3} we define the so called \textit{embedded best response (mapping)} $\Phi: [0,1]^{U^2}\to \mathcal{P}([0,1]^{U^2})$ via $\Phi:= \Phi_U:=  \overline{\iota} \circ \Tilde{\Phi}_U \circ \overline{\iota}^{-1}$.
\end{definition}

\begin{remark}
    The best response mapping maps a pair of Markovian randomized times  $(\tau_1,\tau_2) \in\mathcal{Z}_U(E_1)\times \mathcal{Z}_U(E_2)$ to a set of pairs of Markovian randomized times. The first components $\tau_1'$ of the pairs $(\tau_1', \tau_2') \in \Tilde{\Phi}(\tau_1,\tau_2)$ are determined only by $\tau_2$ through the conditions \eqref{equi} and \eqref{nonstop}, while the second components $\tau_2'$ are determined only by $\tau_1$ through \eqref{equi2} and \eqref{nonstop2}. The conditions \eqref{equi}, \eqref{equi2} ensure that $\tau_1'$ resp.\ $\tau_2'$ maximize $J^1$ resp.\ $J^2$ given $\tau_2$ resp. $\tau_1$ at every $x$. In that sense $\tau_1', \tau_2'$ are best responses to the game strategies $\tau_2,\tau_1$. The conditions \eqref{nonstop} and \eqref{nonstop2} have more of a technical nature. They can be interpreted as the requirement that if the opponent stops at the first entry time into any given point and it is beneficial for us to keep going in terms of payoff, then we will do so by choosing a stopping rate of 0 at that point even though any finite rate would yield the same profit.
\end{remark}

\begin{theorem}\label{ex1}
    Assume \eqref{A}, \eqref{B1}, \eqref{B2}. Further let $U\subset I$ finite. Then, there exists a $U$-equilibrium with the properties \eqref{nonstop} and \eqref{nonstop2}.
\end{theorem}

The theorem is shown by applying Kakutani's fixed point theorem \cite[Lemma 20.1]{osborne_rubinstein} to the embedded best response $\Phi$. One of the most crucial  prerequisites for Kakutani's fixed point theorem is the closedness of the graph of the (embedded) best response which will be established in Lemma \ref{lemma_closed}.

\begin{lemma}\label{lemma_closed}(closedness of the graph of the best response)\\
Assume \eqref{A}, \eqref{B1}, \eqref{B2}. Let $U\subset I$ be finite. The graph of the best response mapping $\tilde \Phi$ is a closed subset of $(\mathcal{Z}_U(E_1)\times \mathcal{Z}_U(E_2))\times (\mathcal{Z}_U(E_1)\times \mathcal{Z}_U(E_2))$ (equipped with the product topology of $\mathcal{M}$-topologies on all components).
\end{lemma}

\begin{proof}
    Fix a $(\tau_1,\tau_2),(\tau_1',\tau_2')\in \mathcal{Z}_U(E_1)\times \mathcal{Z}_U(E_2)$ such that $(\tau_1',\tau_2') \not \in \Tilde{\Phi}(\tau_1,\tau_2)$. We have to find a neighborhood $V_1 \times V_2 \times W_1\times W_2 \subset (\mathcal{Z}_U(E_1)\times \mathcal{Z}_U(E_2)) \times (\mathcal{Z}_U(E_1)\times \mathcal{Z}_U(E_2)) $ of $(\tau_1,\tau_2,\tau_1',\tau_2')$ such that for all $(\sigma_1,\sigma_2,\sigma_1',\sigma_2') \in V_1 \times V_2 \times W_1\times W_2 $ we have $(\sigma_1',\sigma_2')\not \in \Tilde{\Phi}(\sigma_1,\sigma_2)$. The assumption $(\tau_1',\tau_2') \not \in \Tilde{\Phi}(\tau_1,\tau_2)$ implies that $(\tau_1,\tau_2,\tau_1',\tau_2')$ does not satisfy one of the conditions \eqref{equi}, \eqref{equi2},\eqref{nonstop}, or \eqref{nonstop2}. By symmetry, we may assume w.l.o.g.\ that either \eqref{equi} or \eqref{nonstop} is not satisfied and construct $V_2$ and $W_1$ such that \eqref{equi} or \eqref{nonstop} are violated for any choice of $\sigma_2\in V_2$, $\sigma_1' \in W_1$. 
    
    We start by assuming that \eqref{nonstop} does not hold for $ \tau_1', \tau_2$. This means, there is some $x\in U$ such that $f_1(x)>g_1(x)$, $\iota_1(\tau_2)(\{x\}) = \infty$ and $\iota_1(\tau_1')(\{x\}) > 0$. Let $s^z=(s^z_u)_{u\in U}\in [0,\infty]^U$ be given by $s_u=0$ for all $u\neq x$ and $s_x=z\in [0,\infty]$. By Lemma \ref{mcont} $\iota^{-1}(s^z)\to \iota^{-1}(s^\infty)$ $\PR_x$-a.s.\ for $z\to \infty$. Additionally, for all $\sigma_2\in \mathcal{Z}_U(E_2)$ with $\iota_1(\sigma_2)(\{x\})\ge z$ we have $\sigma_2 \le \iota^{-1}(s^z)$. Combined with $\PR_x(\iota^{-1}(s^\infty)=0)=1$, we find $\sigma_2\to 0$ $\PR_x$-a.s. for $\iota_1(\sigma_2)(\{x\}) \to \infty$. Now, by dominated convergence
    \begin{align*}
        J^1(x,\eta^{U\setminus\{x\}},\sigma_2)=&\Ex_x[e^{-r_1\sigma_2}f_1(X_{\sigma_2})\ind{\sigma_2<\eta^{U\setminus \{x\}}}+ e^{-r_1\eta^{U\setminus \{x\}}}g_1(X_{\eta^{U\setminus \{x\}}})\ind{\sigma_2\ge\eta^{U\setminus \{x\}}}] 
        \to f_1(x) 
    \end{align*} for $\iota_1(\sigma_2)(\{x\})\to \infty$. Thus, there is $c>0$ such that $J^1(x,\eta^{U\setminus\{x\}},\sigma_2)> g_1(x)$ for all $\sigma_2\in \mathcal{Z}_U(E_2)$ with $\iota_1(\sigma_2)(\{x\})>c $. We set $ V_2:=\iota^{-1}(\{(s_u)_{u\in U}\in [0,\infty]^U: s_x>c\})$ and $W_1:=\iota^{-1}(\{ (s_u)_{u\in U}\in [0,\infty]^U: s_x\in (0,\infty)\})$. Now let $(\sigma_2,\sigma_1')\in V_2\times W_1$ arbitrary. If $\iota_1(\sigma_2)(\{x\})=\infty$, then $(\sigma_1',\sigma_2)$ violates \eqref{nonstop}. Next we show that \eqref{equi} is violated by all $(\sigma_1',\sigma_2)\in V_2\times W_1$ with $\iota_1(\sigma_2)(\{x\})<\infty $. Assuming the contrary, due to the definitions of $V_2$ and $W_1$, Lemma \ref{indiff} \eqref{weakindiff} yields $J^1(x,\sigma_1',\sigma_2)=V_U^1(x,\sigma_2)=g_1(x)$. However, by choice of $c$, since $\eta^{U\setminus\{x\}}=\tau^{(U\setminus\{x\})^c}\in \mathcal{Z}_U$ this contradicts $V^1_U(x,\tau_2) \ge J^1(x,\eta^{U\setminus\{x\}},\tau_2) >g_1(x)$.
    
    
    Now we assume that \eqref{equi} does not hold for $ \tau_1', \tau_2$. If the pair $(\tau_1', \tau_2)$ satisfies the preconditions of Lemma \ref{MC} \eqref{MC1}, the function $J^1(x,\bigcdot,\bigcdot)$ is continuous in $(\tau_1', \tau_2)$, so there is a neighborhood $V_2 \times W_1$ such that $ \vert J^1(x,\tau_1',\tau_2) -J^1(x,\sigma_1',\sigma_2) \vert  < V^1_U(x,\tau_2) - J(x,\tau_1', \tau_2)>0$, i.e.\ $(\sigma_1', \sigma_2') \not \in \Tilde{\Phi}(\sigma_1,\sigma_2)$ for all $(\sigma_1,\sigma_2,\sigma_1', \sigma_2') \in \mathcal{Z}_U(E_1)\times V_2 \times W_1 \times \mathcal{Z}_U(E_2)$. If $(\tau_1', \tau_2)$ does not satisfy the preconditions of Lemma \ref{MC} \eqref{MC1} either $L^{\tau_1'}_x = L^{\tau_2}_x$ and $f_1(L^{\tau_2}_x) > g_1(L^{\tau_2}_x)$ or $R^{\tau_1'}_x = R^{\tau_2}_x$ and $f_1(R^{\tau_2}_x) > g_1(R^{\tau_2}_x)$. In particular, by Assumption \eqref{B2} we must have $L^{\tau_1'}_x = L^{\tau_2}_x\in I$ ($R^{\tau_1'}_x = R^{\tau_2}_x\in I$) for $f_1(L^{\tau_2}_x) > g_1(L^{\tau_2}_x)$ ($f_1(R^{\tau_2}_x) > g_1(R^{\tau_2}_x)$) to hold. This means \eqref{nonstop} is violated at either $L^{\tau_2}_x$ or $R^{\tau_2}_x$.
\end{proof}

\begin{proof}(of Theorem \ref{ex1})\\
We now show that the embedded best response $\Phi$ satisfies the preconditions of Kakutani's fixed point theorem, cf.\ \cite[Lemma 20.1]{osborne_rubinstein} to show that it possesses a fixed point. First note that $[0,1]^{U^2}$ clearly satisfies the compactness requirement of the fixed point theorem. Moreover, by Lemma \ref{lemma_closed} the graph of $\tilde \Phi$ is closed. This property immediately translates to $\Phi$ as $\overline{\iota}$ is homeomorphic by definition. To check the other preconditions, let $s=(s_1,s_2)\in [0,1]^{U^2}$. By Lemma \ref{indiff} \eqref{indiff2}, the set $\Phi(s)$ is non-empty. 
We finish by showing the convexity of the image sets of $\Phi$. The best responses are characterized through Lemma \ref{indiff} \eqref{weakindiff2}. Let $\tau_i = \iota^{-1}(s_i)$. We denote the set $T_j$, $j=1,2$ from Lemma \ref{indiff} \eqref{weakindiff2} corresponding to $\tau_i$, $i=1,2$ by $T_j(\tau_i)$. By the definition of $\tilde \Phi$ and $T_1$ and the lemma we find $$\Phi(s)= \overline{\iota}(\tilde\Phi(\tau_1,\tau_2))=\overline{\iota}( T_1(\tau_2)\times T_1(\tau_1))= \overline{\iota}(T_2(\tau_2)\times T_2(\tau_1)).$$
Inspecting the definition of $T_2$, the right hand side is clearly convex.
\end{proof}

\section{Existence of equilibria: Absorbing boundary} \label{sec:final_proof}


In this section we derive a general existence result for Markovian equilibria in the sense of Definition \ref{defequi} for absorbing boundaries. Our results from Section \ref{sectopo} and Section \ref{secdisc} allow us to construct a candidate quite easily. All objects introduced in the following construction of the candidate equilibrium are fixed throughout the rest of this section. For this entire section we assume \eqref{A}, \eqref{B1}, \eqref{B2}. Those assumptions were already used to establish the main results of Section \ref{secdisc}. Here we need \eqref{A} in the proof of Lemma \ref{M1cont} \eqref{M1cont1} in order to apply Lemma \ref{genport} \eqref{genport1}. \eqref{B1} and \eqref{B2} will be used in order to construct the equilibrium candidate.

Let $(U_n)_{n \in \N}$ be a sequence of finite subsets of $I$ and set $U:=I^\circ$. We assume that $(U_n)_{n \in \N}$ is increasing in the sense that $ U_n \subset U_{n+1}$ for all $n\in \N$, $\partial I\subset U_0$ and that $\bigcup_{n\in \N} U_n$ is dense in $I$. By Theorem \ref{ex1}, there is a $U_n$-equilibrium $(\tau^1_n, \tau^2_n) \in \mathcal{Z}_{U_n}(E_1)\times \mathcal{Z}_{U_n}(E_2)$ for each $n \in \N$, i.e. for all $x \in I$ and all $n\in \N$
\begin{align}
    &J^1(x, \tau^1_n, \tau^2_n) = V^1_{U_n}(x,\tau^2_n),\label{Unequi1}\\
    &J^2(x, \tau^1_n, \tau^2_n) = V^2_{U_n}(x,\tau^1_n)\label{Unequi2}\quad \text{and}\\
    &\forall x \in U_n: f_1(x)\neq g_1(x), \iota_1(\tau^2_n)(\{x\})=\infty \Longrightarrow \iota_1(\tau^1_n)(\{x\})=0 ,\label{as3}\\
    &\forall x \in U_n: f_2(x)\neq g_2(x), \iota_1(\tau^1_n)(\{x\})=\infty \Longrightarrow \iota_1(\tau^2_n)(\{x\})=0\label{as3_2}.
\end{align}
By \eqref{B1} and \eqref{B2}, the stopping times $(\tau^1_n\wedge \eta^{\partial I}, \tau^2_n\wedge\eta^{\partial I}) \in (\mathcal{Z}_{U_n}(E_1)\times \mathcal{Z}_{U_n}(E_2))\cap (\mathcal{Z}^U(E_1) \times \mathcal{Z}^U(E_2))$ do still constitute $U_n$-equilibria for all $n\in \N$. Remember that all stopping times in $\mathcal{Z}^U(E_1)$ must stop outside $U$, while stopping times in $\mathcal{Z}_{U_n}$ can only stop inside $U_n$. Form now on we set $(\tau^1_n, \tau^2_n):=(\tau^1_n\wedge \eta^{\partial I}, \tau^2_n\wedge\eta^{\partial I})$ for all $n\in \N$ w.l.o.g.

By Theorem \ref{comp2} the sequence $((\tau^1_n,\tau^2_n))_{n \in \N}$ possesses  a subsequence that converges in the $\mathcal{M}^1$-product topology to a limit $(\tau^1,\tau^2) \in \mathcal{Z}^U(E_1) \times \mathcal{Z}^U(E_2)$. W.l.o.g., we assume $(\tau^1_n,\tau^2_n)\stackrel{n\to \infty}{\longrightarrow} (\tau^1,\tau^2)$ in the $\mathcal{M}^1$-product topology. 
The rest of this section is dedicated to proving that $(\tau^1,\tau^2)$ is a Markov perfect $U$-equilibrium.

The next lemma investigates the (semi-)continuity of $J^1$, $J^2$ along sequences which converge in the $\mathcal{M}^1$-product topology. The lemma is one of the core results in the proof of Theorem \ref{absbdryexists} and also for Theorem \ref{natbdryexists}.

\begin{lemma}\label{M1cont} Let $U\subset I$ be open with compact closure $\overline{U}\subset I$ (closure taken in $\R$), $x\in U$, $\sigma^1 \in \mathcal{Z}(E_1)$, $\sigma^2\in \mathcal{Z}(E_2)$ and $((\tilde \tau^1_n,\tilde \tau^2_n))_{n\in \N}=((\tau^{\tilde D^1_n, \tilde \lambda^1_n}, \tau^{\tilde D^2_n, \tilde \lambda^2_n}))_{n\in \N}$ a sequence in $\mathcal{Z}^U(E_1)\times \mathcal{Z}^U(E_2)$ that converges to some $(\tilde \tau^1,\tilde \tau^2)=(\tau^{\tilde D^1, \tilde \lambda^1}, \tau^{\tilde D^2, \tilde \lambda^2})\in \mathcal{Z}^U(E_1)\times \mathcal{Z}^U(E_2)$ in the $\mathcal{M}^1$-product topology. For all $n\in \N$ and all $z\in I$ assume that
\begin{align}
    &J^1(z, \tilde \tau^1_n,\tilde \tau^2_n) \ge J^1(z, \tau^{U} ,\tilde \tau^2_n), \quad J^2(y, \tilde \tau^1_n,\tilde \tau^2_n) \ge J^2(y, \tilde \tau^1_n,\tau^{U}), \label{aspt:1}\\
    &\supp(\tilde \lambda^1_n)\cap \tilde D^2_n\subset \{y\in U: g_1(y)\ge J^1(y, \tilde \tau^1_n,\tilde \tau^2_n)\},\label{aspt:2}\\
    &\supp(\tilde \lambda^2_n)\cap \tilde D^1_n\subset \{y\in U: g_2(y)\ge J^2(y, \tilde \tau^1_n,\tilde \tau^2_n)\}, \notag\\
    &\supp(\tilde \lambda^1_n) \cap (\tilde D^2_n)^c\cap\{y\in I:f_1(y)>g_1(y)\} =\varnothing,\label{aspt:3} \\ 
    &\supp(\tilde \lambda^2_n) \cap (\tilde D^1_n)^c\cap\{y\in I:f_2(y)>g_2(y)\}=\varnothing.\notag
\end{align} Then, we have
\begin{enumerate}[(i)]
    \item \label{M1cont1}\begin{align*}
            &\liminf_{n \to \infty} J^1(x,\sigma^1,\tilde \tau^2_n) \ge J^1(x,\sigma^1, \tilde \tau^2),\\
            &\liminf_{n \to \infty} J^2(x,\tilde \tau^1_n,\sigma^2) \ge J^2(x,\tilde \tau^1, \sigma^2)\quad \text{and}
\end{align*}
\item \label{M1cont2}\begin{align*}
    &J^1(x, \tilde \tau^1,\tilde  \tau^2) = \lim_{n\to \infty} J^1(x,\tilde \tau^1_n,\tilde \tau^2_n),\\
    &J^2(x, \tilde \tau^1,\tilde  \tau^2) = \lim_{n\to \infty} J^2(x,\tilde \tau^1_n,\tilde \tau^2_n).
\end{align*}  
\item \label{M1cont3} The requirement \eqref{aspt:3} prevails in the limit, i.e.\
\begin{align*}
        &\supp(\tilde \lambda^1)\cap (\tilde D^2)^c \cap \{y\in I : f_1(y)>g_1(y)\}=\varnothing \quad\text{and} \\
        &\supp(\tilde \lambda^2)\cap (\tilde D^1)^c \cap \{y\in I : f_2(y)>g_2(y)\}=\varnothing.
    \end{align*}
\end{enumerate}
\end{lemma}

\begin{proof} \eqref{M1cont1} W.l.o.g.\ we only show the first estimate. Proposition \ref{weakconv} yields $\PR_x^{(\tilde \tau^2_n,X_{\tilde \tau^2_n})} \stackrel{n\to \infty}{\longrightarrow}\PR_x^{(\tilde \tau^2,X_{\tilde \tau^2})}$ weakly. Slutsky's theorem and Cramér Wold theorem provide $\PR_x^{(\sigma^1,X_{\sigma^1},\tilde \tau^2_n,X_{\tilde \tau^2_n})} \stackrel{n\to \infty}{\longrightarrow}\PR_x^{(\sigma^1,X_{\sigma^1},\tilde \tau^2,X_{\tilde \tau^2})}$ weakly. Since we assumed $g_1\le f_1$, the function $h:[0,\infty]\times I\times [0,\infty]\times I \to \R $, $(y_1,y_2,y_3,y_4) \mapsto \ind{y_1 \le y_3}e^{-r_1 y_1} g_1(y_2)+\ind{y_1 > y_3} e^{-r_1 y_3} f_1(y_4)$ is lower semicontinuous. By construction, $J^1(x,\sigma^1,\tilde \tau^2_n)= \int h d\PR_x^{(\sigma^1,X_{\sigma^1},\tilde \tau^2_n,X_{\tilde \tau^2_n})} $ and $J^1(x,\sigma^1, \tilde \tau^2)=\int h d\PR_x^{(\sigma^1,X_{\sigma^1},\tilde \tau^2,X_{\tilde \tau^2})}$, thus Lemma \ref{genport} \eqref{genport1} yields the claim.    

\eqref{M1cont2} W.l.o.g. we only show the first statement. Once again by Proposition \ref{weakconv}, $\PR_x^{(\tilde \tau^1_n,X_{\tilde \tau^1_n},\tilde \tau^2_n,X_{\tilde \tau^2_n})} \stackrel{n\to \infty}{\longrightarrow}\PR_x^{(\tilde \tau^1,X_{\tilde \tau^1},\tilde \tau^2,X_{\tilde \tau^2})}$ weakly. As above $ J^1(x, \tilde \tau^1, \tilde \tau^2)= \int hd\PR_x^{(\tilde \tau^1,X_{\tilde \tau^1},\tilde \tau^2,X_{\tilde \tau^2})}$ and $J^2(x,\tilde \tau^1_n,\tilde \tau^2_n)=\int hd \PR_x^{(\tilde \tau^1_n,X_{\tilde \tau^1_n},\tilde \tau^2_n,X_{\tilde \tau^2_n})} $. By Lemma \ref{genport} \eqref{genport2}, it is sufficient to show 
\begin{align*}
    \PR_x^{(\tilde \tau^1,X_{\tilde \tau^1},\tilde \tau^2,X_{\tilde \tau^2})}(\{(y_1,y_2,y_3,y_4)&\in [0,\infty]\times I\times [0,\infty]\times I :\\ &h \,\text{ is not continuous in }\,(y_1,y_2,y_3,y_4)\})=0.
\end{align*} By Lemma \ref{qi} \eqref{qi1}, this boils down to proving $L_x^{\tilde \tau^1}\neq L_x^{\tilde \tau^2}$ or $g_1(L_x^{\tilde \tau^1})=f_1(L_x^{\tilde \tau^1})$ and $R_x^{\tilde \tau^1} \neq R_x^{\tilde \tau^2}$ or $g_1(R_x^{\tilde \tau^1})=f_1(R_x^{\tilde \tau^1})$. W.l.o.g.\ we will only show $L_x^{\tilde \tau^1}\neq L_x^{\tilde \tau^2}$ provided $g_1(L_x^{\tilde \tau^1})<f_1(L_x^{\tilde \tau^1})$. We argue by contradiction and assume $L_x^{\tilde \tau^1}= L_x^{\tilde \tau^2}$. Let $ B\subset I$ open such that $L_x^{\tilde \tau^2}\in B$ and $\inf_{y\in B} f_1(z)> \sup_{y\in B}g_1(y)$. Let $\delta_1,\delta_2\in (0,1)$, $\eps\in (0,\infty)$ such that $$\delta_1 \delta_2 e^{-r_1\eps} \inf_{y\in B}f_1(y)>\sup_{y\in B}g_1(y).$$ By \cite[Sections 3.3., 4.2.]{ito1974diffusion}, there is some open neighborhood $B'\subset B$ of $L_x^{\tilde \tau^2}$ such that $\Ex_z[e^{-r_1\eta_{L_x^{\tilde \tau^2}}}\ind{\eta_{L_x^{\tilde \tau^2}}<\tilde \tau^B}]\ge \delta_1$ for all $z\in B'$. Since $$1=\PR_{L_x^{\tilde \tau^2}}((\tilde \tau^2,X_{\tilde \tau^2})\in [0,\eps)\times B)\le \liminf_{n\to \infty}\PR_{L_x^{\tilde \tau^2}}((\tilde \tau^2_n,X_{\tilde \tau^2_n})\in [0,\eps)\times B), $$ there is some $N\in \N$ such that $\PR_{L_x^{\tilde \tau^2}}((\tilde \tau^2_n,X_{\tilde \tau^2_n})\in [0,\eps)\times B)\ge \delta_2$ for all $n\ge N$. This yields
\begin{align*}
    \Ex_{L_x^{\tilde \tau^2}}[e^{-r_1\tilde \tau^2_n} f_1(X_{\tilde \tau^2_n})]
    \ge \Ex_{L_x^{\tilde \tau^2}}[\ind{(\tilde \tau^2_n,X_{\tilde \tau^2_n})\in [0,\eps)\times B}]e^{-r_1\eps}\inf_{y\in B}f_1(y)
    \ge \delta_2 e^{-r_1\eps}\inf_{y\in B}f_1(y)
\end{align*} for all $n\ge N$. Now 
\begin{align*}
    &J^1(z,\tau^{U},\tilde \tau^2_n)
    =\Ex_z[e^{-r_1\tilde \tau_2}f_1(X_{\tilde \tau^2_n})] \\
    =& \Ex_z[\ind{\tilde \tau^2_n<\eta^{L_x^{\tilde \tau^2}}\wedge \tau^{B}}e^{-r_1\tilde \tau_2}f_1(X_{\tilde \tau^2_n})]+\Ex_z[\ind{\tilde \tau^2_n \ge \eta^{L_x^{\tilde \tau^2}}\wedge \tau^{B}}e^{-r_1\tilde \tau_2}f_1(X_{\tilde \tau^2_n})] \\
    \ge& \Ex_z[\ind{\tilde \tau^2_n<\eta^{L_x^{\tilde \tau^2}}} \ind{\eta^{L_x^{\tilde \tau^2}}< \tilde \tau^B} e^{-r_1\eta^{L_x^{\tilde \tau^2}} }] \inf_{y\in B}f_1(y) \\
    &+ \Ex_z[\ind{\tilde \tau^2_n\ge \eta^{L_x^{\tilde \tau^2}}}\ind{\eta^{L_x^{\tilde \tau^2}}< \tilde \tau^B} e^{-r_1\eta^{L_x^{\tilde \tau^2}}} \Ex_z[e^{-r_1\theta_{ \eta^{L_x^{\tilde \tau^2}}}\circ \tilde \tau^2_n} f_1(X_{\theta_{ \eta^{L_x^{\tilde \tau^2}}}\circ \tilde \tau^2_n})  \vert \f_{\eta_{L_x^{\tilde \tau^2}}}]]\\
    =& \Ex_z[\ind{\tilde \tau^2_n<\eta^{L_x^{\tilde \tau^2}}} \ind{\eta^{L_x^{\tilde \tau^2}}< \tilde \tau^B} e^{-r_1\eta^{L_x^{\tilde \tau^2}} }] \inf_{y\in B}f_1(y) \\
    &+ \Ex_z[\ind{\tilde \tau^2_n\ge \eta^{L_x^{\tilde \tau^2}}}\ind{\eta^{L_x^{\tilde \tau^2}}< \tilde \tau^B} e^{-r_1\eta^{L_x^{\tilde \tau^2}}} \Ex_{X_{\eta_{L_x^{\tilde \tau^2}}}}[e^{-r_1\theta_{ \eta^{L_x^{\tilde \tau^2}}}\circ \tilde \tau^2_n} f_1(X_{\theta_{ \eta^{L_x^{\tilde \tau^2}}}\circ \tilde \tau^2_n})  ]]\\
    =& \Ex_z[\ind{\tilde \tau^2_n<\eta^{L_x^{\tilde \tau^2}}} \ind{\eta^{L_x^{\tilde \tau^2}}< \tilde \tau^B} e^{-r_1\eta^{L_x^{\tilde \tau^2}} }] \inf_{y\in B}f_1(y) \\
    &+ \Ex_z[\ind{\tilde \tau^2_n\ge \eta^{L_x^{\tilde \tau^2}}}\ind{\eta^{L_x^{\tilde \tau^2}}< \tilde \tau^B} e^{-r_1\eta^{L_x^{\tilde \tau^2}}} ]\Ex_{L_x^{\tilde \tau^2}}[e^{-r_1 \tilde \tau^2_n} f_1(X_{\tilde \tau^2_n}) ]\\
    \ge&  \Ex_z[\ind{\tilde \tau^2_n<\eta^{L_x^{\tilde \tau^2}}} \ind{\eta^{L_x^{\tilde \tau^2}}< \tilde \tau^B} e^{-r_1\eta^{L_x^{\tilde \tau^2}} }] \inf_{y\in B}f_1(y) +\Ex_z[\ind{\tilde \tau^2_n\ge \eta^{L_x^{\tilde \tau^2}}}\ind{\eta^{L_x^{\tilde \tau^2}}< \tilde \tau^B} e^{-r_1\eta^{L_x^{\tilde \tau^2}}} ]\delta_2 e^{-r_1\eps}\inf_{y\in B}f_1(y)\\
    \ge&\Ex_z[\ind{\eta^{L_x^{\tilde \tau^2}}< \tilde \tau^B} e^{-r_1\eta^{L_x^{\tilde \tau^2}} }] \delta_2 e^{-r_1\eps} \inf_{y\in B}f_1(y) \\
    \ge& \delta_1 \delta_2 e^{-r_1\eps} \inf_{y\in B}f_1(y)>\sup_{y\in B}g_1(y)
\end{align*} for all $n\ge N$ and all $z\in B'$. By Corollary \ref{corollary:LR} we can use \eqref{6fin} to obtain that there is some $n'\ge N$ and $z_0\in B'$ such that either $z_0 \in\supp(\tilde \lambda^1_{n'})\cap  (D^2_{n'})^c$,  $z_0\in \supp(\tilde \lambda^1_{n'})\cap D^2_{n'}$ or $z_0\in \tilde D^1_{n'}$. The first case is ruled out by \eqref{aspt:3}. In both other cases \eqref{aspt:2} yields.
\begin{align*}
     J^1(z_0,\tau^{U},\tau^2_{n'}) >\sup_{y\in B}g_1(y) \ge g_1(z_0) \ge J^1(z_0,\tilde \tau^1_{n'},\tilde \tau^2_{n'}),
\end{align*} which contradicts \eqref{aspt:1} to finish the proof.

\eqref{M1cont3}  W.l.o.g. we only show the first statement. Arguing by contradiction we assume that there is some $x\in \supp(\tilde \lambda^1) \cap (\tilde D^2)\cap \{y\in I: f_1(y)>g_1(y)\}$. This means $x=L^{\tilde \tau^2}_x$ and $L^{\tilde \tau^1}_x<x<R^{\tilde \tau^1}_x$. By Lemma \ref{lemma:ratesconverge} there is a subsequence $((\tilde \tau^1_{n_k},\tilde \tau^2_{n_k}))_{k\in \N}$ such that $\iota^1(\tilde \tau^1_{n_k})\stackrel{k\to \infty}{\longrightarrow} \iota^1(\tilde \tau^1)$ vaguely. Let $B''(\subset L^{\tilde \tau^1}_x, R^{\tilde \tau^1})$ be an open neighborhood of $x$. $x\in \supp(\tilde\lambda^1)$ yields
\begin{align*}
    0<\tilde\lambda^1(B'') \ge \liminf_{k\to \infty} \tilde\lambda^1_{n_k}(B'')
\end{align*} and thus there is $N'_{B''}\in \N$ such that $B''\cap \supp(\tilde\lambda^1_{n_k})\neq\varnothing $ for all $k\ge N'_{B''} $. In particular, if we consider the neighborhood $B'$ constructed in \eqref{M1cont2}, $B'\cap \supp(\tilde\lambda^1_{n_k})\neq\varnothing $ for all $k\ge N'_{B'}$. Let $k\ge N'_{B''}$ with $n_k\ge N'$ and $z_1\in B'\cap \supp(\tilde\lambda^1_{n_k})$. By \eqref{aspt:3} $z_1\in \tilde D^2_{n_k}$. But by choice of $B'$ and \eqref{aspt:2}
\begin{align*}
    J^1(z_1,\tau^{U},\tau^2_{n_k}) >\sup_{y\in B}g_1(y) \ge g_1(z_1) \ge J^1(z_1,\tilde \tau^1_{n_k},\tilde \tau^2_{n_k})
\end{align*} which contradicts \eqref{aspt:1}.
\end{proof}

\begin{theorem}\label{5.2}
    Assume \eqref{A}, \eqref{B1}, \eqref{B2}. The limit $(\tau^1,\tau^2)\in \mathcal{Z}^U(E_1) \times \mathcal{Z}^U(E_2)$ of a subsequence of $((\tau^1_n,\tau^2_n))_{n \in \N}\in (\mathcal{Z}^U(E_1) \times \mathcal{Z}^U(E_2))^{\N}$ constructed in the beginning of this section with $U:=I^\circ$ is a Markov perfect equilibrium. Additionally, if we write $(\tau^{D^1,\lambda^1},\tau^{D^2,\lambda^2})=(\tau^1,\tau^2)$ then 
    \begin{align*}
        &\supp(\lambda^1)\cap (D^2)^c \cap \{y\in I : f_1(y)>g_1(y)\}=\varnothing \quad\text{and} \\
        &\supp(\lambda^2)\cap ( D^1)^c \cap \{y\in I : f_2(y)>g_2(y)\}=\varnothing.
    \end{align*}
\end{theorem}

\begin{proof} Of \eqref{cond1} and \eqref{cond2} we only show \eqref{cond1} w.l.o.g. We start by checking the preconditions of Lemma \ref{M1cont} for the sequence $((\tau^1_n,\tau^2_n))_{n \in \N}\in (\mathcal{Z}^U(E_1) \times \mathcal{Z}^U(E_2))^{\N}$. Due to \eqref{Unequi1}, \eqref{Unequi2} the sequence $((\tau^1_n,\tau^2_n))_{n \in \N}\in (\mathcal{Z}^U(E_1) \times \mathcal{Z}^U(E_2))^{\N}$ satisfies \eqref{aspt:1}. Together with Lemma \ref{indiff} \eqref{weakindiff} we infer \eqref{aspt:2} from \eqref{Unequi1}, \eqref{Unequi2} as well. \eqref{as3} and \eqref{as3_2} yield \eqref{aspt:3}. 

Due to Lemma \ref{indiff} \eqref{indiff15} it suffices to show \eqref{cond1} for all $\sigma^1\in \bigcup_{k\in \N} \mathcal{Z}_{U_k}(E_1)$ instead of all $\sigma^1\in\ran(E_1)$. Now let $x\in I$ and $\sigma^1\in \mathcal{Z}_{U_m}$ for some $m\in \N$. By Lemma \ref{M1cont} and \eqref{Uequi1}
    \begin{align*}
        J^1(x,\tau^1,\tau^2) 
        &=\lim_{n\to \infty}J^1(x,\tau^1_n,\tau^2_n) 
        = \lim_{n\to \infty} V^1_{U_n}(x,\tau^2_n) 
        \ge \limsup_{n\to \infty} V^1_{U_m}(x,\tau^2_n) \\
        &\ge \liminf_{n\to \infty} J^1(x,\sigma^1_m,\tau^2_n)
        \ge J^1(x,\sigma^1_m,\tau^2).
    \end{align*} By \eqref{B1} and \eqref{B2} $J^1(x,\tau^1,\tau^2) =g_1(x)=J^1(x,\sigma^1_m,\tau^2)$ for $x\in \partial I$. This proves that $(\tau^1,\tau^2)$ is the desired equilibrium. The other claim follows from Lemma \ref{M1cont} \eqref{M1cont3}.
\end{proof}

\begin{proof} (of Theorem \ref{absbdryexists})\\
    Follows immediately from Theorem \ref{5.2}.
    
\end{proof}

\section{Existence of equilibria: Natural boundary}\label{sec:naturalboundary}

In this section we show that with the existence of equilibria in games driven by diffusions absorbed at the boundary, we can even construct equilibria in games based on diffusions with natural boundary. More specifically, in this section we generally assume that \eqref{A} and \eqref{C1} hold. The main theorem will also include the additional technical conditions \eqref{C2} and \eqref{C3}.

We start by constructing a candidate for the equilibrium. Let $I_n:=[a_n,b_n]$, $a_n<b_n\in I$, $n\in \N$, be an exhaustion of $I$ in the sense that $I_n\subset I_{n+1}$ for all $n\in \N$ and $\bigcup_{n\in \N} I_n=I$. Next we define games $I_n$ for each $n\in\N$ that satisfy the assumptions \eqref{A}, \eqref{B1} and \eqref{B2} such that Theorem \ref{absbdryexists} is applicable. For that set $X^n:= X^{\tau^{I_n^\circ}}$, where $X^{\tau^{I_n^\circ}}= (X^{\tau^{I_n^\circ}}_t)_{t\in [0,\infty)}:=(X_{t\wedge \tau^{I_n^\circ}})_{t\in [0,\infty)}$ denotes the process absorbed outside $I_n^\circ$. We choose a sequence of truncation functions $(\varphi_n)_{n\in \N}\in \ccc^\infty(I)^{\N}$ with the properties
\begin{align*}
    0\le \varphi_n\le 1, \quad \varphi(y)=0 \,\,\text{ for all }y\in I\setminus I_n \,\, \text{ and} \,\, \varphi(y)=1 \,\,\text{ for all }y\in I_{n-1}.
\end{align*} Now let $n\in \N$. We define
\begin{align*}
    f_{i,n}&:=g_i + \varphi_n (f_i-g_i), \quad i=1,2
\end{align*} and denote the set of randomized stopping times with respect to $X^n$ drawing external randomness from a random variable $E$ by $\ran_n(E)$. For $\tau^1\in \ran_n(E_1)$, $\tau^2\in \ran_n(E_2)$ and $x\in I_n$ set

\begin{align*}
    J^1_n(x,\tau^1,\tau^2)&:= \Ex_x[g_1(X^n_{\tau^1})e^{-r_1\tau^1}\ind{\tau^1\le \tau^2}+f_{1,n}(X^n_{\tau^2})e^{-r_1\tau^2}\ind{\tau^1>\tau^2}],\\
    J^2_n(x,\tau^1,\tau^2)&:= \Ex_x[g_2(X^n_{\tau^2})e^{-r_2\tau^2}\ind{\tau^2\le \tau^1}+f_{2,n}(X^n_{\tau^1})e^{-r_2\tau^1}\ind{\tau^2>\tau^1}] \quad \text{and}\\
    V^1_n(x,\tau^2)&:=\sup_{\tilde \tau^1\in \ran_n(E_1)}J^1_n(x,\tilde \tau^1,\tau^2),\\
    V^2_n(x,\tau^1)&:=\sup_{\tilde \tau^2\in \ran_n(E_2)}J^1_n(x,\tau^1,\tilde \tau^2).
\end{align*} 
The Dynkin game associated with the driving process $X^n$ on $I_n$ and the payoffs $J^1_n,J^2_n$ satisfies \eqref{A}, \eqref{B1} and \eqref{B2}. Thus by Theorem \ref{5.2} this game possesses an equilibrium $(\tau^1_n,\tau^2_n)=( \tau^{D^1_n, \lambda^1_n},\tau^{D^2_n,\lambda^2_n})\in \mathcal{Z}^{I_n^\circ}(E_1)\times \mathcal{Z}^{I_n^\circ}(E_2)$. To distinguish the equilibria for different $n$ from each other and from the desired equilibrium of the original game, we will refer to $(\tau^1_n,\tau^2_n)$ as $(J^1_n,J^2_n)$-equilibrium. By Theorem \ref{5.2} we may also assume 
\begin{align}
    &\supp(\lambda^1_n)\cap (D^2_n)^c \cap \{y\in I : f_{1,n}(y)>g_1(y)\}=\varnothing \quad\text{and}\label{52prop} \\
    &\supp(\lambda^2_n)\cap ( D^1_n)^c \cap \{y\in I : f_{2,n}(y)>g_2(y)\}=\varnothing.\notag
\end{align}
Technically $\tau^i_n$, $i=1,2$, is a Markovian randomized stopping time with respect to $X^n$ in the sense that the corresponding functional $A^{D^i_n,\lambda^i_n}$, see \eqref{func}, is generated by the local time of $X^n$ and $\tau^{D^i_n}=\inf\{t\ge0 : X^n_t \not \in D^i_n\}$, i.e.\ a first entry time of $X^n$. However, we can also regard this as a Markovian randomized stopping time with respect to $X$. Indeed, the functional $A^{D^i_n,\lambda^i_n}$ given by \eqref{func} but based on the local time of $X$ and $\tau^{D^i_n}=\inf\{t\ge0 : X_t \not \in D^i_n\}$, i.e. a first entry time of $X$, has the same distribution under $\PR_x$, $x\in I_n$ due to $D^i_n\subset I_n^\circ$. In view of the fact that we identify $\tau^i_n$ with the Markovian randomized stopping time that is generated by the functional based on the local time of $X$ and $\tau^{D^i_n}=\inf\{t\ge0 : X_t \not \in D^i_n\}$.

The next step is to extract a suitable limit point of the sequence $((\tau^1_n,\tau^2_n))_{n\in \N}$ as the equilibrium candidate for the original game. The problem is that the $(J^1_n,J^2_n)$-equilibria $(\tau^1_n,\tau^2_n)$, $n\in \N$, are not necessarily unique and could possibly have very different continuation regions $D^i_n$ and stopping rates $\lambda^i_n$, so $\lim_{n\to \infty} \tau^i_n$ may not be well defined in any reasonable sense. To find our limit point anyway, we construct a convergent subsequence of the $(J^1_n,J^2_n)$-equilibria next. 

For $n,m\in \N$ let
\begin{align}
    \tau^i_{n,m}:= \tau^{D^i_n\cap I^\circ_m, \lambda^i_n\vert_{I^\circ_m}}\in \mathcal{Z}^{I^\circ_m}, \quad i=1,2.\label{taunm}
\end{align} Given any $m\in \N$, by Theorem \ref{comp2} every subsequence $((\tau^1_{n^{m-1}_k,m},\tau^2_{n^{m-1}_k,m}))_{k\in \N}$ of $((\tau^1_{n,m},\tau^2_{n,m}))_{n\in \N}$ possess a subsequence $((\tau^1_{n^{m}_k,m},\tau^2_{n^{m}_k,m}))_{k\in \N}$ that converges in the $\mathcal{M}^1$-product topology, i.e.\ there is $(\tau^1_{*,m},\tau^2_{*,m})=(\tau^{D^1_{*,m}, \lambda^1_{*,m}}, \tau^{D^2_{*,m}, \lambda^2_{*,m}})\in \mathcal{Z}^{I^\circ_m}(E_1)\times\mathcal{Z}^{I^\circ_m}(E_2)$ such that 
\begin{align*}
    (\tau^1_{n^{m}_k,m},\tau^2_{n^{m}_k,m}) \stackrel{k\to \infty}{\longrightarrow} (\tau^1_{*,m},\tau^2_{*,m})
\end{align*} in the $\mathcal{M}^1$-product topology. By Lemma \ref{lemma:ratesconverge}, by possibly passing over to a further subsequence, we may assume that for each $i\in\{1,2\}$
\begin{align*}
    \iota_1(\tau^i_{n^{m}_k,m}) \stackrel{k\to \infty}{\longrightarrow} \iota_1(\tau^i_{*,m})=\lambda^i_{*,m}
\end{align*} 
vaguely on every connected component of $D^i_{*,m}$. Now the diagonal sequence $((\tau^1_{n^{k}_k},\tau^2_{n^{k}_k}))_{k\in \N}$ has the properties 
\begin{align*}
    (\tau^1_{n^{k}_k,m},\tau^2_{n^{k}_k,m}) \stackrel{k\to \infty}{\longrightarrow} (\tau^1_{*,m},\tau^2_{*,m})
\end{align*} in the $\mathcal{M}^1$-product topology and \begin{align*}
    \iota_1(\tau^i_{n^{m}_k,m}) \stackrel{k\to \infty}{\longrightarrow} \iota_1(\tau^i_{*,m})=\lambda^i_{*,m}
\end{align*} vaguely on every connected component of $D^i_{*,m}$ for all $m\in \N$. It remains to be shown that $\lim_{k\to \infty} (\tau^1_{n^{k}_k},\tau^2_{n^{k}_k})$ exists in a reasonable sense. For that we start by establishing a consistency condition.

\begin{lemma}\label{lemma:6.1}
    The sequence $((\tau^1_{*,m},\tau^2_{*,m}))_{m\in \N}$ constructed above under the assumptions \eqref{A} and \eqref{C1} satisfies the following consistency conditions:
    \begin{align*}
         D^i_{*,m} &= D^i_{*,m'}\cap I^\circ_m \quad \text{and}\\
         \lambda^i_{*,m} &= \lambda^i_{*,m'}\vert_{I^\circ_m}
    \end{align*} for both $i \in \{1,2\}$ and all $m'\ge m\in \N$
\end{lemma}

\begin{proof}
    For every $n\in \N$, we denote the connected component of $D^i_{*,n}$ that contains $x\in I$ by $D^i_{*,n,x}$ with $D^i_{*,n,x}:=\varnothing$ for $x\not\in D^i_{*,n}$. Then $D^i_{*,n}=\bigcup_{x \in I^\circ_n\cap \mathbb{Q}} D^i_{*,n,x}$. Now fix some $x\in I^\circ_m\cap \mathbb{Q}$. By Corollary \ref{corollary:LR}, $D^i_{*,n,x}=(L^i_{*,n,x},R^i_{*,n,x})$, where the indices $*,n$ express that $L^i_{*,n,x}$, $R^i_{*,n,x}$ denote $L^i_x$ and $R^i_x$ from \eqref{LRgl} corresponding to $D^i_{*,n,x}$. By \eqref{6fin}, and due to the definition of $\tau^i_{n_1,n_2}$, $n_1,n_2\in \N$
    \begin{align*}
        &L^i_{*,m,x}\\
        =& \sup\left\{ y \le x : \liminf_{k \to \infty} \infty \ind{D^i_{n_k^m}\cap I^\circ_m \cap (y-\eps,y+\eps) \neq \varnothing} + \lambda^i_{n_k^m}\vert_{I^\circ_m} ((y-\eps,y+\eps)) = \infty \,\, \forall \eps >0\right\}\\
        =&\sup\left\{ y \le x : \liminf_{k \to \infty} \infty \ind{D^i_{n_k^{m'}}\cap I^\circ_m \cap (y-\eps,y+\eps) \neq \varnothing} + \lambda^i_{n_k^{m'}}\vert_{I^\circ_m} ((y-\eps,y+\eps)) = \infty \,\, \forall \eps >0\right\}\\
        =& \inf I^\circ_m \vee \sup\left\{ y \le x : \liminf_{k \to \infty} \infty \ind{D^i_{n_k^{m'}}\cap I^\circ_{m'} \cap (y-\eps,y+\eps) \neq \varnothing} + \lambda^i_{n_k^{m'}}\vert_{I^\circ_{m'}} ((y-\eps,y+\eps)) = \infty \,\, \forall \eps >0\right\}\\
        =& \inf I^\circ_m \vee L^i_{*,m',x}
    \end{align*} Note that in the second step we used that by Lemma \ref{lemma:eineUmg}, $L^i_{*,m,x}$ is uniquely determined by the limit $(\tau^1_{*,m},\tau^2_{*,m})$ of $((\tau^1_{n^m_k,m},\tau^2_{n^m_k,m}))_{k\in \N}$, so the right hand side of \eqref{6fin} must remain the same when we pass over to the subsequence $((\tau^1_{n^{m'}_k,m},\tau^2_{n^{m'}_k,m}))_{k\in \N}$, which must have the same limit. Analogously we infer $R^i_{*,m,x}=\sup I^\circ_m \wedge R^i_{*,m',x}$, so $D^i_{*,m,x}= D^i_{*,m',x} \cap I^\circ_m$. For $x_1\neq x_2\in I^\circ_m$, as $D^i_{*,m',x_1}$, $D^i_{*,m',x_2}$ are connected components and $I^\circ_m$ is open and connected, the sets $D^i_{*,m',x_1} \cap I^\circ_m$ and $D^i_{*,m',x_2}\cap I^\circ_m$ must be either equal or disjoint. Putting pieces together, we get
    \begin{align*}
        D^i_{*,m}
        = \bigcup_{x \in I^\circ_m\cap \mathbb{Q}} D^i_{*,m,x} 
        = \bigcup_{x \in I^\circ_m\cap \mathbb{Q}} D^i_{*,m',x} \cap I^\circ_m
        = \bigcup_{x \in I^\circ_{m'}\cap \mathbb{Q}} D^i_{*,m',x} \cap I^\circ_m
        = I^\circ_m\cap \bigcup_{x \in I^\circ_{m'}\cap \mathbb{Q}} D^i_{*,m',x}
        = I^\circ_m\cap D^i_{*,m'}.
    \end{align*} Invoking the definition of $\tau^i_{n_1,n_2}$, $n_1,n_2\in \N$, the second claim follows from
    \begin{align*}
        \lambda^i_{*,m}= \lim_{k\to \infty} \lambda^i_{n^m_k}\vert_{I^\circ_m} = \lim_{k\to \infty} \lambda^i_{n^{m'}_k} \vert_{I^\circ_m} = \lambda^i_{*,m'} \vert_{I^\circ_m}.
    \end{align*}
\end{proof}

With Lemma \ref{lemma:6.1} at hand, for each $i\in \{1,2\}$ we set
\begin{align}
    D^i_*:=\bigcup_{m\in \N} D^i_{*,m}\label{D*}
\end{align} and define $\lambda^i_*$ to be the unique measure on $D^i_*$ such that 
\begin{align}
    \lambda^i_*\vert_{I^\circ_m}=  \lambda^i_{*,m}\label{l*}
\end{align} for all $m\in \N$. The rest of this section is dedicated to show that $(\tau^1_*,\tau^2_*):=(\tau^{D^1_*,\lambda^1_*},\tau^{D^2_*,\lambda^2_*})$ is a Markov-perfect Nash equilibrium in the original game with natural boundaries. We start by gathering some technical results.

\begin{lemma}\label{6.2}
    Assume \eqref{A}, \eqref{C1}, \eqref{C2}, \eqref{C3}. For both $i\in \{1,2\}$, all $x\in I$, and the sets $I_m$, $m\in\N$, introduced in the beginning of Section \ref{sec:naturalboundary} it holds that
    \begin{align*}
        \sup_{(\tau^1,\tau^2)\in \ran(E_1) \times \ran(E_2), k \ge m} \vert J^i(x,\tau^1,\tau^2) -J^i_k(x,\tau^1\wedge \tau^{I^\circ_m}, \tau^2\wedge \tau^{I^\circ_m})\vert \stackrel{m\to \infty}{\longrightarrow}0.
    \end{align*} 
\end{lemma}

\begin{proof}
    Fix $i\in \{1,2\}$ and $x\in I$. By \eqref{C1}, $I$ is open, so $\tau^{I^\circ_m}\to \infty$ for $m\to \infty$ $\PR_x$-a.s.\ due to the continuity of the paths of $X$. By \eqref{A} and \eqref{C3}

    \begin{align*}
        \sup_{s\in [t,\infty)}e^{-r_is}\vert g_i(X_s) \vert \stackrel{t\to \infty}{\longrightarrow}0 \quad\text{and}\quad \sup_{s\in [t,\infty)}e^{-r_is}\vert f_i(X_s)\vert \stackrel{t\to \infty}{\longrightarrow}0\quad \PR_x\text{-a.s.}
    \end{align*} Using that $f_{i_m}=f_i$ on $I_{m-1}$ first, with dominated convergence due to \eqref{C2} the above yields
    \begin{align*}
        &\sup_{(\tau^1,\tau^2)\in \ran(E_1) \times \ran(E_2), k \ge m} \vert J^i(x,\tau^1,\tau^2) -J^i_k(x,\tau^1\wedge \tau^{I^\circ_m}, \tau^2\wedge \tau^{I^\circ_m})\vert\\
        =&\sup_{(\tau^1,\tau^2)\in (\mathcal{T}\cup\mathcal{Z}(E_1)) \times (\mathcal{T}\cup\mathcal{Z}(E_2)), k \ge m} \vert J^i(x,\tau^1,\tau^2) -J^i_k(x,\tau^1\wedge \tau^{I^\circ_m}, \tau^2\wedge \tau^{I^\circ_m})\vert \\
        =& \sup_{(\tau^1,\tau^2)\in (\mathcal{T}\cup\mathcal{Z}(E_1)) \times (\mathcal{T}\cup\mathcal{Z}(E_2))} \vert \Ex_x[\ind{\tau^1\wedge \tau^2 >\tau^{I^\circ_{m-1}}} (e^{-r_i\tau^1}g_i(X_{\tau^1})\ind{\tau^1\le \tau^2} + e^{-r_i\tau^2}f_i(X_{\tau^2})\ind{\tau^1> \tau^2}\\
        &-  e^{-r_i\tau^{I^\circ_{m-1}}} g_i(X_{\tau^{I^\circ_{m-1}}}))]\vert\\
        \le& \Ex_x\left[ \sup_{s\in [\tau^{I^\circ_{m-1}},\infty)}e^{-r_is}\vert f_i(X_s)\vert  \right] \stackrel{m\to \infty}{\longrightarrow}0.
    \end{align*}
\end{proof}

\begin{corollary}\label{onlyz} Assume \eqref{A}, \eqref{C1}, \eqref{C2}, \eqref{C3} and fix some $\tau^2\in \mathcal{Z}(E_2)$. Now
\begin{align*}
    \sup_{\tau^1\in \ran(E_1)}J^1(x,\tau^1,\tau^2) &=\sup_{m\in \N}\sup_{\tau^1\in \ran(E_1)}J^i_m(x,\tau^1\wedge \tau^{I^\circ_m}, \tau^2\wedge \tau^{I^\circ_m})=\sup_{m\in \N}\sup_{\tau^1\in \mathcal{Z}(E_1)}J^i_m(x,\tau^1\wedge \tau^{I^\circ_m}, \tau^2\wedge \tau^{I^\circ_m})\\
    &= \sup_{\tau^1\in \mathcal{Z}(E_1)}J^1(x,\tau^1,\tau^2)
\end{align*}
    
\end{corollary}

\begin{proof}
    The first and third equality follow from Lemma \ref{6.2} and the second from Lemma \ref{indiff} \eqref{indiff2} since the Dynkin game associated with the driving process $X^n$ on $I_n$ and the payoffs $J^1_n,J^2_n$ satisfies \eqref{A}, \eqref{B1} and \eqref{B2}.
\end{proof}

\begin{lemma}\label{6.3}
Let $m\in \N$. For all $n\ge m$, all $z\in I_m$ and $(\tau^1_{n,m},\tau^2_{n,m})$ defined in \eqref{taunm}, under the assumptions \eqref{A} and \eqref{C1} we have
\begin{align}
    &J^1_m(z, \tau^1_{n,m}, \tau^2_{n,m}) \ge J^1_m(z, \tau^{I_m^\circ} , \tau^2_{n,m}), \quad J^2_n(y,  \tau^1_{n,m},\tau^2_{n,m}) \ge J^2_n(y,  \tau^1_{n,m},\tau^{I_m^\circ}), \label{aspt:11}\\
    &\supp(\lambda^1_{n,m})\cap \tilde D^2_{n,m}\subset \{y\in U: g_1(y)\ge J^1_m(y, \tilde \tau^1_{n,m},\tilde \tau^2_{n,m})\},\label{aspt:21}\\
    &\supp( \lambda^2_{n,m})\cap D^1_{n,m}\subset \{y\in U: g_2(y)\ge J^2_m(y, \tau^1_{n,m},\tau^2_n)\}, \notag\\
    &\supp( \lambda^1_{n,m}) \cap ( D^2_{n,m})^c\cap\{y\in I:f_{1,m}(y)>g_1(y)\} =\varnothing,\label{aspt:31} \\ 
    &\supp(\lambda^2_{n,m}) \cap (D^1_{n,m})^c\cap\{y\in I:f_{2,m}(y)>g_2(y)\}=\varnothing.\notag
\end{align}
\end{lemma}

\begin{proof}
    Let $n\ge m \in\N$, $z\in I_m$ and $i=1$ w.l.o.g. By definition $V^1_n(z,\tau^2_n)\ge V^1_m(\tau^2_{n,m})$. \eqref{52prop} and $\PR_z(\tau^1_{n,n}<\infty)=1$ allow to apply Lemma \ref{indiff} \eqref{weakindiff} which yields $\PR_z(X_{\tau^1_{n,n}}\in \{y\in I_n: V^1_n(y,\tau^2_n) = g_1(y)\})=1$. Putting things together we find $\{y\in I_n: V^1_n(y,\tau^2_n) = g_1(y)\}\cap I_m \subset \{y\in I_m: V^1_m(y,\tau^2_{n,m}) = g_1(y)\}$ and so $\PR_z(X_{\tau^1_{n,m}}\in \{y\in I_m: V^1_m(y,\tau^2_{n,m}) = g_1(y)\})=1$. We infer \eqref{aspt:11} via
    \begin{align*}
        &J^1_m(z,\tau^1_{n,m},\tau^2_{n,m})- J^1_m(z,\tau^{I_m^\circ},\tau^2_{n,m})\\
        =&\Ex_z[\ind{\tau^1_{n,m}\le \tau^2_{n,m}} \ind{ \tau^1_{n,m} < \tau^{I_m^\circ}}( g_1(X_{\tau^1_{n,m}})e^{-r_1 \tau^1_{n,m}} -g_1(X_{\tau^{I_m^\circ}})e^{-r_1 \tau^{I_m^\circ}}\ind{ \tau^{I_m^\circ}\le \tau^2_{n,m}}-f_{1,m}(X_{\tau^2_{n,m}})e^{-r_1 \tau^2_{n,m}}\ind{ \tau^2_{n,m} < \tau^{I_m^\circ}}) ]\\
        =&\Ex_z[\ind{\tau^1_{n,m}\le \tau^2_{n,m}} \ind{ \tau^1_{n,m} < \tau^{I_m^\circ}}e^{-r_1 \tau^1_{n,m}} (V^1_m(X_{\tau^1_{n,m}},\tau^2_{n,m})- J^1_m(X_{\tau^1_{n,m}},\tau^{I_m^\circ},\tau^2_{n,m}))] \ge 0
    \end{align*} where we use the Markov property in the second step. By Lemma \ref{indiff} \eqref{weakindiff} $\supp(\lambda^1_n)\cap D^2_n\subset \{y\in I_n: V^1_n(y,\tau^2_n) = g_1(y)\}$. We obtain \eqref{aspt:21} through
    \begin{align*}
        &\supp(\lambda^1_{n,m})\cap D^2_{n,m} 
        \subset \supp(\lambda^1_n)\cap D^2_n 
        \subset \{y\in I_n: V^1_n(y,\tau^2_n) = g_1(y)\}\\
        \subset&\{y\in I_m: V^1_m(y,\tau^2_{n,m}) = g_1(y)\} 
        \subset\{y\in I_m: J^1_m(y,\tau^1_{n,m},\tau^2_{n,m})\le g_1(y)\}.
    \end{align*} The definitions of $f_{1,m}$, $\lambda^1_{n,m}$ and $(D^2_{n,m})^c$ combined with \eqref{52prop} provide
    \begin{align*}
        \supp(\lambda^1_{n,m})\cap (D^2_{n,m})^c \cap \{y\in I : f_{1,m}(y)>g_1(y)\}\subset \supp(\lambda^1_n)\cap (D^2_n)^c \cap \{y\in I : f_{1,n}(y)>g_1(y)\}=\varnothing,
    \end{align*} i.e.\ \eqref{aspt:31}.
\end{proof}

\begin{theorem} \label{6.4}Assume \eqref{A}, \eqref{C1}, \eqref{C2} and \eqref{C3}. $(\tau^1_*,\tau^2_*):=(\tau^{D^1_*,\lambda^1_*},\tau^{D^2_*,\lambda^2_*})\in \mathcal{Z}(E_1)\times \mathcal{Z}(E_2)$ with $D^i_*,\lambda^i_*$, $i=1,2$ given by \eqref{D*} and \eqref{l*} is a Markov perfect Nash equilibrium. 
\end{theorem}

\begin{proof}
    W.l.o.g.\ we only show \eqref{cond1}. Let $\tau^1\in\mathcal{Z}(E_1)$ and $\eps>0$. By Lemma \ref{6.2}, there is $N\in\N$ such that $\sup_{(\tau^1,\tau^2)\in \ran(E_1) \times \ran(E_2)} \vert J^i(x,\tau^1,\tau^2) -J^i_m(x,\tau^1\wedge \tau^{I^\circ_m}, \tau^2\wedge \tau^{I^\circ_m})\vert <\eps$ for all $m\ge N$. By Lemma \ref{6.3} we can apply Lemma \ref{M1cont} \eqref{M1cont1} and \eqref{M1cont2} to find $n\ge N$ with $J^1_N(x,\tau^1\wedge \tau^{I^\circ_N},\tau^2_{*,N}) - J^1_N(x,\tau^1\wedge \tau^{I^\circ_N},\tau^2_{n,N}) <\eps$ and $\vert J^1_N(x,\tau^1_{*,N},\tau^2_{*,N})- J^1_N(x,\tau^1_{n,N},\tau^2_{n,N}) \vert <\eps$. Due to the definitions of the involved stopping times, all of this yields
    \begin{align*}
        &J^1(x,\tau^1_*,\tau^2_*)-J^1_n(x,\tau^1_{n,n},\tau^2_{n,n}) \\
        =& J^1(x,\tau^1_*,\tau^2_*) - J^1_N(x,\tau^1_{*,N},\tau^2_{*,N}) + J^1_N(x,\tau^1_{*,N},\tau^2_{*,N})- J^1_N(x,\tau^1_{n,N},\tau^2_{n,N}) \\
        &+ J^1_N(x,\tau^1_{n,N},\tau^2_{n,N})- J^1(x,\tau^1_{n},\tau^2_{n})+ J^1(x,\tau^1_{n},\tau^2_{n})-J^1_n(x,\tau^1_{n,n},\tau^2_{n,n}) \\
        >&- 4\eps
    \end{align*} and
    \begin{align*}
        &J^1(x,\tau^1,\tau^2_*)-J^1_n(x,\tau^1\wedge \tau^{I^\circ_n},\tau^2_{n,n}) \\
        =& J^1(x,\tau^1,\tau^2_*) - J^1_N(x,\tau^1\wedge \tau^{I^\circ_N},\tau^2_{*,N}) + J^1_N(x,\tau^1\wedge \tau^{I^\circ_N},\tau^2_{*,N})- J^1_N(x,\tau^1\wedge \tau^{I^\circ_N},\tau^2_{n,N}) \\
        &+ J^1_N(x,\tau^1\wedge \tau^{I^\circ_N},\tau^2_{n,N})- J^1(x,\tau^1,\tau^2_{n})+ J^1(x,\tau^1,\tau^2_{n}) -J^1_n(x,\tau^1\wedge \tau^{I^\circ_n},\tau^2_{n,n}) \\
        <& 4\eps.
    \end{align*}
    Since $(\tau^1_{n,n},\tau^2_{n,n})$ is a $(J^1_n,J^1_n)$-equilibrium this implies
    \begin{align*}
        J^1(x,\tau^1,\tau^2_*)+ 4\eps >J^1_n(x,\tau^1\wedge \tau^{I^\circ_n},\tau^2_{n,n})
        \ge J^1(x,\tau^1,\tau^2_{n})+ J^1(x,\tau^1,\tau^2_{n}) >J^1(x,\tau^1,\tau^2_*) -4\eps.
    \end{align*} Sending $\eps\searrow 0$ we find that $J^1(x,\tau^1_*,\tau^2_*)\ge J^1(x,\tau^1,\tau^2_*) $. By Corollary \ref{onlyz} this implies \eqref{cond1}.
\end{proof}

\begin{proof}(of Theorem \ref{natbdryexists})\\
The claim follows directly from Theorem \ref{6.4}.
\end{proof}

\section{Discussion}\label{sec:discussion}
In this paper we have presented a technique for proving the existence of Markov-perfect equilibria in war-of-attrition type Dynkin games. 

For this purpose, we have equipped the class of Markovian randomized stopping times with two topologies, the $\mathcal M$ and the $\mathcal M^1$-product topology. Using the former and Kakutani's fixed point theorem, we first found equilibria in discretized games. By forming limits with respect to the $\mathcal M^1$-product topology, we then obtained equilibria for the original Dynkin game. 

We have chosen the detour via the discretization of the game in order to avoid a technical
problem in the direct application of the Kakutani fixed point theorem to the non-restricted
best-response mapping with the $\mathcal M^1$-product topology. In this case, convexity turns out to
be a difficulty: convex combinations of the distributions of Markovian randomized stopping
times are of course generally not of this form. Thus, other variants of the Kakutani fixed point
theorem would be necessary. However, we did not succeed in finding a suitable one for this
situation.

We have chosen the war-of-attrition type Dynkin game class to implement the approach because it is one of the essential classes of stopping games with natural randomization and has been studied in many contexts. However, by studying our proof carefully, it becomes clear that the approach presented here is by no means restricted to our framework. Many elements can certainly also be used for proofs in other stopping games with underlying linear diffusions. We will discuss this briefly below: 

The compactness discussion in Section \ref{sectopo} hardly uses any elements of the specific problem and can form the basis for existence proofs in other problem classes. 
The existence of equilibria in the discretized game in Section \ref{secdisc} then uses the problem structure more directly. Of course, it must first be ensured that the best response problem on a Markovian stopping time is of Markovian structure. In addition, the $\mathcal{M}$-continuity of $J$ (Lemma \ref{MC}) and the closedness of the graph of the best response mapping (Lemma \ref{lemma_closed}) must be given, which here is particularly related to the question of whether players want to stop simultaneously.  While the main line of argument in Section \ref{sec:final_proof} can be adapted for other problems, the details of the estimates again depend on the continuity structure of the problem. The mentioned continuity properties are not only of a technical nature: As mentioned above, there are Dynkin games without equilibria. However, starting from equilibria in some classes of games, it is often possible to directly construct $\epsilon$ equilibria in other classes. We are currently working on this in another project for zero-sum Dynkin games.

All in all, one can be optimistic that the approach developed in this paper can be used as a blueprint to show the existence of Markov-perfect equilibria in other classes of Dynkin games and beyond. 

\appendix
\section{Appendix}


\begin{lemma}\label{genport}(Generalized Portmanteau theorem)\\
Let $\mathcal{X}$ be a locally compact, second countable metric space, $(\mu_n)_{n \in \N}\in \mathcal{M}^1(\mathcal{X})^{\N}$, $\mu \in \mathcal{M}^1(\mathcal{X})$ such that $\mu_n \stackrel{n \to \infty}{\longrightarrow} \mu$ weakly.
\begin{enumerate}[(i)]
    \item \label{genport1}If $f: \mathcal{X} \to \R$ is bounded and $\mu(\{x \in \mathcal{X}: f$ is not lower semicontinuous at $x\})=0$ then
    \begin{align*}
        \liminf_{n \to \infty} \int f d\mu_n \ge \int f d\mu.
    \end{align*}
    \item \label{genport2}If $f: \mathcal{X}\to \R$ is bounded and $\mu(\{x \in \mathcal{X}: f$ is not continuous at $x\})=0$ then
    \begin{align*}
        \lim_{n \to \infty} \int f d \mu_n = \int fd\mu.
    \end{align*}
\end{enumerate}
\end{lemma}

\begin{proof} \eqref{genport1} We define 
\begin{align*}
    \overline{f}(x):= \min \left\{f, \liminf_{y \stackrel{y\neq x}{\longrightarrow} x} f(y) \right\}, \quad x\in \mathcal{X}.
\end{align*} By construction $\overline{f}$ is lower semicontinuous and $f=\overline{f}$ $\mu$-a.s. 
Set 
\begin{align*}
f_n(x):= \frac{1}{n} \sum_{k\in \N}\inda{(\frac{k}{n},\infty)}\left(\overline{f}(x)-\inf_{y \in \mathcal{X}}\overline{f}(y)\right), \quad x\in \mathcal{X}.
\end{align*} By construction $f_n \stackrel{n \to \infty}{\nearrow} \overline{f}-\inf_{y\in \mathcal{X}} \overline{f}(y)$ pointwise. Note that since $\overline{f}(-\inf_{y\in \mathcal{X}} \overline{f}(y))$ is lower semicontinuous the sets $\{\overline{f}-\inf_{y\in \mathcal{X}} \overline{f}(y)\in (\frac{k}{n},\infty)\}$ are open for all $k,n\in \N$. By the usual Portmanteau theorem
\begin{align*}
    &\liminf_{n\to \infty} \int f -\inf_{y\in \mathcal{X}} \overline{f}(y) d \mu_n 
    \ge \liminf_{n\to \infty} \int \overline{f}-\inf_{y\in \mathcal{X}} \overline{f}(y) d \mu_n \\
    \ge& \liminf_{m\to \infty}\liminf_{n\to \infty}  \int f_m d\mu_n
    = \liminf_{m\to \infty}\frac{1}{m} \sum_{k\in \N}\liminf_{n\to \infty}  \int \inda{(\frac{k}{m},\infty)}\left(\overline{f}-\inf_{y \in \mathcal{X}}\overline{f}(y)\right)d\mu_n\\
    \ge& \liminf_{m\to \infty}\frac{1}{m} \sum_{k\in \N} \int \inda{(\frac{k}{m},\infty)}\left(\overline{f}-\inf_{y \in \mathcal{X}}\overline{f}(y)\right)d\mu
    =\int \overline{f} -\inf_{y \in \mathcal{X}}\overline{f}(y) d\mu\\
    =&  \int f -\inf_{y\in \mathcal{X}} \overline{f}(y) d \mu. 
\end{align*}Adding $\int  \inf_{y\in \mathcal{X}} \overline{f}(y) d \mu = \lim{n \to \infty}\inf_{y\in \mathcal{X}} \overline{f}(y) d \mu_n$ on both sides yields the claim.

\eqref{genport2} By applying \eqref{genport1} to $f$ and $-f$ we obtain
\begin{align*}
    \liminf_{n \to \infty} \int f d\mu_n \ge \int f d\mu\\
\end{align*} and
\begin{align*}
    -\limsup_{n \to \infty} \int f d\mu_n  = \liminf_{n \to \infty} \int -f d\mu_n \ge \int -f d\mu =-\int f d\mu,
\end{align*} so 
\begin{align*}
    \limsup_{n \to \infty} \int f d\mu_n \le \int fd \mu.
\end{align*}
    
\end{proof}

\bibliographystyle{abbrv}
\bibliography{literature}

\end{document}